\RequirePackage{fix-cm}
\documentclass[smallextended]{svjour3}       
\smartqed  
\usepackage{csquotes}
\usepackage{graphicx}
\newcommand{\B}[1]{\boldsymbol{#1}}

\usepackage{amsfonts}
\usepackage{amsmath}
\usepackage{geometry}
\usepackage{csquotes}
\usepackage{algorithmicx}
\usepackage{algorithm}
\usepackage{algpseudocode}
\usepackage{color}

\usepackage{empheq} 
\usepackage{comment}

\begin{document}

\title{Spectral Schur analysis of structured moment matrices arising from enriched quadratic histopolation on simplicial meshes
}

\titlerunning{Spectral Schur analysis of structured moment matrices for quadratic histopolation}        

\author{Allal Guessab \and 
        Federico Nudo  
}


\institute{Allal Guessab \at
             Avenue Al Walae, 73000, Dakhla, Morocco\\ 
             \email{guessaballal7@gmail.com} 
         \and 
            Federico Nudo (corresponding author) \at
              Department of Mathematics and Computer Science, University of Calabria, Rende (CS), Italy\\
              \email{federico.nudo@unical.it}
}

\date{Received: date / Accepted: date}

\maketitle

\begin{abstract}
In this paper we study parameter-dependent structured moment matrices with a canonical block form arising from weighted quadratic histopolation on simplicial meshes. 
For a strictly positive density on a simplex, we construct  compatible face densities and an orthogonal decomposition of the quadratic polynomial space into face and interior components, which induces a natural face-interior block structure. 
A reduced Schur complement is identified that fully characterizes enrichment and well-posedness and provides a sharp spectral stability result. 
We show that this quantity coincides with the square root of the smallest eigenvalue of a low-dimensional symmetric positive definite operator. 
This matrix-based viewpoint yields simple spectral criteria for the invertibility of local moment systems and motivates spectrally preferable choices of face and interior bases with improved conditioning. 
Using the resulting degrees of freedom together with density and scaling parameters as design variables, we formulate a small eigenvalue optimization problem aimed at improving stability and reducing the condition number of the global reconstruction system. 
Three-dimensional experiments on uniform and quasi-uniform simplicial meshes illustrate the predicted stability, conditioning, and convergence behaviour of the enriched quadratic reconstruction.

\end{abstract}

\section{Introduction}
Block-structured symmetric positive definite systems and their reduced operators are a recurring
theme in applied numerical linear algebra~\cite{van2005orthogonal,benzi2005numerical,mastronardi2006computing}. In many discretization methods, the algebraic structure
is generated by a small number of elementary local building blocks, such as mass, stiffness, or
moment matrices, assembled over a mesh. Understanding how these blocks interact, how they can
be factorized, and how their spectra control stability and conditioning has long been a central
topic in numerical linear algebra, with direct impact on the robustness and efficiency of large-scale
solvers and preconditioners~\cite{capizzano2001spectral,calvetti2002restarted,jagels2013structure,capizzano2016spectral,capizzano2021matrix,noschese2022estimating,el2023spectral,bruni2025numerical,stanic2025analysis}. 
Related projection-based approaches and Krylov subspace techniques for large-scale structured problems
have been extensively investigated in the literature; see, e.g.,~\cite{bentbib2018computational}.
A particularly relevant instance arises in discretization methods based on integral data, where the
unknown function is reconstructed not from pointwise values, but from averages or weighted
integrals over subsets of the domain. Such formulations naturally lead to families of low-dimensional
structured matrices at the element level and appear in several applications, including inverse
problems and imaging, medical and seismic tomography, and transport and flow models
\cite{kak2001principles,natterer2001mathematics,palamodov2016reconstruction}.
This paper develops a spectral Schur analysis for a family of \emph{structured moment matrices}.
These matrices exhibit a canonical face-interior organization. We show that enrichment, unisolvence, and stability
can be characterized in a dimension-independent way by the invertibility and spectral properties
of a reduced symmetric positive definite operator of moderate size. In particular, a functional
stability quantity admits an exact matrix interpretation: it coincides with the square root of the
smallest eigenvalue of a suitably normalized reduced operator. This spectral viewpoint provides
a concrete and computable robustness measure and turns the choice of local spaces, weights,
and scalings into a low-dimensional eigenvalue design problem. 
While histopolation motivates the construction, the analysis is matrix-based and applies more broadly to block-structured positive semidefinite moment systems induced by integral (or moment) constraints on simplicial meshes. In this sense, the results connect directly to Schur-complement techniques and spectral design/preconditioning principles in numerical linear algebra. The motivating application for this analysis is \emph{local weighted enriched quadratic histopolation on simplicial meshes}. Given a finite collection of linear functionals defined by integrals, the goal of histopolation
is to construct an approximant whose associated averages reproduce the prescribed data. Over the
last decades, this principle has led to a variety of developments, ranging from shape-preserving
spline schemes to global polynomial constructions
\cite{Fischer:2005:MPR,Fischer:2007:CSP,Siewer:2008:HIS,bruni2024polynomial,bruni2025polynomial,bruno2025bivariate}.
In the local formulation, the domain is partitioned into elements and, on each simplex, a local
histopolant is computed from the prescribed integral data. A global approximation is then obtained
by assembling these local contributions \cite{Guessab:2022:SAB}. This construction leads to small
structured linear systems at the element level, whose coefficients encode the interaction between
the density, the polynomial reconstruction space, and the imposed moment constraints. From a
numerical linear algebra perspective, the robustness of these systems is determined by the spectral
properties of reduced operators obtained by Schur-type reductions. A related three-dimensional
quadratic weighted histopolation scheme on tetrahedral meshes, based on different reconstruction
functionals, is studied in~\cite{guessab2025quadratic}.  In most local histopolation schemes, the reconstruction space is chosen as a finite-dimensional
polynomial space, which allows the associated moment matrices to be expressed in closed form.
While purely linear reconstruction often yields limited accuracy due to the low order of the trial
functions, quadratic enrichment increases approximation power while preserving an algebraic
structure that remains suitable for detailed spectral analysis. This balance between approximation
capability, conditioning control, and compatibility with general simplicial meshes makes quadratic
histopolation particularly well suited to modern reconstruction tasks involving integral or
moment-based data. Despite their apparent simplicity, the corresponding moment matrices exhibit
a highly structured internal organization, and understanding this structure is essential for deriving
robust and dimension-independent stability results.

Our contributions can be summarized as follows.
\begin{itemize}
\item \textbf{Canonical face/interior decomposition and intrinsic block structure.}
  For any strictly positive density on a simplex, we construct compatible face densities and derive
  an orthogonal decomposition of the quadratic enrichment space into face and interior parts.
  This makes the underlying block organization of the associated moment matrices explicit and
  isolates the reduced operator that controls enrichment.

\item \textbf{Dimension-independent unisolvence criterion via Schur reduction.}
We establish two equivalent linear algebraic characterizations of the unisolvence of the
enriched quadratic reconstruction from prescribed moments.
In a reduced formulation, unisolvence holds if and only if an affine-level matrix and a
Schur-type operator associated with the face/interior coupling are nonsingular.
Equivalently, unisolvence holds if and only if the affine and quadratic moment matrices are
nonsingular.
This yields a sharp characterization of unisolvence that depends only on local moment
information and is independent of the spatial dimension.

\item \textbf{Sharp spectral characterization of stability.}
  We derive an operator formulation of the inf--sup constant of the enriched reconstruction and
  show that, after Schur reduction and normalization, it is exactly the square root of the smallest
  eigenvalue of a reduced symmetric positive definite operator. This provides a clean spectral
  stability metric to assess the effect of the density and of local choices.

\item \textbf{Spectrally preferable bases and conditioning control.}
  Building on the reduced spectral description, we construct explicit families of face and interior
  quadratic basis functions that minimize inter-block coupling and improve conditioning. The
  resulting design principles are purely finite-dimensional (Gram structures, Rayleigh quotients,
  and standard inequalities) and fit naturally into the numerical linear algebra perspective.

\item \textbf{Affine-invariant robustness and spectral optimization.}
  By transporting densities and test functions via barycentric coordinates under affine element maps,
  we establish invariance of the relevant reduced operators across affine-equivalent simplices,
  explaining the observed stability on uniform and quasi-uniform meshes. We finally formulate
  a small eigenvalue optimization problem over density and scaling parameters, interpretable as
  an automatic spectral tuning (or preconditioning) of the moment system.
\end{itemize}

The paper is organized as follows.
In Section~\ref{sec2} we introduce the enriched weighted quadratic histopolation problem and derive
the structured block form of the associated moment matrices. Within this framework, we establish two equivalent unisolvence criteria, identify a Schur-type reduced operator that characterizes unisolvence
and stability, and prove that the resulting strategy is robust under affine transformations of the
elements, yielding uniform bounds on shape-regular meshes. In Section~\ref{sec3} we develop a
canonical orthogonal decomposition of the quadratic polynomial space into face and interior
components, leading to optimal choices of quadratic basis functions and clarifying their spectral
role within the moment matrix structure. In Section~\ref{sec4} we provide explicit realizations of the
framework for Dirichlet and affine weights. In Section~\ref{sec5} we introduce parametric scalings
of the moment functionals and of the underlying densities and formulate spectral optimization
problems aimed at improving stability or reducing the condition number of the reconstruction
system. Finally, in Section~\ref{sec6} we present three-dimensional numerical experiments that
validate the theoretical predictions. These experiments confirm that the $L^2$-error of the global
quadratic histopolation operator decays at the expected rate $O\left(h^3\right)$, in line with standard
quadratic finite elements~\cite{Ciarlet2002TFE}, and illustrate the predicted spectral behaviour,
stability, and conditioning of the proposed scheme.

\section{General strategy for weighted quadratic histopolation}
\label{sec2}
Let $S_d \subset \mathbb{R}^d$ be a nondegenerate $d$-simplex with vertices 
$\B v_0,\dots,\B v_d$, and let $\lambda_0,\dots,\lambda_d$ denote its barycentric coordinates. 
For any $k \in \mathbb{N}$, we define
\begin{equation*}
\mathbb{P}_k\left(S_d\right)=\operatorname{span}\left\{\lambda_0^{i_0}\cdots\lambda_d^{i_d}:\ i_0+\cdots+i_d\le k\right\},    
\end{equation*}
the space of all polynomials on $S_d$ of total degree at most $k$.  
It is well known that
\begin{equation*}
    \dim \left(\mathbb{P}_k\left(S_d\right)\right)=\binom{d+k}{k}.
\end{equation*}
The barycentric coordinates $\lambda_i \in \mathbb{P}_1\left(S_d\right)$, $i=0,\dots,d,$ satisfy 
\begin{equation*}
    \sum_{i=0}^d \lambda_i = 1, 
    \quad 
    \lambda_i\left(\B v_j\right) = \delta_{ij},
\end{equation*}
where $\delta_{ij}$ denotes the Kronecker delta. These functions satisfy the following classical identity.
\begin{lemma}\label{lem1}
  Let $\alpha_0,\ldots,\alpha_d$ be positive real numbers. Then
  \begin{equation}\label{idbc}
     \frac{1}{\left|S_d\right|}\int_{S_d}\prod_{i=0}^{d}\lambda_i^{\alpha_i}(\B x)d\B x
     =
     \frac{d!\prod_{i=0}^{d}\Gamma(\alpha_i+1)}{\Gamma\left(d+1+\sum_{i=0}^{d}\alpha_i\right)},
  \end{equation}
  where $\left|S_d\right|$ denotes the volume of $S_d$ and $\Gamma(\cdot)$ is the Gamma function~\cite{Abramowitz:1948:HOM}.
\end{lemma}
For each $j\in\{0,\dots,d\}$, we denote by $F_j$ the face of $S_d$ opposite to the vertex $\B v_j$, and by 
$\left\{\mu_{j,i}\right\}_{i=0, i\neq j}^d$ its barycentric coordinates.

Throughout the paper, we assume that 
$\Omega \in L^1\left(S_d\right)$ is a probability density that is strictly positive almost everywhere in the 
\emph{interior} of $S_d$. 
For each face $F_j$, we define the induced face density $\omega_j$ 
as the normalized conditional trace of $\Omega$ on $F_j$.  
More precisely, for any $\phi \in C^{0}\left(F_j\right)$ we set
\begin{equation*}
    \int_{F_{j}} \phi(\B x)\omega_{j}(\B x)d\B x
=
\lim_{\varepsilon\to 0^{+}}
\frac{
\displaystyle \int_{\left\{\B \sigma\in S_{d}\, :\,\lambda_{j}(\B \sigma)<\varepsilon\right\}} 
\phi\left(\B \sigma\right)\Omega(\B \sigma)d\B \sigma
}{
\displaystyle \int_{\left\{\B \sigma\in S_{d}\,:\,\lambda_{j}(\B \sigma)<\varepsilon\right\}} \Omega(\B \sigma)d\B \sigma
}.
\end{equation*}
We assume that this limit exists and defines a probability density 
$\omega_j \in L^1\left(F_j\right)$ which is positive almost everywhere on $F_j$.
We consider the weighted inner products
\begin{eqnarray*}
    \left\langle f, g \right\rangle_\Omega
    &=& \int_{S_d} f\left(\B \sigma\right) g\left(\B \sigma\right) 
      \Omega\left(\B \sigma\right) d\B \sigma, \quad f,g\in L^2_{\Omega}\left(S_d\right),\\
       \left\langle f,g \right\rangle_{\omega_j}
    &=& \int_{F_j} f\left(\B x\right) g\left(\B x\right) 
      \omega_j\left(\B x\right) d\B x, \quad f,g\in L^2_{\omega_j}\left(F_j\right).
\end{eqnarray*}

Before introducing the full unisolvence framework, we provide a set of explicit 
examples illustrating the meaning of the induced face densities $\omega_j$.
\begin{example}
We consider the following weight functions:
\begin{itemize}
  \item \textbf{Constant weight.}\\
  Let 
  \begin{equation*}
  \Omega(\B \sigma)=\frac{1}{\left|S_d\right|}, \quad \B \sigma\in S_d,  
  \end{equation*}
    be the uniform probability density on the simplex $S_d\subset\mathbb{R}^d$.
  Then, by definition, the induced face density is uniform on each face $F_j$, that is
  \begin{equation*}
    \omega_j(\B x) = \frac{1}{|F_j|}, \quad \B x \in F_j.
  \end{equation*}
  \item \textbf{Affine weight.}\\
 Let
  \begin{equation*}
    \Omega(\B \sigma) = \sum_{i=0}^{d} \alpha_i\lambda_i(\B \sigma), \quad \alpha_i > 0, \quad \B \sigma\in S_d,
  \end{equation*}
be a positive affine weight on $S_d$.
  The induced face density on $F_j$ is 
  \begin{equation*}
    \omega_j(\B x)
    =
    \frac{\displaystyle \sum_{\substack{i=0 \\ i \neq j}}^{d} \alpha_i \mu_{j,i}(\B x)}
         {\displaystyle \int_{F_j} \sum_{\substack{i=0 \\ i \neq j}}^{d} \alpha_i \mu_{j,i}(\B x)d\B x},
    \quad \B x \in F_j.
  \end{equation*}
  \item \textbf{Dirichlet weight.}\\
  Let 
  \begin{equation*}
    \Omega(\B \sigma) = \prod_{i=0}^{d} \lambda_i^{\alpha_i - 1}(\B \sigma), \quad \alpha_i > 0, \quad \B \sigma\in S_d,
  \end{equation*}
be a Dirichlet type weight. Then, on the face $F_j$, the induced face density is
  again of Dirichlet type in dimension $d-1$ with the proportionality constant chosen to ensure normalization on $F_j$.
  Thus, Dirichlet weights are closed under face restriction.
\end{itemize}
\end{example}
We introduce
\begin{equation*}
\mathbb{S}_2\left(S_d\right)=\left\{p\in \mathbb{P}_2\left(S_d\right) \, : \, \langle p,\varphi\rangle_\Omega=0,\ \forall\varphi\in \mathbb{P}_1\left(S_d\right)\right\},
\end{equation*}
the orthogonal complement of $\mathbb{P}_1\left(S_d\right)$ in $\mathbb{P}_2\left(S_d\right)$ 
with respect to the weighted inner product $\left\langle \cdot, \cdot \right\rangle_\Omega$. Accordingly,
\begin{equation}\label{P2}
\mathbb{P}_2\left(S_d\right)=\mathbb{P}_1\left(S_d\right)\oplus \mathbb{S}_2\left(S_d\right),
\end{equation}
and, with a slight abuse of notation, we sometimes write
\begin{equation*}
\mathbb{S}_2\left(S_d\right)=\mathbb{P}_2\left(S_d\right)\ominus\mathbb{P}_1\left(S_d\right).
\end{equation*}
Hence,
\begin{equation}\label{dimS2}
\dim\left(\mathbb{S}_2\left(S_d\right)\right)=\dim\left(\mathbb{P}_2\left(S_d\right)\right)-\dim\left(\mathbb{P}_1\left(S_d\right)\right)= \frac{(d+2)(d+1)}{2}-(d+1)=\frac{d(d+1)}{2}.
\end{equation}
We then select
\begin{equation*}
    \tilde{d}=\frac{(d-2)(d+1)}{2}
\end{equation*}
linearly independent polynomials  $\left\{\rho_k\right\}_{k=1}^{\tilde{d}}$ in $\mathbb{S}_2\left(S_d\right)$, so that the remaining 
$d+1$ basis functions will be associated with the faces. We set 
\begin{equation}\label{spaceV}
\mathbb{V}=\operatorname{span}\left\{\rho_k\, :\, k=1,\dots,\tilde{d} \right\}
\end{equation}
and define
\begin{equation*}
    \mathbb{W}=\operatorname{span}\left\{\psi_j \, :\, j=0,\dots,d\right\}, 
\end{equation*}
so that the decomposition
\begin{equation}\label{S2}
\mathbb{S}_2\left(S_d\right)=\mathbb{V}\oplus\mathbb{W}
\end{equation}
holds.

We denote by 
\begin{equation}\label{projOmega}
\Pi_{1,\Omega}:L^2_{\Omega}\left(S_d\right)\to \mathbb{P}_1\left(S_d\right)    
\end{equation}
and 
\begin{equation}\label{projomegaj}
\Pi_{1,\omega_j}:L^2_{\omega_j}\left(F_j\right)\to\mathbb{P}_1\left(F_j\right)
\end{equation}
the weighted $L^{2}$-orthogonal
projections onto $\mathbb{P}_{1}\left(S_{d}\right)$ and $\mathbb{P}_{1}\left(F_{j}\right)$  with respect to the weights $\Omega$ and $\omega_j$, respectively. We then introduce the following linear functionals on $C^0\left(S_d\right)$:
\begin{itemize}
    \item \textit{Face averages}
    \begin{equation}\label{Ij}
        \mathcal{I}_j: f\mapsto \left\langle f,1 \right\rangle_{\omega_j}=\int_{F_j} f_{_{\mkern 1mu \vrule height 2ex\mkern2mu F_j}}(\B x) \omega_j(\B x) d\B x, \quad j=0,\dots,d;
    \end{equation}
    \item \textit{Face quadratic moments}
 \begin{equation}\label{Lj}
     \mathcal{L}_j:f \mapsto\left\langle f,q_j \right\rangle_{\omega_j}= \int_{F_j}    f_{_{\mkern 1mu \vrule height 2ex\mkern2mu F_j}}(\B x)q_j(\B x) \omega_j(\B x) d\B x, \quad j=0,\dots,d,
 \end{equation}
  where each $q_j \in \mathbb{P}_2\left(F_j\right)$ is chosen 
  such that
  \begin{equation}\label{propqj}
     \left\langle q_j, \varphi \right\rangle_{\omega_j}=0, \quad \forall\varphi\in\mathbb{P}_1\left(F_j\right); 
  \end{equation}
  \item  \textit{Interior linear functionals}  
\begin{equation}\label{Vk}    \mathcal{V}_k:f\mapsto\left\langle f,\rho_k \right\rangle_{\Omega}= \int_{S_d}  f(\B \sigma) \rho_k(\B \sigma) \Omega(\B \sigma) d\B \sigma, \quad k=1,\dots,\tilde{d},
\end{equation}
where $\rho_k$ is defined in~\eqref{spaceV}.
\end{itemize}
The total number of linear functionals is therefore
\begin{equation*}
   (d+1) + (d+1)+\tilde{d}= 2 (d+1)+ \frac{(d+1)(d-2)}{2} = \frac{(d+1)(d+2)}{2} = \dim\left(\mathbb{P}_2\left(S_d\right)\right),
\end{equation*}
which ensures the dimensional consistency of the construction.

\begin{remark}\label{dsdsdsdsdss}
Since the restriction of any $\varphi\in\mathbb{P}_1\left(S_d\right)$ to each face $F_j$
is a linear polynomial on $F_j$, by~\eqref{propqj}, we have
\begin{equation*}
\mathcal L_j(\varphi)=0, \quad j=0,\dots,d.
\end{equation*}
Moreover, in view of~\eqref{spaceV} and by the definition of $\rho_k$, we also obtain
\begin{equation*}
\mathcal V_k(\varphi)=0, \quad k=1,\dots,\tilde{d},    
\end{equation*}
 for any 
$\varphi\in\mathbb{P}_1\left(S_d\right)$. 
\end{remark}

Using the decomposition~\eqref{S2}, any 
$p \in \mathbb{S}_2\left(S_d\right)$ can be written as
\begin{equation*}
    p=v+w, 
\end{equation*}
where
\begin{equation*}
    v=\sum_{\ell=1}^{\tilde{d}} \xi_{\ell} \rho_{\ell}\in \mathbb{V}, \quad w=\sum_{i=0}^{d} \gamma_{i} \psi_{i}\in \mathbb{W}.
\end{equation*}
The vanishing constraints
\begin{equation*}
\mathcal{V}_k(p)=\mathcal{L}_j(p)=0, \quad k=1,\dots,\tilde{d}, \quad j=0,\dots,d,    
\end{equation*}
yield the following block linear system
\begin{equation*}
   H
\begin{pmatrix}
\B \xi\\
\B \gamma
\end{pmatrix}
=\B 0,
\end{equation*}
where $\B \xi=\left(\xi_1,\dots,\xi_{\tilde{d}}\right)^{\top}$, $\B \gamma=\left(\gamma_0,\dots,\gamma_{d}\right)^{\top}$ and
\begin{equation}\label{matH}
H=
\begin{pmatrix}
G & C\\
\tilde C & M
\end{pmatrix}.
\end{equation}
Here the blocks are given by
\begin{eqnarray}
\label{matG} G\in\mathbb{R}^{\tilde{d}\times\tilde{d}} &\text{ with entries }& \left[G\right]_{k\ell}=\mathcal{V}_k\left(\rho_\ell\right)=\left\langle\rho_\ell, \rho_k\right\rangle_{\Omega}, \\ 
\label{matC} C\in\mathbb{R}^{\tilde{d}\times (d+1)} &\text{ with entries }& \left[C\right]_{ki}=\mathcal{V}_k\left(\psi_i\right)=\left\langle\psi_{i},\rho_k\right\rangle_{\Omega}, \\ 
\label{mattildeC} \tilde{C}\in\mathbb{R}^{(d+1)\times \tilde{d}} &\text{ with entries  }& \left[\tilde{C}\right]_{j\ell}=\mathcal{L}_j\left(\rho_\ell\right)=\left\langle\rho_\ell,q_j\right\rangle_{\omega_j}, \\ 
\label{matM} M\in\mathbb{R}^{(d+1)\times (d+1)} &\text{ with entries  }& \left[M\right]_{j i}=\mathcal{L}_j\left(\psi_{i}\right)=\left\langle\psi_{i},q_j\right\rangle_{\omega_j}.
\end{eqnarray}
Since $\Omega$ is strictly positive almost everywhere in the interior of $S_{d}$ and 
$\rho_{1},\dots,\rho_{\tilde d}$ are linearly independent,
the matrix $G$ is a Gram matrix and therefore symmetric positive definite.
In particular,
\begin{equation}\label{Gmat}
    \det(G) \neq 0.
\end{equation}

Finally, we denote by
\begin{equation}\label{matA}
    A\in\mathbb{R}^{(d+1)\times(d+1)} \text{ with entries } [A]_{ji}=\mathcal{I}_j\left(\lambda_{i}\right)=\left\langle \lambda_{i},1 \right\rangle_{\omega_j},
\end{equation}
and obtain the following characterization theorem.

\begin{theorem}
\label{thm:CNS-nonsym}
The set of degrees of freedom
\begin{equation*}
\Sigma=\left\{\mathcal{I}_j,\mathcal{L}_j,\mathcal{V}_k\, :\, j=0,\dots,d, \, k=1,\dots,\tilde{d}\right\}    
\end{equation*}
is unisolvent on $\mathbb{P}_2\left(S_d\right)$ if and only if 
\begin{enumerate}
\item[(a)] the matrix $A$ is invertible;
\item[(b)] the Schur complement
    \begin{equation}\label{matT}
        T = M - \tilde{C} G^{-1} C\in\mathbb{R}^{(d+1)\times(d+1)}
    \end{equation}
    is invertible.
\end{enumerate}
\end{theorem}
\begin{proof}
($\Rightarrow$) 
We assume that $\Sigma$ is unisolvent on $\mathbb{P}_2\left(S_d\right)$.
Suppose, by contradiction, that $A$ is singular.  
Then there exists a nonzero vector
$\B a=\left(a_0,\dots,a_d\right)^{\top}$ such that
\begin{equation*}
    \B a\in \ker(A).
\end{equation*}
Hence the polynomial
$$\varphi=\sum_{j=0}^d a_j\lambda_j \in \mathbb{P}_1\left(S_d\right)$$ satisfies
\begin{equation*}
   \mathcal{I}_j(\varphi)=0, \quad j=0,\dots,d.  
\end{equation*}
At the same time, since each $\rho_k$ is orthogonal to  $\mathbb{P}_1\left(S_d\right)$ 
with respect to  $\langle \cdot, \cdot \rangle_{\Omega}$, we have
\begin{equation}\label{new1}
    \mathcal{V}_k(\varphi)=\left\langle \varphi,\rho_k \right\rangle_{\Omega}=0, \quad k=1,\dots,\tilde{d}.
\end{equation}
Moreover, by the construction of the face functions $q_j$, we obtain
\begin{equation}\label{new2}
    \mathcal{L}_j(\varphi)=\left\langle \varphi,q_j \right\rangle_{\omega_j}=0, \quad j=0,\dots,d.
\end{equation}
Since the barycentric coordinates $\left\{\lambda_j\right\}_{j=0}^d$ form a basis of $\mathbb{P}_1\left(S_d\right)$ and 
$\B a\neq\mathbf 0$, it follows that $\varphi\neq 0$.
Hence, the nonzero polynomial
$\varphi$ annihilates all degrees of freedom, contradicting unisolvence.
Therefore, $A$ must be nonsingular. \\

Next assume that $T$ is singular. 
Then the block matrix $H$ is singular, and there exist vectors
\begin{equation*}
\bar{\B\xi}=\left(\bar{\xi}_1,\dots,\bar{\xi}_{\tilde{d}}\right)^{\top}\in\mathbb{R}^{\tilde{d}}, \quad \bar{\B \gamma}=\left(\bar{\gamma}_0,\dots,\bar{\gamma}_{d}\right)^{\top}\in\mathbb{R}^{d+1}
\end{equation*}
such that
\begin{equation*}
\begin{pmatrix}
\bar{\B\xi}\\[2pt]
\bar{\B\gamma}
\end{pmatrix}\in \ker(H), \quad \begin{pmatrix}
\bar{\B\xi}\\[2pt]
\bar{\B\gamma}
\end{pmatrix}\neq \B 0.
\end{equation*}
We define 
\begin{equation*}
    \bar{v}=\sum_{\ell=1}^{\tilde{d}} \bar{\xi}_\ell \rho_\ell \in \mathbb{V}, \quad
\bar{w}=\sum_{i=0}^{d} \bar{\gamma}_i \psi_i \in \mathbb{W}
\end{equation*}
and we set
\begin{equation*}
\bar{q} =\bar{v} + \bar{w}\in\mathbb{S}_2\left(S_d\right).    
\end{equation*}
Then $\bar{q} \neq 0$ and satisfies
\begin{equation*}
\mathcal{L}_j\left(\bar{q}\right)=0, \quad \mathcal{V}_k\left(\bar{q}\right)=0, \quad j=0,\dots,d, \quad k=1,\dots,\tilde{d}. 
\end{equation*}
Since $A$ is invertible, there exists a unique $\bar \varphi\in\mathbb{P}_1\left(S_d\right)$ such that
\begin{equation}\label{new3}
    \mathcal{I}_j\left(\bar{\varphi}\right)
    = -\mathcal{I}_j\left(\bar{q}\right),
    \quad j = 0,\dots,d.
\end{equation}
Hence $\bar p=\bar \varphi+\bar q$ is nonzero and satisfies
\begin{equation*}
\mathcal{L}_j\left(\bar{p}\right)
=\mathcal{I}_j\left(\bar{p}\right)=\mathcal{V}_k\left(\bar{p}\right)
=0,
\quad j=0,\dots,d,\quad k=1,\dots,\tilde d,
\end{equation*}
again contradicting unisolvence. 
Therefore $T$ must be nonsingular.\\

($\Leftarrow$) 
Conversely, we assume $\det(A)\neq 0$ and $\det(T)\neq 0$. 
Let $p\in\mathbb{P}_2\left(S_d\right)$ satisfy
\begin{equation*}
\mathcal{I}_j(p)=\mathcal{L}_j(p)=\mathcal{V}_k(p)=0, \quad j=0,\dots,d, \quad k=1,\dots,\tilde{d}. 
\end{equation*}
To prove that $\Sigma$ is unisolvent on $\mathbb{P}_2\left(S_d\right)$, it suffices to show that $p=0$. Using~\eqref{P2} and~\eqref{S2}, we can express this polynomial as
\begin{equation*}
p = \varphi + v + w,
\end{equation*}
where
\begin{equation*}
\varphi=\sum_{i=0}^d a_i\lambda_i\in\mathbb{P}_1\left(S_d\right), \quad v=\sum_{\ell=1}^{\tilde{d}} {\xi}_\ell \rho_\ell \in \mathbb{V}, \quad
{w}=\sum_{j=0}^{d} {\gamma}_j \psi_j \in \mathbb{W}.
\end{equation*}
By Remark~\ref{dsdsdsdsdss}, it follows that
\begin{eqnarray}
\label{ddssccc1} 0&=&\mathcal{V}_k(p)=\mathcal{V}_k(v)+\mathcal{V}_k(w), \quad k=1,\dots,\tilde{d},\\ \label{ddssccc}
0&=&\mathcal{L}_j(p)=\mathcal{L}_j(v)+\mathcal{L}_j(w), \quad j=0,\dots,d. 
\end{eqnarray}
By setting
\begin{equation*}
  \B \xi=\left(\xi_1,\dots,\xi_{\tilde{d}}\right)^{\top}, \quad \B \gamma=\left(\gamma_0,\dots,\gamma_{d}\right)^{\top},
\end{equation*}
we can express the relations~\eqref{ddssccc1} and~\eqref{ddssccc} in matrix form as
\begin{equation*}
H\begin{pmatrix}
{\B\xi}\\[2pt]
{\B\gamma}
\end{pmatrix}
=\B 0,   
\end{equation*}
or equivalently
\begin{equation*}
\begin{pmatrix}
G & C\\
\tilde C & M
\end{pmatrix}\begin{pmatrix}
{\B\xi}\\[2pt]
{\B\gamma}
\end{pmatrix}
=\B 0,   
\end{equation*}
where $G$, $C$, $\tilde{C}$ and $M$ are defined in~\eqref{matG}, \eqref{matC}, \eqref{mattildeC} and \eqref{matM}, respectively. 
This system is equivalent to
\begin{equation*}
    \begin{cases}
G\B \xi + C\B \gamma = \B 0, \\[4pt]
\tilde{C}\B \xi + M\B \gamma = \B 0.
\end{cases}
\end{equation*}
Since $G$ is invertible, the first equation gives
\begin{equation}\label{xiaa}
\B \xi = -G^{-1}C\B \gamma.    
\end{equation}
Substituting into the second relation and using~\eqref{matT}, we have
\begin{equation*}
T\B \gamma = \B 0.
\end{equation*}
Since $T$ is nonsingular, $\B\gamma=\mathbf{0}$, and by~\eqref{xiaa},
$\B\xi=\mathbf{0}$. Hence
\begin{equation*}
    p=\varphi=\sum_{i=0}^d a_i\lambda_i,
\end{equation*}
and the remaining constraints reduce to
\begin{equation*}
\mathcal{I}_j(p)=\mathcal{I}_j(\varphi)=0, \quad j=0,\dots,d,
\end{equation*}
i.e.
\begin{equation*}
    A\B a=\B 0, \quad   \B a=\left(a_0,\dots,a_d\right)^{\top}.
\end{equation*}
Since $A$ is nonsingular, $\B a=\mathbf{0}$, and therefore $p=0$.
This proves unisolvence.
\end{proof}

As an immediate consequence of Theorem~\ref{thm:CNS-nonsym} and of the classical
Schur-complement factorization for block matrices (see, e.g.,
\cite[Sec.~3.1]{benzi2005numerical}), the unisolvence condition can be equivalently
reformulated in terms of the nonsingularity of the affine and quadratic moment
matrices.

\begin{corollary}\label{cor1}
The set of degrees of freedom $\Sigma$ is unisolvent on $\mathbb{P}_2\left(S_d\right)$ if and only if both
the affine moment matrix $A$ and the quadratic moment matrix $H$
are nonsingular.
\end{corollary}

\begin{proof}
By Theorem~\ref{thm:CNS-nonsym}, unisolvence of $\Sigma$ is equivalent to the nonsingularity of $A$
and of the associated Schur complement $T$.
It remains to show
that
\[
T = M - \tilde{C}G^{-1}C
\quad\text{is invertible if and only if}\quad
H =
\begin{pmatrix}
G & C\\
\tilde{C} & M
\end{pmatrix}
\ \text{is invertible.}
\]
Since the Gram matrix $G$ is symmetric positive definite and therefore invertible,
the standard Schur-complement factorization yields, see e.g.~\cite{benzi2005numerical,MatComp}
\[
\det(H)=\det(G)\det(T).
\]
Consequently, $H$ is nonsingular if and only if $T$ is nonsingular.
\end{proof}

\begin{remark}
\label{rem:stochasticA}
The matrix $A$ introduced in~\eqref{matA} is a row-stochastic matrix with zero diagonal and strictly positive off-diagonal entries. Indeed, since $\lambda_j=0$ on $F_j$ while $\lambda_i>0$ on $F_j$ for $i\neq j$, and since $\omega_j$ is nonnegative on $F_j$, we have
\begin{equation*}
[A]_{jj}=0
\quad\text{and}\quad
[A]_{ji}
=\mathcal{I}_j\left(\lambda_i\right)
=\int_{F_j}\lambda_i(\B x)\omega_j(\B x)d\B x>0,
\quad i\neq j.
\end{equation*}
Moreover, since $\omega_j$ is normalized on $F_j$, we obtain
\begin{equation*}
\sum_{i=0}^d [A]_{ji}
=\int_{F_j}\left(\sum_{i=0}^d\lambda_i(\B x)\right)\omega_j(\B x)d\B x
=\int_{F_j}\omega_j(\B x)d\B x
=1.
\end{equation*}
When $d=2$, the matrix $A$ is $3\times 3$, and one easily verifies that
every row-stochastic matrix with zero diagonal and strictly positive off-diagonal entries is necessarily nonsingular. For $d\ge 3$, however, this property fails: as shown in~\cite{Tarazaga1993}, 
row-stochastic zero-diagonal matrices of size at least $4$ may be singular even when all off-diagonal entries are strictly positive. 
Hence assumption~($a$) in Theorem~\ref{thm:CNS-nonsym} is essential in higher dimensions.
\end{remark}

\begin{remark}
The proposed framework remains valid under arbitrary partial use of 
quadratic moments, whether located on edges or in the interior of $S_d$. 
More precisely, any subset of face moments $\left\{\mathcal{L}_{j}\right\}$, together with 
an arbitrary user–selected subspace $\mathbb{U} \subset \mathbb{S}_{2}\left(S_d\right)$ of interior 
quadratic polynomials, defines a well-posed histopolation system on a quadratic subspace, 
provided that the face averages $\left\{\mathcal{I}_j\right\}$ are retained. Unisolvence follows from the same structural ingredients used in 
Theorem~\ref{thm:CNS-nonsym}, namely:
\begin{itemize}
    \item[$(i)$] the invertibility of the affine block $A$;
    \item[$(ii)$] the linear independence of the activated 
    quadratic moments on their respective supporting subspace.
\end{itemize}
No further assumptions and no modifications of the proof are required.
\end{remark}

\begin{remark}
While Theorem~\ref{thm:CNS-nonsym} provides necessary and sufficient conditions
for unisolvence, its formulation involves the matrices 
$G$, $C$, $\tilde{C}$, and $M$, which depend on the particular choice
of the face and interior subspaces $\mathbb{W}$ and $\mathbb{V}$.
We introduce the linear map
\begin{equation}\label{LW}
    \mathcal{L} : 
    q \in \mathbb{S}_2\left(S_d\right) \longmapsto 
    \left(\mathcal{L}_0(q),\dots,\mathcal{L}_d(q)\right)^{\top} 
    \in \mathbb{R}^{d+1},
\end{equation}
and denote by $\mathcal{L}_{\mathbb{W}}=\mathcal{L}_{_{\mkern 1mu \vrule height 2ex\mkern2mu \mathbb{W}}}$. 
If the face functionals $\mathcal{L}_0,\dots,\mathcal{L}_d$ are linearly independent on $\mathbb{W}$ 
(equivalently, if $M$ has full rank), then $\mathcal{L}_{\mathbb{W}}$ is an isomorphism. 
In this case, it is possible to select a face basis $\tilde{\psi}_0,\dots,\tilde{\psi}_d \in \mathbb{W}$ such that
\begin{equation*}
\mathcal{L}_{\mathbb{W}}\left(\tilde{\psi}_j\right) = e_j,
\end{equation*}
where $e_j$ denotes the $j$-th canonical basis vector in $\mathbb{R}^{d+1}$. Then
\begin{equation*}
    \mathcal{L}_j\left(\tilde{\psi}_i\right) = \delta_{ij}
    \quad \Rightarrow \quad
    M = I_{d+1}.
\end{equation*}
Such a normalization can be achieved through an appropriate change of basis within $\mathbb{W}$. As a consequence, the Schur complement simplifies to
\begin{equation*}
    T = I_{d+1} - \tilde{C} G^{-1} C,
\end{equation*}
which is particularly convenient for both theoretical analysis 
and the numerical implementation of the associated degrees of freedom.
\end{remark}

\begin{theorem}\label{thm14}
If the face functionals $\mathcal{L}_0,\dots,\mathcal{L}_d$ are linearly independent on $\mathbb{S}_2\left(S_d\right)$,
then there exists a decomposition
\begin{equation*}
    \mathbb{S}_2\left(S_d\right) = \mathbb{V} \oplus \mathbb{W},
\end{equation*}
such that
\begin{itemize}
  \item $\mathbb{W}$ has dimension $d+1$ and satisfies 
  $\mathbb{W} \perp \mathbb{V}$ with respect to the inner product 
  $\left\langle \cdot,\cdot \right\rangle_{\Omega}$;
  \item the matrix $M$, defined in~\eqref{matM}, is invertible;
  \item by an appropriate change of basis in $\mathbb{W}$, it is possible to normalize $M$ to the identity, $M = I_{d+1}$.
\end{itemize}
If, in addition, the basis of $\mathbb{V}$ is chosen to be orthonormal  
with respect to $\langle\cdot,\cdot\rangle_{\Omega}$, then $G = I_{\tilde d}$, $C=0$,
and the Schur complement associated with the unisolvence condition takes the canonical form
\begin{equation*}
    T = M = I_{d+1}.
\end{equation*}
\end{theorem}

\begin{proof}
Since the functionals $\mathcal{L}_0,\dots,\mathcal{L}_d$ are linearly 
independent on $\mathbb{S}_2\left(S_d\right)$, the linear map
\begin{equation*}
    \mathcal{L} : q \in \mathbb{S}_2\left(S_d\right) \longmapsto 
\left(\mathcal{L}_0(q),\dots,\mathcal{L}_d(q)\right)^{\top} \in \mathbb{R}^{d+1}
\end{equation*}
has rank $d+1$. Then $\mathcal{L}$ is 
surjective. Hence, for each canonical vector $e_i \in \mathbb{R}^{d+1}$ 
there exists $\psi_i \in \mathbb{S}_2\left(S_d\right)$ such that 
\begin{equation*}
\mathcal{L}\left(\psi_i\right) = e_i,
\quad i=0,\dots,d.    
\end{equation*}
Equivalently,
\begin{equation*}
\mathcal{L}_j\left(\psi_i\right) = \delta_{ij},
\quad i,j = 0,\dots,d,   
\end{equation*}
so that the matrix
\begin{equation*}
    M= \left[\mathcal{L}_j\left(\psi_i\right)\right]_{i,j=0}^d
\end{equation*}
is the identity matrix and, in particular, is invertible. We set
\begin{equation*}
    \mathbb{W} = \operatorname{span}\left\{\psi_0,\dots,\psi_d\right\}.
\end{equation*}
By construction, the linear map 
$\mathcal{L}_{\mathbb{W}} : \mathbb{W} \to \mathbb{R}^{d+1}$, defined in~\eqref{LW}, 
is an isomorphism, and the corresponding matrix with respect to the basis 
$\left\{\psi_0,\dots,\psi_d\right\}$ of $\mathbb{W}$ and the canonical basis of $\mathbb{R}^{d+1}$ 
is $M= I_{d+1}$.

We then define the interior subspace as
\begin{equation*}
    \mathbb{V} 
    = \left\{
        v \in \mathbb{S}_2\left(S_d\right) :
        \langle v, w \rangle_\Omega = 0 \ \text{for all } w \in \mathbb{W}
      \right\}.
\end{equation*}
Hence,
\begin{equation*}
    \mathbb{S}_2\left(S_d\right) = \mathbb{V} \oplus \mathbb{W},
    \quad \mathbb{V} \perp \mathbb{W}.
\end{equation*}
If, in addition, the basis of $\mathbb{V}$ is chosen to be orthonormal 
with respect to $\langle\cdot,\cdot\rangle_{\Omega}$, then $G = I_{\tilde d}$ and $C=0$.
Consequently, the Schur complement associated with the unisolvence condition reduces to
\begin{equation*}
    T = M - \tilde{C} G^{-1} C = I_{d+1}.
\end{equation*}
\end{proof}

\begin{remark}
If the face functionals $\mathcal{L}_0,\dots,\mathcal{L}_d$ are linearly dependent on $\mathbb{S}_2\left(S_d\right)$, 
then for any $(d+1)$-dimensional subspace 
$\mathbb{W} \subset \mathbb{S}_2\left(S_d\right)$ and any basis 
$\left\{\psi_j\right\}_{j=0}^d$ of          $\mathbb{W}$, 
the associated face matrix $M$ is singular.
Hence, the face-separation assumption 
is not only sufficient but also necessary for the construction of an orthogonal decomposition
\begin{equation*}
    \mathbb{S}_2\left(S_d\right) = \mathbb{V} \oplus \mathbb{W},
\end{equation*}
with an invertible face block.
\end{remark}

To conclude the unisolvence analysis, we exhibit an explicit orthogonal decomposition of the quadratic space and a constructive choice of the face test polynomials that makes the face matrix $M$ invertible.

\begin{theorem}\label{newtheam}
If the matrix $A$ defined in~\eqref{matA} is invertible, then there exist explicit choices of
\begin{enumerate}
    \item functions $q_j\in \mathbb{S}_2\left(F_j\right)$;
    \item an interior space $\mathbb{V}$ with orthonormal basis
$\left\{\rho_k\right\}_{k=1}^{\tilde{d}}$; 
    \item a face space $\mathbb{W}$ with basis $\left\{\psi_j\right\}_{j=0}^d$
\end{enumerate}
such that:
\begin{enumerate}
\item the set $\Sigma$ is unisolvent on $\mathbb{P}_2\left(S_d\right)$;
\item the spaces $\mathbb{V}$ and $\mathbb{W}$ are orthogonal with respect to 
      $\langle\cdot,\cdot\rangle_{\Omega}$, and therefore the associated blocks satisfy 
      $G = I_{\tilde d}$ and $C = 0$.
\end{enumerate}
\end{theorem}

\begin{proof}
Let $p\in\mathbb P_2\left(S_d\right)$ satisfy
\begin{equation*}
    \mathcal I_j(p)=\mathcal L_j(p)=\mathcal V_k(p)=0,\quad j=0,\dots,d,\ k=1,\dots,\tilde d.
\end{equation*}
Since the matrix $A$ is invertible, Theorem~\ref{thm:CNS-nonsym}
implies that unisolvence follows once the corresponding Schur complement
  \begin{equation*}
        T = M - \tilde{C} G^{-1} C\in\mathbb{R}^{(d+1)\times(d+1)}
    \end{equation*}
is invertible.

We construct an explicit decomposition
\begin{equation*}
\mathbb{S}_2\left(S_d\right)=\mathbb{V}\oplus \mathbb{W},
\end{equation*}
for which $G = I_{\tilde d}$, $C = 0$, and $T = M$. To this aim, let
\begin{equation}\label{gjj}
    g_j=\lambda_j\lambda_{j+1},\quad j=0,\dots,d, \text{ mod}(d+1),
\end{equation}
and define
\begin{equation*}
    \psi_j=\left(I-\Pi_{1,\Omega}\right)\left(g_j\right)\in \mathbb{S}_2\left(S_d\right),
\end{equation*}
where $\Pi_{1,\Omega}$ is defined in~\eqref{projOmega}. Since
\begin{equation*}
    \operatorname{span}\left\{\lambda_j\lambda_{j+1}\, : \, j=0,\dots,d\right\}\cap \mathbb{P}_1\left(S_d\right)=\{0\},
\end{equation*}
it follows that, if 
\begin{equation*}
    0=\sum_{j=0}^d c_j \psi_j=\sum_{j=0}^d c_j \left(I-\Pi_{1,\Omega}\right)\left(g_j\right)=\left(I-\Pi_{1,\Omega}\right)\left(\sum_{j=0}^d c_j g_j\right),
\end{equation*}
then 
\begin{equation*}
    \sum_{j=0}^d c_j g_j\in \mathbb{P}_1\left(S_d\right).
\end{equation*}
But, by definition~\eqref{gjj}, we also have
\begin{equation*}
   \sum_{j=0}^d c_j g_j\in \operatorname{span}\left\{\lambda_j\lambda_{j+1}\, : \, j=0,\dots,d\right\}.
\end{equation*}
Then
\begin{equation*}
    \sum_{j=0}^d c_j g_j\in \operatorname{span}\left\{\lambda_j\lambda_{j+1}\, : \, j=0,\dots,d\right\}\cap \mathbb{P}_1\left(S_d\right)
\end{equation*}
and hence
\begin{equation*}
     \sum_{j=0}^d c_j g_j=0.
\end{equation*}
Since the family $\left\{g_j\right\}_{j=0}^d$ is linearly independent, we conclude $c_0=\dots=c_d=0$. Consequently, the set $\left\{\psi_j\right\}_{j=0}^d$ is linearly independent. We therefore set
\begin{equation*}
   \mathbb{W}=\operatorname{span}\left\{\psi_0,\dots,\psi_d\right\}\subset \mathbb{S}_2\left(S_d\right),
\end{equation*}
and define $\mathbb{V}$ as the $\langle\cdot,\cdot\rangle_\Omega$–orthogonal complement of $\mathbb{W}$ in $\mathbb{S}_2\left(S_d\right)$. 
Hence $\mathbb{V}\perp \mathbb{W}$ and
\begin{equation*}
\mathbb{S}_2\left(S_d\right)=\mathbb{V}\oplus \mathbb{W}. 
\end{equation*}
Let $\left\{\rho_k\right\}_{k=1}^{\tilde d}$ be an orthonormal basis of $\mathbb{V}$ with respect to $\langle\cdot,\cdot\rangle_\Omega$. Then
\begin{equation*}
G=\left[\left\langle\rho_\ell,\rho_k\right\rangle_\Omega\right]_{k\ell}=I_{\tilde d},\quad
C=\left[\left\langle \psi_i,\rho_k\right\rangle_\Omega\right]_{ki}=0.    
\end{equation*}
In this setting, the Schur complement reduces to
\begin{equation*}
    T=M-\tilde CG^{-1}C=M.
\end{equation*}
We now show that $M$ is invertible. To this end, we fix $j\in\{0,\dots,d\}$ and we set 
\begin{equation}\label{well}
h_{\ell}=\left(I-\Pi_{1,\omega_j}\right)\left(g_{{\ell}_{\mkern 1mu \vrule height 2ex\mkern2mu F_j}}\right)\in\mathbb{S}_2\left(F_j\right), \quad \ell\notin\{j-1,j\},   
\end{equation}
and define the linear functionals 
\begin{equation*}
\chi_\ell:q\in\mathbb{S}_2\left(F_j\right)\to \left\langle h_{\ell},q\right\rangle_{\omega_j}\in\mathbb{R}, \quad \ell\in B_j:=\{0,\dots,d\}\setminus\{j-1,j\}.
\end{equation*}
The family of $\left\{g_{{\ell}_{\mkern 1mu \vrule height 2ex\mkern2mu F_j}}\right\}_{\ell\in B_j}$ is linearly independent, since each function is a
product of two face barycentric coordinates and vanishes on a distinct subset of the edges of
$F_j$. Since $\Pi_{1,\omega_j}$ is an orthogonal projection onto $\mathbb{P}_1\left(F_j\right)$, the same
injectivity argument used above applies and shows that the family
$\left\{h_\ell\right\}_{\ell\in B_j}$ is linearly independent. In particular, for a fixed $\ell^{\star}\in B_j$, we have
\begin{equation*}
h_{\ell^{\star}}\notin\operatorname{span}\left\{h_{\ell}\, :\, \ell\in B_j\setminus\{\ell^{\star}\}\right\}.
\end{equation*}
As the inner product $\langle\cdot,\cdot\rangle_{\omega_j}$ is nondegenerate on $\mathbb{S}_2\left(F_j\right)$, the associated set of linear functionals $\left\{\chi_\ell\right\}_{\ell\in B_j}$ is also linearly independent in the dual space $\left(\mathbb{S}_2\left(F_j\right)\right)^{\star}$. Therefore, since $$\dim\left(\mathbb{S}_2\left(F_j\right)\right)>\left|B_j\right|-1=d-2,$$ there exists a nonzero element
\begin{equation}\label{cond1}
    q_j \in \bigcap_{\ell\in B_j\setminus\{\ell^{\star}\}}\ker\left(\chi_\ell\right)\subset \mathbb{S}_2\left(F_j\right),
\end{equation}
and moreover
\begin{equation}\label{cond2}
\chi_{\ell^{\star}}\left(q_j\right)=\langle h_{\ell^{\star}},q_j\rangle_{\omega_j}\neq 0.
\end{equation}
Using~\eqref{well} and~\eqref{propqj}, we have
\begin{equation}\label{cond3}
\left\langle h_{\ell},q_j\right\rangle_{\omega_j}
=\left\langle \left(I-\Pi_{1,\omega_j}\right)\left(g_{{\ell}_{\mkern 1mu \vrule height 2ex\mkern2mu F_j}}\right),q_j\right\rangle_{\omega_j}
=\left\langle g_{{\ell}_{\mkern 1mu \vrule height 2ex\mkern2mu F_j}},q_j\right\rangle_{\omega_j}, \quad \ell\in B_j.
\end{equation}
Combining~\eqref{cond1},~\eqref{cond2} and~\eqref{cond3}, we conclude
\begin{equation*}
\left\langle g_{{\ell}_{\mkern 1mu \vrule height 2ex\mkern2mu F_j}},q_j\right\rangle_{\omega_j}=0, \quad \ell\in B_j\setminus\{\ell^{\star}\}, \quad \left\langle g_{{\ell^{\star}}_{\mkern 1mu \vrule height 2ex\mkern2mu F_j}},q_j\right\rangle_{\omega_j}\neq0.
\end{equation*}
By selecting the indices $\ell^{\star}$ in a coordinated way, for instance
\begin{equation*}
    \ell^{\star}=j+2, \ \operatorname{mod}(d+1),
\end{equation*}
each row of the face matrix
\begin{equation*}
M=\left[\mathcal{L}_j\left(\psi_{i}\right)\right]_{ji}
=\left[\mathcal{L}_j\left(g_i\right)\right]_{ji}
\end{equation*}
contains exactly one nonzero entry, located in column $\ell^{\star}=j+2$.  
Since the chosen indices $\left\{\ell^{\star}\right\}_{j=0}^d$ are pairwise distinct,
the nonzero columns are all different. Up to a permutation of columns,
$M$ is therefore diagonal with nonzero diagonal entries, and hence invertible.
\end{proof}

\begin{remark}
In summary, the orthogonal construction of Theorem~\ref{newtheam} ensures that, under the sole
assumption that the matrix $A$ is invertible, it is possible to realize a configuration with
$G = I_{\tilde d}$, $C = 0$, and $T = M$ invertible, thereby guaranteeing unisolvence according to
Theorem~\ref{thm:CNS-nonsym}.
\end{remark}

\subsection{Inf-sup stability via the Schur complement}
We now investigate the \emph{quantitative stability} of the quadratic histopolation problem.
Building on the unisolvence result of Theorem~\ref{thm:CNS-nonsym}, we derive an
inf--sup condition that characterizes the robustness of the coupling between interior and
face degrees of freedom. In particular, Theorem~\ref{thm:CNS-nonsym} provides a qualitative unisolvence
criterion based on the invertibility of the matrices $A$ and
\begin{equation*}
    T = M - \tilde{C} G^{-1} C.
\end{equation*}
Although this characterization guarantees well-posedness, it remains
\emph{purely qualitative} and does not quantify the stability of the reconstruction, nor
the strength of the coupling encoded in the Schur complement~$T$. Our goal is therefore to obtain a quantitative refinement. In particular, we show that
the \emph{inf--sup stability} of the enriched quadratic histopolation scheme is determined by the
spectral properties of a reduced operator~$\hat{S}$. This leads to an explicit and dimension-independent
expression for the stability constant.  Before establishing the main result,
we introduce the necessary preliminary definitions and auxiliary lemmas.\\

Throughout this section, we assume that
the matrix $M$ is \emph{invertible}, which corresponds to the case where the linear map $\mathcal{L}$ defined in~\eqref{LW} is linearly independent
on~$\mathbb{W}$. We denote by
\begin{equation}\label{matDs}
  D =
  \begin{pmatrix}
  M^{\top}M & 0\\[2pt]
  0 & G
  \end{pmatrix},
  \quad
  \B u(q)
  =  \begin{pmatrix}
 \mathcal{L}(q) \\[2pt]
 \mathcal{V}(q)
  \end{pmatrix}
  \in\mathbb{R}^{(d+1)+\tilde{d}}, \quad q\in\mathbb{S}_2\left(S_d\right),
\end{equation}
where
\begin{equation}\label{funLeV}
\mathcal{L}(q)
=\left(\mathcal{L}_0(q),\dots,\mathcal{L}_d(q)\right)^{\top}
\in\mathbb{R}^{d+1},
\quad
\mathcal{V}(q)
=\left(\mathcal{V}_1(q),\dots,\mathcal{V}_{\tilde{d}}(q)\right)^{\top}
\in\mathbb{R}^{\tilde{d}}.
\end{equation}
For any
\begin{equation*}
\B y(\B\eta,\B\zeta)= \begin{pmatrix}
 \B \eta \\[2pt]
  \B \zeta
  \end{pmatrix}
\in Y := \mathbb{R}^{d+1}\times \mathbb{R}^{\tilde{d}},
\quad
\B\eta=\left(\eta_0,\dots,\eta_d\right)^{\top},\quad
\B\zeta=\left(\zeta_1,\dots,\zeta_{\tilde d}\right)^{\top},
\end{equation*}
we define the norm
\begin{equation}\label{semnorm}
   \left \| \B y(\B \eta,\B \zeta)\right\|_{Y,0}^2
    = \B\eta^{\top} \left(M^{\top}M\right)^{-1} \B\eta
    + \B\zeta^{\top} G^{-1} \B\zeta,
\end{equation}
where $M$ and $G$ are the matrices defined in~\eqref{matM} and~\eqref{matG}.
The associated coupling bilinear form is
\begin{equation}\label{bil:form}
\Psi(q;\B\eta,\B\zeta)
= \sum_{j=0}^{d}\eta_j\,\mathcal{L}_j(q)
  + \sum_{k=1}^{\tilde d}\zeta_k\,\mathcal{V}_k(q),
  \quad q\in \mathbb{S}_2\left(S_d\right),\quad \B y(\B \eta,\B \zeta)= 
   \begin{pmatrix}
 \B \eta \\[2pt]
  \B \zeta
  \end{pmatrix}\in Y.
\end{equation}

\begin{lemma}\label{lem:dual-sup}
For any $q\in\mathbb{S}_2\left(S_d\right)$, the following identity holds
\begin{equation*}
    \sup_{\B y(\B\eta,\B\zeta)\neq \B 0}
    \frac{\Psi(q;\B\eta,\B\zeta)}
         {\left\|\B y(\B\eta,\B\zeta)\right\|_{Y,0}}
    =
    \left(
      \mathcal{L}(q)^{\top}\left(M^{\top}M\right)\mathcal{L}(q)
      + \mathcal{V}(q)^{\top}G\mathcal{V}(q)
    \right)^{1/2}. 
\end{equation*}
\end{lemma}

\begin{proof}
If $\B u(q)=\B 0$, the result is trivial. Hence, in the following we assume that $\B u(q)\neq\B 0$.

For any $\B y(\B\eta,\B\zeta)= \begin{pmatrix}
 \B \eta \\[2pt]
  \B \zeta
  \end{pmatrix}
\in Y,$ by~\eqref{matDs} and~\eqref{semnorm}, it holds
\begin{equation}\label{eq:norm-bilinear}
   \left\|\B y(\B\eta,\B\zeta)\right\|_{Y,0}^2
    =
    \begin{pmatrix}
 \B \eta \\[2pt]
  \B \zeta
  \end{pmatrix}^{\top}
    D^{-1}
    \begin{pmatrix}
 \B \eta \\[2pt]
  \B \zeta
  \end{pmatrix}.
\end{equation}
Moreover, by~\eqref{matDs} and~\eqref{bil:form}, it results
\begin{equation}\label{eq:psi-inner}
    \Psi\left(q;\B\eta,\B\zeta\right)
    =
   \begin{pmatrix}
 \B \eta \\[2pt]
  \B \zeta
  \end{pmatrix}^{\top}\B u(q).
\end{equation}
Since $M^{\top}M$ and $G$ are symmetric positive definite, their block-diagonal
combination $D$ is also symmetric positive definite. Hence $D^{-1}$ exists and
induces the inner product
\begin{equation*}
    \left\langle \B y_1,\B y_2\right\rangle_{D^{-1}}
    := \B y_1^{\top} D^{-1} \B y_2,
    \quad
    \B y_1,\B y_2\in \mathbb{R}^{(d+1)+\tilde d}.
\end{equation*}
Then, by~\eqref{eq:norm-bilinear}, we have
\begin{equation}\label{dsdsaaaa}
   \left\|\B y(\B\eta,\B\zeta)\right\|_{Y,0}
    = \left\|\B y(\B\eta,\B\zeta)\right\|_{D^{-1}}
    = \sqrt{\begin{pmatrix}
 \B \eta \\[2pt]
  \B \zeta
  \end{pmatrix}^{\top}D^{-1}\begin{pmatrix}
 \B \eta \\[2pt]
  \B \zeta
  \end{pmatrix}}.
\end{equation}
Combining~\eqref{eq:psi-inner} and~\eqref{dsdsaaaa}, we have
\begin{equation}\label{eq:ratio}
\frac{\Psi\left(q;\B\eta,\B\zeta\right)}
     {\left\|\B y(\B\eta,\B\zeta)\right\|_{Y,0}}
= \frac{\B y(\B\eta,\B\zeta)^{\top}\B u(q)}
       {\left\|\B y(\B\eta,\B\zeta)\right\|_{D^{-1}}}
= \frac{\left\langle \B y(\B\eta,\B\zeta), D \B u(q)\right\rangle_{D^{-1}}}
       {\left\|\B y(\B\eta,\B\zeta)\right\|_{D^{-1}}}.
\end{equation}
Using the Cauchy--Schwarz inequality with respect to the inner product
$\langle\cdot,\cdot\rangle_{D^{-1}}$, and exploiting the symmetry of $D$, we obtain
\begin{equation}\label{eq:upper}
\frac{\left\langle \B y(\B\eta,\B\zeta), D\B u(q)\right\rangle_{D^{-1}}}
     {\left\|\B y(\B\eta,\B\zeta)\right\|_{D^{-1}}}
\le \left\|D\B u(q)\right\|_{D^{-1}}
= \sqrt{(D\B u(q))^{\top}D^{-1}(D\B u(q))}
= \sqrt{\B u(q)^{\top} D \B u(q)}.
\end{equation}
Combining~\eqref{eq:ratio} and~\eqref{eq:upper} gives
\begin{equation*}
    \frac{\Psi\left(q;\B\eta,\B\zeta\right)}
     {\left\|\B y(\B\eta,\B\zeta)\right\|_{Y,0}}\le \sqrt{\B u(q)^{\top} D \B u(q)}.
\end{equation*}
Equality is achieved by choosing 
\begin{equation*}
   \B y(\B\eta,\B\zeta)= D\B u(q).
\end{equation*}
Indeed, using~\eqref{eq:norm-bilinear} and~\eqref{eq:psi-inner}, and exploiting the symmetry of $D$, we obtain
\begin{eqnarray*}
 \Psi\left(q;D\B u(q)\right)
    &=& (D\B u(q))^{\top}\B u(q)
    = \B u(q)^{\top} D \B u(q),\\
    \left\| D\B u(q)\right\|_{Y,0}
    &=& \sqrt{ (D\B u(q))^{\top} D^{-1} (D\B u(q)) }
    = \sqrt{\B u(q)^{\top} D \B u(q)}.
\end{eqnarray*}
Then this proves that
\begin{equation*}
    \sup_{ \B y(\B\eta,\B\zeta)\neq \B 0}
    \frac{\Psi\left(q;\B\eta,\B\zeta\right)}
         {\left\| \B y(\B\eta,\B\zeta)\right\|_{Y,0}}
    =
    \sqrt{\B u(q)^{\top} D \B u(q)}
    =
    \left(
      \mathcal{L}(q)^{\top}\left(M^{\top}M\right)\mathcal{L}(q)
      + \mathcal{V}(q)^{\top}G\,\mathcal{V}(q)
    \right)^{1/2}.
\end{equation*}
\end{proof}

The following lemma shows that, once the quadratic polynomial $q$ is decomposed
into its interior and face components, the combined contribution of the face
and volume moments admits a block representation involving the matrices
$G$, $C$, $\tilde C$ and $M$.

\begin{lemma}\label{lem:block-form}
For any $\B\xi=\left(\xi_1,\dots,\xi_{\tilde d}\right)^{\top}\in\mathbb{R}^{\tilde d}$ and $\B\gamma=\left(\gamma_0,\dots,\gamma_d\right)^{\top}\in\mathbb{R}^{d+1}$, we consider
\begin{equation}\label{ausaas}
  q=v+w\in\mathbb{S}_2\left(S_d\right), \quad v=\sum_{\ell=1}^{\tilde d}\xi_\ell\rho_\ell\in\mathbb{V},\quad
w=\sum_{i=0}^{d}\gamma_i\psi_i\in\mathbb{W}.  
\end{equation}
The following identity holds
\begin{equation}\label{claim2}
  \mathcal{L}(q)^{\top}\left(M^{\top}M\right)   \mathcal{L}(q) + \mathcal{V}(q)^{\top} G \mathcal{V}(q)
= 
\begin{pmatrix}\B\xi\\ \B\gamma\end{pmatrix}^{\top}
\begin{pmatrix}K_{11}&K_{12}\\ K_{21}&K_{22}\end{pmatrix}
\begin{pmatrix}\B\xi\\ \B\gamma\end{pmatrix},
\end{equation}
where
\begin{eqnarray}\label{matKij}
    K_{11}&=& \tilde C^{\top}(M^{\top}M)\tilde C + G^3\in \mathbb{R}^{\tilde{d}\times \tilde{d}},\\ \notag
    K_{12}&=& \tilde C^{\top}(M^{\top}M)M + G^2 C\in \mathbb{R}^{\tilde{d}\times(d+1)},\\ \notag
    K_{21}&=& M^{\top}(M^{\top}M)\tilde C + C^{\top}G^2\in \mathbb{R}^{(d+1)\times\tilde{d}},\\ \notag
    K_{22}&=& M^{\top}(M^{\top}M)M + C^{\top}GC\in\mathbb{R}^{(d+1)\times(d+1)},
\end{eqnarray}
with 
\begin{equation*}
    G^2=GG, \quad G^3=GGG.
\end{equation*}
\end{lemma}

\begin{proof}
Using~\eqref{ausaas} and the linearity of the moment functionals 
$\mathcal{L}_j$ and $\mathcal{V}_k$, we obtain
\begin{eqnarray*}
\mathcal{L}_j(q)&=&\sum_{\ell=1}^{\tilde d}\xi_\ell\mathcal{L}_j\left(\rho_\ell\right)+\sum_{i=0}^{d}\gamma_i \mathcal{L}_j\left(\psi_i\right), \quad j=0,\dots,d,\\
    \mathcal{V}_k(q)&=&\sum_{\ell=1}^{\tilde d}\xi_\ell\mathcal{V}_k\left(\rho_\ell\right)+\sum_{i=0}^{d}\gamma_i\mathcal{V}_k\left(\psi_i\right), \quad k=1,\dots,\tilde{d}.   
\end{eqnarray*}
Then, using the definitions of the matrices in 
\eqref{matG}, \eqref{matC}, \eqref{mattildeC}, and~\eqref{matM}, we obtain
\begin{equation*}
\mathcal{L}(q)=\tilde C\B\xi+M\B\gamma,\quad
    \mathcal{V}(q)=G\B\xi+C\B\gamma,  
\end{equation*}
where $\mathcal{L}$ and $\mathcal{V}$ are defined in~\eqref{funLeV}.
Hence,
\begin{eqnarray*}
    \mathcal{L}(q)^{\top}\left(M^{\top}M\right)   \mathcal{L}(q)&=& \left(\tilde C\B\xi+M\B\gamma\right)^{\top}\left(M^{\top}M\right)\left(\tilde C\B\xi+M\B\gamma\right), \\
     \mathcal{V}(q)^{\top}G \mathcal{V}(q)&=&(G\B\xi+C\B\gamma)^{\top}G\,(G\B\xi+C\B\gamma).
\end{eqnarray*}
By expanding both expressions, collecting the coefficients of $\B\xi$ and
$\B\gamma$, and exploiting the symmetry of $G$, we obtain the block matrix in~\eqref{matKij} with entries $K_{11}$, $K_{12}$, $K_{21}$, and $K_{22}$.
\end{proof}

\begin{remark}
We observe that the matrices $K_{11}$ and $K_{22}$ are symmetric, and that
$K_{21}=K_{12}^\top$. Moreover, since both $M^{\top}M$ and $G$ are symmetric
positive definite, the matrix
\begin{equation*}
    K_{22} = M^{\top}(M^{\top}M)M + C^{\top} G C
\end{equation*}
is symmetric positive definite as well. Indeed, for any $\B x \in \mathbb{R}^{d+1}$, we get
\begin{eqnarray*}
    \B x^{\top} K_{22} \B x
    &=& 
    \B x^{\top}\left(M^{\top}(M^{\top}M)M + C^{\top}GC\right)\B x\\[4pt]
    &=& 
    (M\B x)^{\top}(M^{\top}M)(M\B x)
    + (C\B x)^{\top}G(C\B x).
\end{eqnarray*}
If $\B x \neq \B 0$, then $M\B x \neq \B 0$ because $M$ is invertible, and hence
\begin{equation*}
    (M\B x)^{\top}\left(M^{\top}M\right)(M\B x) > 0,
\end{equation*}
while the second term is nonnegative since $G$ is symmetric positive definite.  
Therefore, 
\begin{equation*}
\B x^{\top}K_{22}\B x > 0, \quad \forall \B x \neq \B 0, 
\end{equation*} 
which proves that $K_{22}$ is symmetric positive definite.
\end{remark}

For the forthcoming results, we set
\begin{equation}\label{matS}
S = K_{11} - K_{12} K_{22}^{-1} K_{21},
\end{equation}
and we define
\begin{equation}\label{matThat}
    \hat{S} = G^{-1/2} S G^{-1/2},
\end{equation}
where $G$ is the Gram matrix defined in~\eqref{matG}. Since $K_{11}$ and $K_{22}$ are symmetric and $K_{21}=K_{12}^{\top}$, 
the Schur complement $S$ is also symmetric.

\begin{lemma}\label{lem:eliminate}
For every $\B\xi \in \mathbb{R}^{\tilde{d}}$, the following identity holds
\begin{equation*}
\min_{\B\gamma\in\mathbb{R}^{d+1}} 
\left(
\B\xi^{\top}K_{11}\B\xi
+2\B\xi^{\top}K_{12}\B\gamma
+\B\gamma^{\top}K_{22}\B\gamma
\right)
=
\B\xi^{\top}S\B\xi,
\end{equation*}
where $S$ is defined in~\eqref{matS}. The minimizer is unique and is given by  
\begin{equation*}
    \bar{\B\gamma}=-K_{22}^{-1}K_{21}\B\xi.
\end{equation*} 
\end{lemma}

\begin{proof}
For any fixed $\B\xi\in\mathbb{R}^{\tilde{d}}$, we define
\begin{equation*}
F_{\B \xi}(\B\gamma)
=\B\xi^{\top}K_{11}\B\xi
+2\B\xi^{\top}K_{12}\B\gamma
+\B\gamma^{\top}K_{22}\B\gamma, \quad \B \gamma\in\mathbb{R}^{d+1}.
\end{equation*}
Since $K_{22}$ is symmetric positive definite, $F_{\B\xi}$ is strictly convex in $\B\gamma$.
Differentiating with respect to $\B\gamma$ gives
\begin{equation*}
\nabla_{\B\gamma}F_{\B \xi}(\B\gamma)
=2\left(K_{12}^{\top}\B\xi+K_{22}\B\gamma\right).
\end{equation*}
Since $K_{12}^{\top}=K_{21}$, the unique stationary point $\bar{\B\gamma}$ satisfies
\begin{equation*}
    K_{22}\B\gamma=-K_{21}\B\xi,
\end{equation*}
hence, since $K_{22}$ is invertible,
\begin{equation*}
\bar{\B\gamma}=-K_{22}^{-1}K_{21}\B\xi.
\end{equation*}
Substituting $\bar{\B\gamma}$ into $F_{\B\xi}$ and using the symmetry of $K_{22}$, we obtain
\begin{equation*}
F_{\B\xi}(\bar{\B\gamma})
=\B\xi^{\top}\left(K_{11}-K_{12}K_{22}^{-1}K_{21}\right)\B\xi=\B\xi^{\top}S\B\xi,
\end{equation*}
which proves the claim.
\end{proof}

The next lemma relates the constrained minimization of the quadratic form
associated with $S$ to the smallest eigenvalue of the normalized matrix
$\hat{S}$ defined in~\eqref{matThat}.

\begin{lemma}\label{lem:rayleigh}
The following identity holds
\begin{equation*}
\inf_{\B\xi^{\top}G\B\xi=1}
\B\xi^{\top}S\B\xi
=
\sigma_{\min}\left(
\hat{S}
\right),
\end{equation*}
where $\sigma_{\min}\left(\hat{S}\right)$ denotes the smallest eigenvalue of $\hat{S}$.
\end{lemma}

\begin{proof}
 Let $\B\eta = G^{1/2}\B\xi$. Then
\begin{equation*}
\|\B\eta\|_2^2=\B\eta^{\top}\B\eta = \B\xi^{\top}G\B\xi,  
\end{equation*}
 so that the normalization $\B\xi^{\top}G\B\xi=1$ is equivalent to $\|\B\eta\|_2=1$.
Moreover,
\begin{equation*}
\B\xi^{\top}S\B\xi=\left(G^{-1/2}\B\eta\right)^{\top} S \left(G^{-1/2}\B\eta\right)= \B\eta^{\top}\left(G^{-1/2} S G^{-1/2}\right)\B\eta.
\end{equation*}
Therefore,
\begin{equation}\label{aus1c}
\inf_{\B\xi^{\top}G\B\xi=1}\B\xi^{\top}S\B\xi
=
\inf_{\|\B\eta\|_2=1}\B\eta^{\top}\left(G^{-1/2} S G^{-1/2}\right)\B\eta.
\end{equation}
Since $G^{-1/2} S G^{-1/2}$ is symmetric, the minimum of its Rayleigh quotient
over the Euclidean unit sphere coincides with its smallest eigenvalue, that is,
\begin{equation}\label{aus1c1}
\inf_{\|\B\eta\|_2=1}\B\eta^{\top}\left(G^{-1/2} S G^{-1/2}\right)\B\eta=\sigma_{\min}\left(G^{-1/2} S G^{-1/2}\right)=\sigma_{\min}\left(\hat{S}\right).
\end{equation}
Combining~\eqref{aus1c} and~\eqref{aus1c1} yields
\begin{equation*}
\inf_{\B\xi^{\top}G\B\xi=1}\B\xi^{\top}S\B\xi
=
\sigma_{\min}\left(G^{-1/2} S G^{-1/2}\right)=\sigma_{\min}\left(\hat{S}\right).
\end{equation*}
\end{proof}

Finally, in the next theorem we prove the inf--sup relation that links the
stability constant of the formulation to the square root of the smallest
eigenvalue of the reduced moment matrix $\hat{S}$.

\begin{theorem}\label{th22p}
The following inf-sup relation holds
\begin{equation*}
    \beta
:=\inf\limits_{\substack{v \in \mathbb{V} \\ \|v\|_{\Omega} = 1}}\inf_{w\in\mathbb{W}} \sup\limits_{\B y(\B\eta,\B\zeta)\neq \B 0}
\frac{\Psi(v+w;\B\eta,\B\zeta)}
     {\left\|\B y(\B\eta,\B\zeta)\right\|_{Y,0}}
=\sqrt{\sigma_{\min}\left(\hat{S}\right)}, 
\end{equation*}
where $\hat{S}$ is the matrix defined in~\eqref{matThat}.   
\end{theorem}

\begin{proof}
For any $\B\xi=\left(\xi_1,\dots,\xi_{\tilde d}\right)^{\top}\in\mathbb{R}^{\tilde d}$ and $\B\gamma=\left(\gamma_0,\dots,\gamma_d\right)^{\top}\in\mathbb{R}^{d+1}$, we consider
\begin{equation*}
  q=v+w\in\mathbb{S}_2\left(S_d\right), \quad v=\sum_{\ell=1}^{\tilde d}\xi_\ell\rho_\ell\in\mathbb{V},\quad
w=\sum_{i=0}^{d}\gamma_i\psi_i\in\mathbb{W}.
\end{equation*}
By combining Lemma~\ref{lem:dual-sup} and Lemma~\ref{lem:block-form}, we get
\begin{eqnarray*}
     \sup_{\B y(\B\eta,\B\zeta)\neq \B 0}
    \frac{\Psi(v+w;\B\eta,\B\zeta)}
         {\left\|\B y(\B\eta,\B\zeta)\right\|_{Y,0}}
    &=&
    \left(
     \begin{pmatrix}\B\xi\\ \B\gamma\end{pmatrix}^{\top}
\begin{pmatrix}K_{11}&K_{12}\\ K_{21}&K_{22}\end{pmatrix}
\begin{pmatrix}\B\xi\\ \B\gamma\end{pmatrix}
    \right)^{1/2}\\ &=&\left(\B\xi^{\top}K_{11}\B\xi
+2\B\xi^{\top}K_{12}\B\gamma
+\B\gamma^{\top}K_{22}\B\gamma\right)^{1/2}.
\end{eqnarray*}
By Lemma~\ref{lem:eliminate}, we have 
\begin{equation*}
\min_{\B\gamma\in\mathbb{R}^{d+1}} 
\left(
\B\xi^{\top}K_{11}\B\xi
+2\B\xi^{\top}K_{12}\B\gamma
+\B\gamma^{\top}K_{22}\B\gamma
\right)
=
\B\xi^{\top}S\B\xi,
\end{equation*}
where $S$ is the matrix defined in \eqref{matS} and the only minimizer is
\begin{equation*}
\bar{\B\gamma}=-K_{22}^{-1}K_{21}\B\xi.
\end{equation*}
Then, we obtain
\begin{eqnarray*}
    \inf_{w\in\mathbb{W}}
\sup_{\B y(\B\eta,\B\zeta)\neq \B 0}
\frac{\Psi(v+w;\B\eta,\B\zeta)}
     {\left\|\B y(\B\eta,\B\zeta)\right\|_{Y,0}}
=\min_{\B\gamma\in\mathbb{R}^{d+1}} 
\left(
\B\xi^{\top}K_{11}\B\xi
+2\B\xi^{\top}K_{12}\B\gamma
+\B\gamma^{\top}K_{22}\B\gamma
\right)^{1/2}=\ 
\sqrt{\B\xi^{\top}S\B\xi}.
\end{eqnarray*}
Enforcing the normalization condition
\begin{equation*}
1=\|v\|_{\Omega}^2=\left\langle v,v\right\rangle_{\Omega}= \sum_{k=1}^{\tilde d}\sum_{\ell=1}^{\tilde d} \xi_\ell \left\langle \rho_\ell,\rho_k\right\rangle_{\Omega} \xi_k= \B\xi^{\top}G\B\xi,   
\end{equation*}
we have
\begin{equation*}
\beta
=\inf\limits_{\substack{v \in \mathbb{V} \\ \|v\|_{\Omega} = 1}}\inf_{w\in\mathbb{W}}
\sup_{\B y(\B\eta,\B\zeta)\neq \B 0}
\frac{\Psi(v+w;\B\eta,\B\zeta)}
     {\left\|\B y(\B\eta,\B\zeta)\right\|_{Y,0}}
=\inf_{\B\xi^{\top}G\B\xi=1}\ 
\sqrt{\B\xi^{\top}S\B\xi}.
\end{equation*}
By Lemma~\ref{lem:rayleigh}, we get
\begin{equation*}
\inf_{\B\xi^{\top}G\B\xi=1}\B\xi^{\top}S\B\xi
=\sigma_{\min}\left(\hat{S}\right).
\end{equation*}
Then, we conclude that
\begin{equation}\label{betaShat}
\beta=\sqrt{\sigma_{\min}\left(\hat{S}\right)}.
\end{equation}
In particular, $\beta>0$ if and only if $\hat{S}$ is positive definite. 
\end{proof}

Eigenvalue-based characterizations of stability play a fundamental role in
numerical linear algebra, especially in the analysis of structured operators and
Schur-type reductions~\cite{calvetti2002restarted,alkilayh2023method,reichel1992eigenvalues}.
In the present setting this viewpoint is particularly appropriate, since the
inf--sup constant can be expressed as the square root of the smallest eigenvalue
of a reduced moment matrix. Related spectral techniques have been developed in the context of 
Toeplitz systems~\cite{lemmerling2000fast,noschese2009structured}, while spectral
approximations of moment quantities based on Lanczos bidiagonalization have been
investigated in~\cite{calvetti1999estimation}. The moment matrices considered in
this work, however, do not have a Toeplitz structure, and therefore the
connection with these contributions is conceptual rather than structural.
The works cited above illustrate how specific matrix structures constrain the
spectral behaviour of an operator, and how moment-type quantities can be analyzed
through spectral tools such as Lanczos-based quadrature formulas. Although our
matrices exhibit a different internal organization, these approaches share with
ours the overarching idea that spectral structure is the key to understanding stability and conditioning.

\begin{remark}
The decomposition $\mathbb{S}_2\left(S_d\right)=\mathbb{V}\oplus \mathbb{W}$ introduced in this
section is not unique. In principle, other choices of the subspaces $\mathbb{V}$
and $\mathbb{W}$ could also be considered. Such alternatives might simplify the
Gram structure $\left(G,C, \tilde{C},M\right)$ and potentially improve the conditioning of the
reduced matrix $\hat{S}$. A systematic analysis of optimal or minimax
decompositions, however, lies beyond the scope of the present work and is left
for future investigation.
\end{remark}

\begin{remark} 
The stability analysis in this section shows that the discrete inf--sup constant is determined by the square root of the smallest eigenvalue of the reduced operator $\hat{S}$. In certain highly anisotropic configurations of the density $\Omega$, this operator may become numerically close to singular, with a spectrum displaying a pronounced flattening. Such behaviour is classical for structured moment matrices and is closely related to phenomena observed in inverse problems and ill-conditioned linear systems, see e.g.~\cite{calvetti1999estimation,noschese2009structured}. In these contexts, mild spectral stabilization strategies of the form 
\begin{equation*}
    S_\alpha = S + \alpha I, \quad \alpha>0,
\end{equation*}
are routinely employed to counteract extreme compression of eigenvalues while preserving the essential structure of the operator. Although the development of such regularized variants lies outside the scope of the present work, the connection suggests a natural direction for future extensions and places the present framework within the broader landscape of numerical linear algebra techniques for structured operators. 
\end{remark}

\begin{proposition}\label{prop:unisolvence-stability}
The following statements hold:
\begin{enumerate}
    \item[$(i)$] The set $$\Sigma = \left\{ \mathcal{I}_j, \mathcal{L}_j, \mathcal{V}_k \, :\, j=0,\dots,d,\ k=1,\dots,\tilde{d}\right\}$$ is unisolvent on $\mathbb{P}_2\left(S_d\right)$ if and only if the matrices $A$ and $T$, defined in~\eqref{matA} and~\eqref{matT}, respectively, are both invertible.

    \item[$(ii)$] The associated discrete system satisfies the inf-sup condition
    \begin{equation*}
   \inf\limits_{\substack{v \in \mathbb{V} \\ \|v\|_{\Omega} = 1}}\inf_{w\in\mathbb{W}}
\sup_{\B y(\B\eta,\B\zeta)\neq \B 0}
\frac{\Psi(v+w;\B\eta,\B\zeta)}
     {\left\|\B y(\B\eta,\B\zeta)\right\|_{Y,0}}
    >0, 
    \end{equation*}
   if and only if $\hat{S}$ is symmetric positive definite.

    \item[$(iii)$] In this case, the inf-sup constant admits the explicit expression
    \begin{equation*}
    \beta
    =
    \inf\limits_{\substack{v \in \mathbb{V} \\ \|v\|_{\Omega} = 1}}\inf_{w\in\mathbb{W}}
\sup_{\B y(\B\eta,\B\zeta)\neq \B 0}
\frac{\Psi(v+w;\B\eta,\B\zeta)}
     {\left\|\B y(\B\eta,\B\zeta)\right\|_{Y,0}}
    = \sqrt{\sigma_{\min}\left(\hat{S}\right)} > 0.
    \end{equation*}
\end{enumerate}
In particular, unisolvence of $\Sigma$ is characterized by the invertibility of $A$ and $T$, whereas quantitative inf-sup stability is characterized by the positive definiteness of $\hat{S}$.
\end{proposition}

\begin{proof}
The first claim follows from Theorem~\ref{thm:CNS-nonsym}. 

By Theorem~\ref{th22p}, the inf--sup constant satisfies
\begin{equation*}
\beta
=
\inf\limits_{\substack{v \in \mathbb{V} \\ \|v\|_{\Omega} = 1}}\inf_{w\in\mathbb{W}}
\sup_{\B y(\B\eta,\B\zeta)\neq \B 0}
\frac{\Psi(v+w;\B\eta,\B\zeta)}
     {\left\|\B y(\B\eta,\B\zeta)\right\|_{Y,0}}
= \sqrt{\sigma_{\min}\left(\hat{S}\right)},
\end{equation*}
where $\hat{S}$ is defined in~\eqref{matThat}. In particular, $\beta>0$ if and only if $\hat{S}$ is symmetric positive definite. This proves $(ii)$ and $(iii)$.
\end{proof}

\begin{remark}
By Theorem~\ref{thm:CNS-nonsym}, the moment system $\Sigma$ is unisolvent on
$\mathbb{P}_2\left(S_d\right)$ if and only if both $A$ and $T$ are invertible (equivalently, $A$ and $H$ are invertible; see Corollary~\ref{cor1}). In contrast,
the quantitative inf--sup stability of the method does not depend
directly on $A$ and $T$, but instead on the symmetric matrix $\hat{S}$
introduced in~\eqref{matThat}. In fact, Theorem~\ref{th22p} yields
\begin{equation*}
\beta
=
\inf\limits_{\substack{v \in \mathbb{V} \\ \|v\|_{\Omega}=1}}
\inf_{w\in\mathbb{W}}
\sup_{\B y(\B\eta,\B\zeta)\neq \B 0}
\frac{\Psi(v+w;\B\eta,\B\zeta)}
     {\left\|\B y(\B\eta,\B\zeta)\right\|_{Y,0}}
=
\sqrt{\sigma_{\min}\left(\hat{S}\right)}.
\end{equation*}
Moreover, since
\begin{equation*}
\hat{S}=G^{-1/2} S G^{-1/2},
\end{equation*}
the matrices $S$ and $\hat{S}$ are congruent, and therefore have the same
inertia. As a consequence, $\hat{S}$ is symmetric positive definite if and only
if $S$ is symmetric positive definite.

In summary:
\begin{itemize}
\item unisolvence $\Longleftrightarrow$ $A$ and $T$ are invertible;
\item uniform inf--sup stability $\Longleftrightarrow S$ is symmetric positive definite.
\end{itemize}
\end{remark}

\subsection{Examples satisfying the face-separation condition}
We now present three representative classes of configurations for which the face
functionals $\mathcal{L}_0,\dots,\mathcal{L}_d$ are linearly independent on a suitable
subspace $\mathbb{W}\subset \mathbb{S}_2\left(S_d\right)$; equivalently, these are configurations
for which the face matrix $M$ defined in~\eqref{matM} is nonsingular. 
These examples illustrate the face–separation assumption employed in the structural 
theorem and demonstrate that it is fully consistent with the unisolvence framework 
of Theorem~\ref{thm:CNS-nonsym}.

\begin{example}
\label{ex:constweight}
Let $S_2$ be a nondegenerate triangle with barycentric coordinates
$\lambda_i$, $i=0,1,2$, and let $\Omega\in L^1\left(S_2\right)$ denote the constant
probability density on $S_2$. Since 
\begin{equation*}
\dim\left(\mathbb{P}_2\left(S_2\right)\right)=6, \quad   \dim\left(\mathbb{P}_1\left(S_2\right)\right)=3,  
\end{equation*}
the orthogonal complement $\mathbb{S}_2\left(S_2\right)$ of $\mathbb{P}_1\left(S_2\right)$
within $\mathbb{P}_2\left(S_2\right)$ has dimension $3$. We introduce the functions
\begin{equation*}
    g_0=\lambda_1\lambda_2,  \quad g_1=\lambda_2\lambda_0, \quad g_2=\lambda_0\lambda_1,
\end{equation*}
which span a three-dimensional subspace of $\mathbb{P}_2\left(S_2\right)$. Since
\begin{equation*}
    \operatorname{span}\left\{g_0,g_1,g_2\right\} \cap \mathbb{P}_1\left(S_2\right) = \{0\},
\end{equation*}
the orthogonal projection $I-\Pi_{1,\Omega}$ is injective on 
$\operatorname{span}\left\{g_0,g_1,g_2\right\}$, where $\Pi_{1,\Omega}$ is defined in~\eqref{projOmega}. We define
\begin{equation}
\psi_i=\left(I-\Pi_{1,\Omega}\right)\left(g_i\right)\in\mathbb{S}_2\left(S_2\right), \quad i=0,1,2.
\end{equation}
 Then, the set $\left\{\psi_i\, :\, i=0,1,2\right\}$ is linearly
independent, forming a basis of $\mathbb{S}_2\left(S_2\right)$.

Adopting the cyclic convention
\begin{equation*}
    \B v_3=\B v_0, \quad \B v_4=\B v_1,
\end{equation*}
we introduce the linear parametrization
\begin{equation*}
    \gamma_j: t\in[0,1]\to t\B v_{j+1}+(1-t)\B v_{j+2}\in F_j, \quad j=0,1,2. 
\end{equation*}
Let $\mu_{j,0}$ be a barycentric coordinate along $F_j$, and define
\begin{equation*}
    q_j(\B x)=6\mu_{j,0}^2(\B x)-6\mu_{j,0}(\B x)+1.
\end{equation*}
Since $\gamma_j$ is affine, the change of variables formula gives,
for any integrable function $\phi$ on $F_j$,
\begin{equation*}
\int_{F_j}\phi(\B x)\omega_j(\B x)d\B x
=\int_0^1 \phi\left(\gamma_j(t)\right) dt. 
\end{equation*}
In particular, a short computation shows that
\begin{equation*}
  \int_{F_j}q_j(\B x)\omega_j(\B x)d\B x=  \int_0^1 q_j\left(\gamma_j(t)\right)dt=\int_0^1 \left(6\mu_{j,0}^2\left(\gamma_j(t)\right)-6\mu_{j,0}\left(\gamma_j(t)\right)+1\right)dt=0
\end{equation*}
and, for $i=0,1$,
\begin{equation*}
  \int_{F_j}\mu_{j,i}(\B x)q_j(\B x)\omega_j(\B x)d\B x=  \int_0^1 \mu_{j,i}\left(\gamma_j(t)\right)q_j\left(\gamma_j(t)\right)dt=0,
\end{equation*}
showing that $q_j$ is orthogonal to $\mathbb{P}_1\left(F_j\right)$ with respect to
$\langle\cdot,\cdot\rangle_{\omega_j}$.
The corresponding face functionals are
\begin{equation*}
\mathcal{L}_j(f)
   =\int_{F_j}  f_{_{\mkern 1mu \vrule height 2ex\mkern2mu F_j}}(\B x)q_j(\B{x})\omega_j(\B x)d\B x,
   \quad j=0,1,2.    
\end{equation*}
Since $\left(\Pi_{1,\Omega}(f)\right)_{_{\mkern 1mu \vrule height 2ex\mkern2mu F_j}}\in\mathbb{P}_1\left(F_j\right)$ and $q_j\perp\mathbb{P}_1\left(F_j\right)$ with respect to $\left\langle \cdot,\cdot \right\rangle_{\omega_j}$, we have 
\begin{equation*}
\mathcal{L}_j\left(\left(I-\Pi_{1,\Omega}\right)(f)\right)=\mathcal{L}_j(f),   
\end{equation*}
so the projection onto $\mathbb{S}_2\left(S_2\right)$ does not modify the face moments. Evaluating
$\mathcal{L}_j$ on the basis functions $\psi_i$ gives
\begin{equation*}
M=\left[\mathcal{L}_j\left(\psi_{i}\right)\right]_{ji}=\left[\mathcal{L}_j\left(g_i\right)\right]_{ji}\in\mathbb{R}^{(d+1)\times(d+1)}.    
\end{equation*}
To compute the entries of $M$, note that on $F_0$, we get 
\begin{equation*}
    g_{1_{\mkern 1mu \vrule height 2ex\mkern2mu F_0}}=g_{2_{\mkern 1mu \vrule height 2ex\mkern2mu F_0}}=0,
\end{equation*}
while 
\begin{equation*}
g_0\left(\gamma_0(t)\right)=\lambda_1\left(\gamma_0(t)\right)\lambda_2\left(\gamma_0(t)\right)=t(1-t)=t-t^2.
\end{equation*}
Then, we obtain
\begin{equation*}
\mathcal{L}_0\left(\psi_0\right)=\mathcal{L}_0\left(g_0\right)
   =\int_0^1 \left(\lambda_1\left(\gamma_0(t)\right)\lambda_2\left(\gamma_0(t)\right)\right)q_0\left(\gamma_0(t)\right)dt=-\frac{1}{30}. 
\end{equation*}
By cyclic permutation $$\mathcal{L}_1\left(\psi_1\right)=\mathcal{L}_2\left(\psi_2\right)=-\frac{1}{30},$$
while all off-diagonal
entries vanish. 
Hence,
\begin{equation*}
M=-
\frac{1}{30}I_3,
\quad
\det(M)\neq 0.    
\end{equation*}
In particular, the face functionals $\mathcal{L}_0,\mathcal{L}_1,\mathcal{L}_2$ are
linearly independent on $\mathbb{S}_2\left(S_2\right)$, and therefore on the face space
$\mathbb{W}=\operatorname{span}\left\{\psi_0,\psi_1,\psi_2\right\}$.
\end{example}

\begin{example}
\label{ex:d3}
Let $S_3$ be a nondegenerate tetrahedron with barycentric coordinates
$\lambda_i$, $i=0,\dots,3$, and let $\Omega\in L^1\left(S_3\right)$ be a strictly positive
probability density on $S_3$. In analogy with Example~\ref{ex:constweight}, we introduce
the quadratic functions
\begin{equation*}
g_0=\lambda_1\lambda_2, \quad
g_1=\lambda_2\lambda_3, \quad
g_2=\lambda_3\lambda_0, \quad
g_3=\lambda_0\lambda_1,
\end{equation*}
and define
\begin{equation*}
    \psi_i=\left(I-\Pi_{1,\Omega}\right)\left(g_i\right), \quad i=0,\dots,3.
\end{equation*}
Since 
\begin{equation*}
    \operatorname{span}\left\{g_0,g_1,g_2,g_3\right\} \cap \mathbb{P}_1\left(S_3\right) = \{0\},
\end{equation*}
the projector $I-\Pi_{1,\Omega}$ is injective on this span. Consequently,
$\psi_0,\dots,\psi_3$ are linearly independent, and we set
\begin{equation*}
\mathbb{W}=\operatorname{span}\left\{\psi_0,\psi_1,\psi_2,\psi_3\right\}\subset\mathbb{S}_2\left(S_3\right).
\end{equation*}
For each $j=0,\dots,3$, let $\omega_j$ denote the normalized restriction of $\Omega$ to
$F_j$, and choose a quadratic polynomial $q_j\in\mathbb{P}_2\left(F_j\right)$ satisfying:
\begin{itemize}
  \item[$(i)$] $q_j\perp \mathbb{P}_1\left(F_j\right)$ with respect to $\langle\cdot,\cdot\rangle_{\omega_j}$;
  \item[$(ii)$] Among the two functions in $\left\{g_0,g_1,g_2,g_3\right\}$ that do not vanish on $F_j$,
  denoted by $\bar{g}_{j,1}$ and $\bar{g}_{j,2}$, the polynomial $q_j$ satisfies
  \begin{equation*}
      \left\langle \bar{g}_{{j,1}_{\mkern 1mu \vrule height 2ex\mkern2mu F_j}}, q_j\right\rangle_{\omega_j} = 0,
      \quad
      \left\langle \bar{g}_{{j,2}_{\mkern 1mu \vrule height 2ex\mkern2mu F_j}}, q_j\right\rangle_{\omega_j} \neq 0.
  \end{equation*}
\end{itemize}
We show that such a choice is always possible. Consider the linear functionals
\begin{equation*}
\chi_{j,1},\chi_{j,2}:\mathbb{S}_2\left(F_j\right)\to\mathbb{R}, \quad
 \chi_{j,m}(q)=\left\langle \bar{g}_{{j,m}_{\mkern 1mu \vrule height 2ex\mkern2mu F_j}},q\right\rangle_{\omega_j},\quad m=1,2.
\end{equation*}
By construction, these functionals are nonzero and linearly independent in the dual space $\left(\mathbb{S}_2\left(F_j\right)\right)^{\star}$. Consequently, the kernels $\ker\left(\chi_{j,1}\right)$ and $\ker\left(\chi_{j,2}\right)$ are distinct 
subspaces of $\mathbb{S}_2\left(F_j\right)=\mathbb{P}_2\left(F_j\right)\ominus\mathbb{P}_1\left(F_j\right)$, and we may select
\begin{equation*}
q_j\in\ker\left(\chi_{j,1}\right)\setminus\ker\left(\chi_{j,2}\right),
\end{equation*}
which satisfies~$(ii)$. Since $\ker(\chi_{j,m})\subset \mathbb{S}_2\left(F_j\right)$, $m=1,2,$ condition~$(i)$ holds automatically.

The corresponding face matrix is
\begin{equation*}
M=\left[\mathcal{L}_j\left(\psi_{i}\right)\right]_{ji}=\left[\mathcal{L}_j\left(g_i\right)\right]_{ji}\in \mathbb{R}^{(d+1)\times(d+1)}.
\end{equation*}
For each $j$, precisely one of the restrictions
$\bar g_{j,1}$, $\bar g_{j,2}$
gives a nonvanishing inner product with $q_j$, while all other terms vanish.
Thus each row of $M$ contains exactly one nonzero entry. After a suitable reordering of
columns, $M$ is diagonal with nonzero diagonal entries, and therefore
\begin{equation*}
\det(M)\neq0.
\end{equation*}
\end{example}

\begin{example}
\label{ex:generic}
Let $S_d\subset\mathbb{R}^d$ be a nondegenerate simplex with barycentric coordinates
$\lambda_i$, $i=0,\dots,d$, and let 
\begin{equation*}
    \Omega(\B{x}) = \sum_{|\alpha|\le m} c_{\alpha}\lambda^{\alpha}(\B x), 
\quad c_{\alpha}>0, \quad m\in\mathbb{N}, \quad \lambda=\left(\lambda_0,\dots,\lambda_d\right),
\end{equation*}
be a smooth, strictly positive probability density on $S_d$. For any $j=0,\dots,d$, we fix an arbitrary nonzero polynomial $q_j\in \mathbb{S}_2\left(F_j\right)$.
From~\eqref{dimS2}, we have
\begin{equation*}
\dim\left(\mathbb{S}_2\left(S_d\right)\right)=\frac{d(d+1)}{2},
\end{equation*}
so that, for $d\ge 2$, it is possible to select $d+1$ linearly independent 
functions $\psi_0,\dots,\psi_d\in\mathbb{S}_2\left(S_d\right)$.
We then define
\begin{equation*}
\mathbb{W}=\operatorname{span}\left\{\psi_0,\dots,\psi_d\right\}\subset \mathbb{S}_2\left(S_d\right),
\end{equation*}
a $(d+1)$-dimensional subspace. The corresponding face matrix is
\begin{equation*}
M=\left[\mathcal{L}_j\left(\psi_{i}\right)\right]_{ji}\in \mathbb{R}^{(d+1)\times(d+1)}.
\end{equation*}
Each entry of $M$ depends \emph{real-analytically} on the geometric and algebraic parameters
of the configuration, namely:
\begin{enumerate} 
\item[$i)$] the coefficients of $\psi_i$ expressed in a fixed basis of $\mathbb{S}_2\left(S_d\right)$;
  \item[$ii)$] the coefficients defining the weight function $\Omega$;
  \item[$iii)$] the coefficients of the face polynomials $q_j$ expressed in a fixed basis of $\mathbb{S}_2\left(F_j\right)$.
\end{enumerate}
Consequently, the determinant
\begin{equation*}
\det\left(M\left(\left\{\psi_{i}\right\}_{i=0}^d, \left\{c_{\alpha}\right\}_{\left\lvert \alpha\right\rvert\le m},\left\{q_j\right\}_{j=0}^d\right)\right)
= F\left(\left\{\psi_{i}\right\}_{i=0}^d, \left\{c_{\alpha}\right\}_{\left\lvert \alpha\right\rvert\le m},\left\{q_j\right\}_{j=0}^d\right)
\end{equation*}
is a real-analytic function of these finitely many parameters. If at least one admissible configuration
\begin{equation*}
    \left(\left\{\psi_{i}\right\}_{i=0}^d, \left\{c_{\alpha}\right\}_{\left\lvert \alpha\right\rvert\le m},\left\{q_j\right\}_{j=0}^d\right)
\end{equation*}
satisfies $\det(M)\neq 0$, then $F$ is not identically zero on the corresponding parameter space. It follows that the zero set of $F$ is a proper real-analytic subset of that space,
with empty interior and Lebesgue measure zero.

By Examples~\ref{ex:constweight} and~\ref{ex:d3}, explicit nondegenerate configurations exist
for $d=2$ and $d=3$. For higher dimensions $d>3$, one can proceed analogously to
Example~\ref{ex:constweight}, taking a constant weight and suitably chosen quadratic
face polynomials, thereby obtaining configurations for which $\det(M)\neq0$.
Since $\det(M)$ depends real-analytically on the algebraic parameters for any $d$,
the existence of at least one nondegenerate configuration in each dimension implies that
$F$ is not identically zero. Hence, the face-separation condition is
\emph{generically satisfied} for all $d\ge2$.
\end{example}

\begin{remark}\label{cor:canonical-infsup}
Under the assumption that $M$ is invertible, the bases
$\left\{\rho_k\right\}_{k=1}^{\tilde d}$ of $\mathbb{V}$ and
$\left\{\psi_j\right\}_{j=0}^{d}$ of $\mathbb{W}$ can be chosen so that
\begin{equation*}
G = I_{\tilde d},
\quad
M = I_{d+1},
\quad
C = 0.
\end{equation*}
In this normalized configuration, the Schur complement becomes
\begin{equation*}
T = M - \tilde{C} G^{-1} C=I_{d+1},
\end{equation*}
and the block matrices $K_{11}$, $K_{12}$, $K_{21}$, $K_{22}$ 
in~\eqref{matKij} yield
\begin{equation*}
S = K_{11} - K_{12} K_{22}^{-1} K_{21} = I_{\tilde{d}}.
\end{equation*}
Hence, by definition~\eqref{matThat}, we have
\begin{equation*}
\hat{S} = G^{-1/2} S G^{-1/2} = I_{\tilde{d}}.
\end{equation*}
Therefore, by Theorem~\ref{th22p},
the inf--sup constant satisfies
\begin{equation*}
\beta 
= \sqrt{\sigma_{\min}\left(\hat{S}\right)}
= \sqrt{\sigma_{\min}\left(I_{\tilde{d}}\right)}
= 1.
\end{equation*}
Thus, in this orthonormal and face-separated configuration,
the quadratic histopolation problem is not only unisolvent but also
\emph{optimally stable}, in the sense that the discrete inf--sup constant
attains the value $\beta = 1$.
\end{remark}

\subsection{Affine-invariant geometric robustness}
The spectral characterization of the local inf--sup constant in terms of
the reduced Schur complement shows that stability may, in principle, depend
on the shape and aspect ratio of each simplex. 
In the practically relevant situation where all elements are obtained by
affine mappings from a single reference simplex and both weights and test
functions are transported accordingly, the following result shows that the
local moment matrices, the reduced Schur complements, and hence the local
inf--sup constants are in fact invariant under such geometric transformations.

Consequently, for any simplicial triangulation in the affine-mapping and barycentric pullback setting, the computation of $\beta$ (and any spectral quantity used for parameter tuning) reduces to a single eigenvalue problem on the reference simplex and can be reused across all elements.

\begin{theorem}\label{thm:affine-robust}
Let $\mathcal{T}_h$ be a simplicial mesh such that each element 
$K \in \mathcal{T}_h$ is the image of a fixed reference simplex 
$\hat K$ under an affine bijection $\Phi_K : \hat K \to K$.
For each $j=0,\dots,d$, we set $F_j:=\Phi_K\left(\hat F_j\right)$ and denote by
\begin{equation*}
    \Phi_{K,j}=\Phi_K{_{\mkern 1mu \vrule height 2ex\mkern2mu \hat{F}_j}},
\end{equation*}
the restriction of $\Phi_K$ to $\hat F_j$. 
Assume that the following conditions hold:
\begin{itemize}
  \item for any $f\in L^1(K)$,
  \begin{equation}\label{eq:bulk-pullback}
  \int_K f(\B \sigma)\Omega_K(\B \sigma)d\B \sigma
  = \int_{\hat K} f\left(\Phi_K\left(\B {\hat\sigma}\right)\right)\hat\Omega\left(\B{\hat\sigma}\right)d\B{\hat\sigma};
  \end{equation}
  \item for each face $\hat F_j\subset\hat K$ and any $g\in L^1\left(F_j\right)$,
  \begin{equation}\label{eq:face-pullback}
  \int_{F_j} g(\B x)\omega_{K,j}(\B x)d \B x
  = \int_{\hat F_j} g\left(\Phi_{K,j}\left(\B{\hat x}\right)\right)\hat\omega_j\left(\B{\hat x}\right)d\B{\hat x};
  \end{equation}
  \item  for each $K$, the interior and face test functions satisfy
\begin{equation}\label{sasaccb}
      \rho_{K,k}(\B \sigma)= \hat\rho_k\left(\Phi_{K}^{-1}(\B \sigma)\right), 
\quad \B \sigma\in K,\quad k=1,\dots,\tilde d,
\end{equation}
and
\begin{equation}\label{sasaccb1}
     \psi_{K,j}(\B \sigma)= \hat\psi_j\left(\Phi_{K}^{-1}(\B \sigma)\right), 
  \quad  \B \sigma\in K, \quad j=0,\dots, d.
\end{equation}
Moreover, the polynomials $q_{K,j}$ satisfy
\begin{equation}\label{sasaccb3}
  q_{K,j}(\B x)= \hat q_j\left(\Phi_{K,j}^{-1}(\B x)\right), 
  \quad \B x\in F_j,  \quad j=0,\dots, d;
\end{equation}
\item The reduced stability matrix $\hat S$ on $\hat K$
is symmetric positive definite.
\end{itemize}
Then the associated local moment matrices
$\left(G_K,C_K,\tilde C_K,M_K\right)$ 
coincide, for all $K\in\mathcal{T}_h$, with the corresponding matrices
$\left(G,C,\tilde C,M\right)$ on $\hat K$.
In particular, the reduced stability matrices $\hat S_K$ and the local inf--sup
constants $\beta_K$ are independent of $K$ and of $h$, and satisfy
$\beta_K=\beta > 0$.
\end{theorem}

\begin{proof}
For each $K\in\mathcal{T}_h$, let $\Phi_K : \hat K \to K$ be the affine map
such that
\begin{equation}\label{facefs}
    \Phi_K\left(\B{\hat{v}}_j\right)=\B{v}_j,\quad j=0,\dots,d,
\end{equation}
where $\B{\hat{v}}_j$, $j=0,\dots,d$, are the vertices of $\hat{K}$. 
Recall that~\eqref{eq:bulk-pullback} and~\eqref{eq:face-pullback} are the change of variables identities under the affine map $\Phi_K$ (and its restriction on faces). In particular, $\hat\Omega$ (and $\hat\omega_j$ on $\hat F_j$) denotes the pullback density on the reference simplex, including the appropriate Jacobian factor associated with $\Phi_K$.

Using~\eqref{eq:bulk-pullback}, the entries of the interior Gram matrix $G_K$
can be written as
\begin{eqnarray*}
    \left[ G_K\right]_{k\ell}
 &=& \int_K \rho_{K,k}(\B \sigma)\rho_{K,\ell}(\B \sigma)\Omega_K(\B \sigma)d\B \sigma
 \\ &=& \int_{\hat K} \rho_{K,k}\left(\Phi_K\left(\B{\hat{ \sigma}}\right)\right)\rho_{K,\ell}\left(\Phi_K\left(\B{\hat{ \sigma}}\right)\right)\hat{\Omega}\left(\B{\hat\sigma}\right)d\B{\hat\sigma}\\ &=& \int_{\hat K} \hat\rho_k\left(\B{\hat\sigma}\right)
      \hat\rho_\ell\left(\B{\hat\sigma}\right)
\hat\Omega\left(\B{\hat\sigma}\right)d\B{\hat\sigma},
\end{eqnarray*}
where in the last equality, we have used~\eqref{sasaccb}. 
The right-hand side does not depend on $K$ and defines the reference matrix
$[G]_{k\ell}$ on $\hat K$. Thus $G_K = G$ for all
$K\in\mathcal{T}_h$.

Analogously, using~\eqref{eq:face-pullback}, \eqref{sasaccb}, \eqref{sasaccb1} and~\eqref{sasaccb3}, it is possible to  show that, for any $K\in\mathcal{T}_h$, the coupling and face blocks satisfy
\begin{equation*}
C_K = C,\quad \tilde C_K = \tilde C,\quad M_K = M. 
\end{equation*}
The reduced stability matrix $\hat S_K$ is obtained from the blocks
$\left(G_K,C_K,\tilde C_K,M_K\right)$ by a fixed sequence of algebraic
operations. Since all these blocks coincide with their reference counterparts,
we obtain
\begin{equation*}
      \hat S_K = \hat S,
    \quad \forall K\in\mathcal{T}_h.
\end{equation*}
By the general inf--sup result relating the local stability constant
$\beta_K$ to the smallest eigenvalue of $\hat S_K$, we have
\begin{equation*}
     \beta_K^2 = \sigma_{\min}\left(\hat S_K\right)
              = \sigma_{\min}\left(\hat S\right)
              = \beta^2.
\end{equation*}
Since $\hat S$ is symmetric positive definite by assumption, $\beta>0$ and
therefore
\begin{equation*}
    \beta_K = \beta > 0.
\end{equation*}
\end{proof}

\section{Optimal choice of the face and interior quadratic polynomials}
\label{sec3}
The face moment functionals
\begin{equation*}
\mathcal{L}_j(f)
= \int_{F_j} f_{_{\mkern 1mu \vrule height 2ex\mkern2mu F_j}}(\B x)
q_j(\B x)\, \omega_j(\B x)\, d\B x
= \left\langle f_{_{\mkern 1mu \vrule height 2ex\mkern2mu F_j}}, q_j \right\rangle_{\omega_j},
\quad j = 0,\dots,d,
\end{equation*}
determine the face block
\begin{equation*}
M = \left[ \mathcal{L}_j\left(\psi_{i}\right) \right]_{ji}
\in \mathbb{R}^{(d+1)\times(d+1)},
\end{equation*}
where $\psi_0,\dots,\psi_d$ is the chosen basis of the face space $\mathbb{W}$. Each polynomial $q_j \in \mathbb{P}_2\left(F_j\right)$ is required to satisfy
\begin{equation*}
\left\langle q_j, \varphi \right\rangle_{\omega_j} = 0
\quad \forall \varphi \in \mathbb{P}_1\left(F_j\right).
\end{equation*}
This condition does not uniquely determine $q_j$, and the particular choice
of these polynomials affects the spectral conditioning of the face block $M$
and, consequently, of the global matrix $H$.

The next theorem provides a spectrally motivated criterion for selecting
the family  
$\left\{q_j\right\}_{j=0}^d$, aimed at improving the conditioning of $M$
while preserving unisolvence. To this end, for each face $F_j$ we restrict
$q_j$ to the orthogonal complement
\begin{equation}\label{alsaas}
\mathbb{S}_2\left(F_j\right)
= \mathbb{P}_2\left(F_j\right)\ominus\mathbb{P}_1\left(F_j\right).
\end{equation}

We next recall a basic spectral property of Gram matrices that will be useful
for our purposes.
\begin{lemma}
\label{lem:loewner}
Let $G$ be the Gram matrix associated with a linearly independent family $\left\{g_s\right\}_{s=1}^n \subset L^2_{\Omega}\left(S_d\right)$, normalized so that
\begin{equation*}
\left\|g_s\right\|_{\Omega} = 1, \quad s=1,\dots,n.    
\end{equation*}
Then $G$ is symmetric positive definite and
\begin{equation}
\sigma_{\min}(G) \le 1 \le \sigma_{\max}(G).
\end{equation}
Consequently
\begin{equation*}
    \kappa_2(G) = \frac{\sigma_{\max}(G)}{\sigma_{\min}(G)} \ge 1,
\end{equation*}
with equality if and only if $\left\{g_s\right\}_{s=1}^n$ is orthonormal, or equivalently $G = I_n$.
\end{lemma}

\begin{proof}
For any $\B x=\left(x_1,\dots,x_n\right)^\top\in\mathbb{R}^n$, we have
\begin{equation*}
     \B x^\top G \B x=\sum_{m,s=1}^n x_m x_s\left\langle g_m,g_s\right\rangle_{\Omega}
  =\left\|\sum_{s=1}^n x_s g_s\right\|_{\Omega}^2\ge 0,
\end{equation*}
with strict inequality for $\B x\neq \B 0$ due to the linear 
independence of the functions $g_s$, $s=1,\dots,n$. Hence $G$ is symmetric 
positive definite. Let
\begin{equation*}
G=Q\Lambda Q^\top, \quad    \Lambda=\mathrm{diag}\left(\sigma_1,\dots,\sigma_n\right), 
\end{equation*}
be its spectral decomposition, where
\begin{equation*}
    0<\sigma_{\min}(G)\le \sigma_i\le \sigma_{\max}(G), \quad i=1,\dots,n.
\end{equation*}
Since $\left\|g_s\right\|_{\Omega}^2 = \left\langle g_s, g_s\right\rangle_{\Omega} = 1$, $s=1,\dots,n$,
all diagonal entries of $G$ are equal to $1$. Therefore
\begin{equation*}
     \operatorname{trace}(G)=\sum_{s=1}^n \left\langle g_s,g_s\right\rangle=n=\sum_{s=1}^n \sigma_s.
\end{equation*}
In particular, the average eigenvalue is $1$, which implies
\begin{equation*}
\sigma_{\min}(G)\le 1\le \sigma_{\max}(G).    
\end{equation*}
Since the spectral condition number satisfies  
\begin{equation*}
    \kappa_2(G)=\frac{\sigma_{\max}(G)}{\sigma_{\min}(G)}\ge 1,
\end{equation*}
equality occurs if and only if all eigenvalues of $G$ are equal to $1$, that is,
$\Lambda = I_n$ and hence $G = I_n$. In this case, the family $\left\{g_s\right\}_{s=1}^n$ is
orthonormal with respect to $\left\langle \cdot,\cdot\right\rangle_{\Omega}$.
\end{proof}

\begin{theorem}
\label{thm:optimal-face}
Let $\left\{\psi_j\right\}_{j=0}^d$ be a basis of $\mathbb{W}$. For each $j\in\{0,\dots,d\}$,
assume that there exists an index $i(j)\in\{0,\dots,d\}$ such that
\begin{equation*}
\left( I - \Pi_{1,\omega_j} \right)
\left( \psi_{i(j)_{_{\mkern 1mu \vrule height 2ex\mkern2mu F_j}}} \right)
\neq 0.
\end{equation*}
We define
\begin{equation}\label{eq:optqi}
q_j^{\star}
=
\arg\max_{\substack{q \in \mathbb{S}_2\left(F_j\right) \\[1mm] \left\| q \right\|_{\omega_j} = 1}}
\left| \left\langle \psi_{i(j)_{_{\mkern 1mu \vrule height 2ex\mkern2mu F_j}}}, q \right\rangle_{\omega_j} \right|.
\end{equation}
Then the maximizer exists, is unique up to sign, and is given by
\begin{equation}\label{eq:opt-face-final}
q_j^{\star}
=
\frac{
\left( I - \Pi_{1,\omega_j} \right)
\left( \psi_{i(j)_{_{\mkern 1mu \vrule height 2ex\mkern2mu F_j}}} \right)
}{
\left\|
\left( I - \Pi_{1,\omega_j} \right)
\left( \psi_{i(j)_{_{\mkern 1mu \vrule height 2ex\mkern2mu F_j}}} \right)
\right\|_{\omega_j}
}.
\end{equation}
In particular, $q_j^{\star}\in\mathbb{S}_2\left(F_j\right)$ and $\left\| q_j^{\star} \right\|_{\omega_j}=1$.
\end{theorem}

\begin{proof}
Let $\left\{\tau_{j,\iota}\right\}_{\iota=1}^{\frac{d(d-1)}{2}}$ be an orthonormal basis 
of $\mathbb{S}_2\left(F_j\right)$ with respect to $\langle \cdot,\cdot\rangle_{\omega_j}$. Any $q\in\mathbb{S}_2\left(F_j\right)$ with $\|q\|_{\omega_j}=1$ can be written as
\begin{equation*}
q = \sum_{\iota=1}^{\frac{d(d-1)}{2}} \zeta_{\iota} \tau_{j,\iota},
\quad
\|\B{\zeta}\|_2 = 1, \quad \B{\zeta}=\left(\zeta_1,\dots,\zeta_{\frac{d(d-1)}{2}}\right)^{\top}.
\end{equation*}
By linearity of the inner product,
\begin{equation}\label{newresa}
\left\langle \psi_{i(j)_{_{\mkern 1mu \vrule height 2ex\mkern2mu F_j}}}, q \right\rangle_{\omega_j}= \sum_{\iota=1}^{\frac{d(d-1)}{2}}\zeta_{\iota}
\left\langle \psi_{i(j)_{_{\mkern 1mu \vrule height 2ex\mkern2mu F_j}}}, \tau_{j,\iota} \right\rangle_{\omega_j}
= \B{\zeta}^{\top}\B{\psi}_{j,i(j)},
\end{equation}
where 
\begin{equation*}
\B{\psi}_{j,i(j)}
= \left(\left\langle \psi_{i(j)_{_{\mkern 1mu \vrule height 2ex\mkern2mu F_j}}}, \tau_{j,1} \right\rangle_{\omega_j}, \dots,
\left\langle \psi_{i(j)_{_{\mkern 1mu \vrule height 2ex\mkern2mu F_j}}}, \tau_{j,\frac{d(d-1)}{2}} \right\rangle_{\omega_j} \right)^{\top}.
\end{equation*}
Since $\|q\|_{\omega_j}=1$ implies $\|\B{\zeta}\|_2=1$, 
by~\eqref{newresa} the maximization problem in~\eqref{eq:optqi} can be written as
\begin{equation*}
\arg\max_{\substack{q \in \mathbb{S}_2\left(F_j\right) \\[1mm] \left\| q \right\|_{\omega_j} = 1}}
\left| \left\langle \psi_{i(j)_{_{\mkern 1mu \vrule height 2ex\mkern2mu F_j}}}, q \right\rangle_{\omega_j} \right|=\arg\max_{\left\|\B{\zeta}_{j}\right\|_2=1} 
\left|\B{\zeta}^{\top}_j\B{\psi}_{j,i(j)}\right|.
\end{equation*}
By the Cauchy--Schwarz inequality, the maximum is attained when
\begin{equation*}
\B{\zeta}_j
=\frac{ \B{\psi}_{j,i(j)}}{\left\|\B{\psi}_{j,i(j)}\right\|_2}.
\end{equation*}
Then the optimal polynomial is
\begin{equation*}
q_j^{\star}=\sum_{\iota=1}^{\frac{d(d-1)}{2}} \zeta_{j,\iota} \tau_{j,\iota}=\sum_{\iota=1}^{\frac{d(d-1)}{2}} \frac{\left\langle \psi_{i(j)_{_{\mkern 1mu \vrule height 2ex\mkern2mu F_j}}}, \tau_{j,\iota} \right\rangle_{\omega_j}}{\left\lVert \B{\psi}_{j,i(j)} \right\rVert_2} \tau_{j,\iota}= \frac{\left(I-\Pi_{1,\omega_j}\right)\left(\psi_{i(j)_{_{\mkern 1mu \vrule height 2ex\mkern2mu F_j}}}\right)}{\left\lVert \B{\psi}_{j,i(j)} \right\rVert_2}.
\end{equation*}
Finally, applying the Parseval identity we have
\begin{equation*}
\left\|\B{\psi}_{j,i(j)}\right\|_2^2
=
\left\|
\left(I-\Pi_{1,\omega_j}\right)\left(\psi_{i(j)_{_{\mkern 1mu \vrule height 2ex\mkern2mu F_j}}}\right)
\right\|_{\omega_j}^2,
\end{equation*}
and then
\begin{equation*}
  q^{\star}_j= \frac{\left(I-\Pi_{1,\omega_j}\right)\left(\psi_{i(j)_{_{\mkern 1mu \vrule height 2ex\mkern2mu F_j}}}\right)}{\left\|
\left(I-\Pi_{1,\omega_j}\right)\left(\psi_{i(j)_{_{\mkern 1mu \vrule height 2ex\mkern2mu F_j}}}\right)
\right\|_{\omega_j}}.
\end{equation*}
\end{proof}

\subsection{Explicit realizations}
We now give explicit expressions for the face quadratic polynomials in the
examples discussed in the previous section.

\begin{itemize}
    \item \textbf{Constant weight.}\\ For the uniform weight $\omega_j(t) = 1$ on a one-dimensional face $t\in[0,1]$,
the unique normalized quadratic polynomial orthogonal to $\mathbb{P}_1([0,1])$ is
\begin{equation*}
    q_j^{\star}(t) = \sqrt{5}\left(6t^2 - 6t + 1\right).
\end{equation*}
This reproduces the choice made in Example~\ref{ex:constweight}.
\item \textbf{Dirichlet weight.}\\ 
For a Dirichlet-type weight
\begin{equation*}
    \omega_j(t) = t^{\alpha_r - 1}\left(1 - t\right)^{\alpha_s - 1},
    \quad \alpha_r,\alpha_s>0,
\end{equation*}
the subspace $\mathbb{P}_2([0,1]) \ominus \mathbb{P}_1([0,1])$ is one-dimensional
and consists of all quadratics orthogonal to $1$ and $t$ with respect to
$\langle\cdot,\cdot\rangle_{\omega_j}$. A canonical generator of this space is
the shifted Jacobi polynomial of degree two, which in $[0,1]$ coordinates reads
\begin{equation*}
    q_j^{\star}(t) = c_j\pi_{J,2}^{(\alpha_r - 1,\alpha_s - 1)}\left(2t - 1\right),
\end{equation*}
where $\pi_{J,2}^{(\alpha_r - 1,\alpha_s - 1)}$ denotes the classical Jacobi polynomial and $c_j>0$ is a normalization constant chosen so that $\left\|q_j^{\star}\right\|_{\omega_j}=1$. By construction, the polynomial $q_j^{\star}$ satisfies
\begin{equation*}
    \int_0^1 q_j^{\star}(t)\omega_j(t)dt = 0,
    \quad
    \int_0^1 tq_j^{\star}(t)\omega_j(t)dt = 0,
\end{equation*}
and it is the unique (up to sign) element of $\mathbb{P}_2([0,1])$ fulfilling
these orthogonality conditions.

In the triangular configuration of Example~\ref{ex:constweight}, the face bubble
functions are
\begin{equation*}
    g_0 = \lambda_1\lambda_2, \quad
    g_1 = \lambda_2\lambda_0, \quad
    g_2 = \lambda_0\lambda_1,
\end{equation*}
and the associated basis functions are 
\begin{equation*}
\psi_i = \left(I - \Pi_{1,\omega_j}\right)\left(g_i\right), \quad i=0,1,2.    
\end{equation*}
On each face $F_j$ we have
\begin{equation*}
    g_{i_{\mkern 1mu \vrule height 2ex\mkern2mu F_j}} =
    \begin{cases}
        0, & i \neq j,\\[4pt]
        \mu_{j,0}\mu_{j,1} \in \mathbb{P}_2\left(F_j\right), & i = j,
    \end{cases}
\end{equation*}
where $\mu_{j,0}$ and $\mu_{j,1}$ denote the local barycentric coordinates on
$F_j$. Consequently,
\begin{equation*}
    \psi_{i_{\mkern 1mu \vrule height 2ex\mkern2mu F_j}}
    = \left(I - \Pi_{1,\omega_j}\right)\left(g_{i_{\mkern 1mu \vrule height 2ex\mkern2mu F_j}}\right)
    =
    \begin{cases}
        0, & i \neq j,\\[4pt]
        \left(I - \Pi_{1,\omega_j}\right)\left(\mu_{j,0}\mu_{j,1}\right) \in \mathbb{S}_2\left(F_j\right), & i = j.
    \end{cases}
\end{equation*}
Since $q_j^{\star}\in\mathbb{S}_2\left(F_j\right)$ is orthogonal to $\mathbb{P}_1\left(F_j\right)$, it follows that
\begin{equation*}
    \left\langle \psi_{i_{\mkern 1mu \vrule height 2ex\mkern2mu F_j}}, q_j^{\star} \right\rangle_{\omega_j}
    = \left\langle g_{i_{\mkern 1mu \vrule height 2ex\mkern2mu F_j}}, q_j^{\star} \right\rangle_{\omega_j},
\end{equation*}
so the face matrix $M$ defined in~\eqref{matM} is diagonal, with entries
\begin{equation*}
    [M]_{jj}=\left\langle \mu_{j,0}\mu_{j,1},q_j^{\star} \right\rangle_{\omega_j}\neq0.
\end{equation*}
\end{itemize}

\subsection{Optimal choice of the interior quadratic polynomials}
\label{subsec:optimal_rho}
In the quadratic histopolation framework introduced in Section~\ref{sec2},
the interior moment functionals
\begin{equation*}
    \mathcal{V}_k(f)
    = \int_{S_d} f(\B\sigma)\rho_k(\B\sigma)\Omega(\B\sigma)d\B\sigma,
    \quad k=1,\dots,\tilde d,
    \label{eq:V-def}
\end{equation*}
are the volume counterparts of the face moments $\mathcal{L}_j(f)$
defined in~\eqref{Lj}.  
While the optimal choice of the face test functions $q_j^{\star}$
improves the conditioning of the face block $M$, 
the selection of the interior quadratics $\rho_\ell$ determines the conditioning of the \emph{volume block}
\begin{equation*}
    G = \left[\mathcal{V}_k\left(\rho_\ell\right)\right]_{k\ell}\in\mathbb{R}^{\tilde{d}\times\tilde{d}},
    \label{eq:G-def}
\end{equation*}
and therefore the spectral stability of the matrix
\begin{equation*}
    H =
    \begin{pmatrix}
        G & C \\[1mm]
        \tilde C & M
    \end{pmatrix}.
\end{equation*}
Our goal is to construct an \emph{optimal family} $\left\{\rho_\ell^{\star}\right\}_{\ell=1}^{\tilde d}$
that minimizes the condition number of $G$, or equivalently maximizes its 
smallest eigenvalue, under the unisolvence assumptions of 
Theorem~\ref{thm:CNS-nonsym}.\\

Each interior polynomial $\rho_\ell$ must be orthogonal, with respect to the 
weighted inner product $\langle\cdot,\cdot\rangle_\Omega$, to all affine functions, namely
\begin{equation*}
   \left\langle \rho_\ell, 1\right\rangle_\Omega = 0,
    \quad
    \left\langle \rho_\ell, \lambda_j\right\rangle_\Omega = 0,
    \quad j=0,\ldots,d, \quad \ell=1,\dots,\tilde{d}.
    \label{eq:rho-orth}
\end{equation*}
Whenever this orthogonalization is possible, we additionally impose that the spaces $\mathbb{V}$ and $\mathbb{W}$ are mutually orthogonal with respect to $\langle\cdot,\cdot\rangle_\Omega$, that is
\begin{equation*}
    \left\langle \rho_\ell, \psi_j \right\rangle_\Omega = 0,
    \quad j=0,\ldots,d, \quad \ell=1,\dots,\tilde{d}.
\end{equation*}
In this setting, the following theorem establishes the spectral optimality of the 
projected interior generators.

\begin{theorem}
\label{thm:spectral-optimality-rho}
Let $\left\{p_\ell\right\}_{\ell=1}^{\tilde d}\subset \mathbb{P}_2\left(S_d\right)$ be such that 
\begin{equation*}
    \mathbb{V}=\operatorname{span}\left\{\left(I-\Pi_{1,\Omega}\right)\left(p_\ell\right)\,:\, \ell=1,\dots,\tilde d\right\}
     \subset \mathbb{S}_2\left(S_d\right).
\end{equation*}
We define the admissible class by
\begin{equation}
    \mathcal{A}
    =
    \left\{
        \left\{\rho_\ell\right\}_{\ell=1}^{\tilde d}\subset \mathbb{S}_2\left(S_d\right)
        \,:\,
        \rho_\ell\in\operatorname{span}\left\{\left(I-\Pi_{1,\Omega}\right)\left(p_\ell\right)\right\},
        \ \left\|\rho_\ell\right\|_{\Omega}=1
    \right\}.
\end{equation}
Consider the normalized family 
\begin{equation}\label{eq:rho-star}
    \rho_\ell^{\star}
    =
    \frac{\left(I-\Pi_{1,\Omega}\right)\left(p_\ell\right)}
         {\left\|\left(I-\Pi_{1,\Omega}\right)\left(p_\ell\right)\right\|_{\Omega}},
    \quad \ell=1,\ldots,\tilde d,
\end{equation}
and set $G^{\star} = \left[\left\langle \rho_\ell^{\star},\rho_k^{\star}\right\rangle_\Omega\right]_{k\ell}$.
Then
\begin{equation}
    \sigma_{\min}\left(G^{\star}\right)
    =
    \max_{\left\{\rho_\ell\right\}_{\ell=1}^{\tilde d}\in\mathcal{A}}
    \sigma_{\min}(G),
    \label{eq:G-optimal}
\end{equation}
that is, the family~\eqref{eq:rho-star} attains the best-conditioned interior Gram block within~$\mathcal{A}$. Moreover, if the projections $\left\{(I-\Pi_{1,\Omega})(p_\ell)\right\}_{\ell=1}^{\tilde d}$ are pairwise orthogonal in $\langle\cdot,\cdot\rangle_\Omega$, then $G^\star=I_{\tilde d}$.
\end{theorem}

\begin{proof}
Let $\left\{\tau_\ell\right\}_{\ell=1}^{\tilde d}$ be an orthonormal basis of $\mathbb{V}$ with respect to 
$\left\langle \cdot,\cdot \right\rangle_{\Omega}$.  
For each $k\in\left\{1,\dots,\tilde d\right\}$, since $
\left(I-\Pi_{1,\Omega}\right)\left(p_k\right)\in\mathbb{V},$ there exist coefficients
$a_{k\ell}\in\mathbb{R}$, $\ell=1,\dots,\tilde{d}$, such that
\begin{equation}\label{sasacc}
   \left(I-\Pi_{1,\Omega}\right)\left(p_k\right) =
\sum_{\ell=1}^{\tilde d} a_{k\ell}\tau_\ell.
\end{equation}
By the spanning assumption,
\begin{equation*}
    \B a_k=\left(a_{k1},\dots,a_{k\tilde d}\right)^{\top}\neq \B 0,
\end{equation*}
and by orthonormality of $\left\{\tau_\ell\right\}_{\ell=1}^{\tilde d}$ we obtain
\begin{equation}\label{sasacc1}
    \left\|\left(I-\Pi_{1,\Omega}\right)\left(p_k\right)\right\|_{\Omega}^2=\left\|\B a_k\right\|_2^2.
\end{equation}
Hence, from \eqref{sasacc} and \eqref{sasacc1}, the normalized functions can be written as
\begin{equation}\label{cscs}
\rho_k^{\star}
=
\frac{\left(I-\Pi_{1,\Omega}\right)\left(p_k\right)}
{\left\|\left(I-\Pi_{1,\Omega}\right)\left(p_k\right)\right\|_{\Omega}}
=
\sum_{\ell=1}^{\tilde d}
\frac{a_{k\ell}}{\left\|\B a_k\right\|_2}\tau_\ell,   \quad k=1,\dots,\tilde{d}. 
\end{equation}
Moreover
\begin{equation}\label{cscs1}
    \rho_k^{\star}\in \operatorname{span}\left\{\left(I-\Pi_{1,\Omega}\right)\left(p_k\right)\right\}, \quad k=1,\dots,\tilde{d},
\end{equation}
hence $\left\{\rho_k^{\star}\right\}_{k=1}^{\tilde d}\in\mathcal{A}$.

Let now $\left\{\rho_k\right\}_{k=1}^{\tilde d}\in\mathcal{A}$.  
Since
\begin{equation}\label{cscs2}
\rho_k\in \operatorname{span}\left\{\left(I-\Pi_{1,\Omega}\right)\left(p_k\right)\right\}\subset \mathbb{V}, \quad k=1,\dots,\tilde{d},    
\end{equation}
we may write
\begin{equation}\label{rhok}
\rho_k=\sum_{\ell=1}^{\tilde d} u_{k\ell}\tau_\ell, 
\end{equation} where \begin{equation*} \B u_k = \left(u_{k1},\dots,u_{k\tilde d}\right)^{\top}. \end{equation*}
Using~\eqref{cscs1} and~\eqref{cscs2}, the normalization 
$\left\|\rho_k\right\|_{\Omega}=1$
implies the existence of $s_k\in\left\{-1,+1\right\}$ such that
\begin{equation}\label{dsdsaac} \B u_k = \left(u_{k1},\dots,u_{k\tilde d}\right)^{\top} = s_k\frac{\B a_k}{\left\|\B a_k\right\|_2}, \quad k=1,\dots,\tilde{d}. \end{equation}

We set \begin{itemize} \item $U\in \mathbb{R}^{\tilde{d}\times\tilde{d}}$ be the matrix whose $k$–th row is $\B u_k^{\top}$; \item $U^{\star} \in \mathbb{R}^{\tilde{d}\times\tilde{d}}$ be the matrix whose $k$–th row is $\left(\B a_k/\left\|\B a_k\right\|_2\right)^\top$; \item $S\in \mathbb{R}^{\tilde{d}\times\tilde{d}} $ be the diagonal matrix with $S_{kk}=s_k$, $k=1,\dots,\tilde{d}$. \end{itemize} 
Then~\eqref{dsdsaac} yields \begin{equation}\label{UUstar} U = S U^{\star}. \end{equation} 
By orthonormality of $\left\{\tau_\ell\right\}_{\ell=1}^{\tilde{d}}$ and~\eqref{rhok},
\begin{equation} \label{langle}\left\langle\rho_\ell,\rho_k\right\rangle_{\Omega}=\B u_\ell^{\top}\B u_k. \end{equation} Combining~\eqref{UUstar} and~\eqref{langle}, we have \begin{equation} G = U U^\top = \left(SU^{\star}\right) \left(SU^{\star}\right)^{\top} = S\left(U^{\star}(U^{\star})^{\top}\right)S = SG^{\star} S. \end{equation} 
 Thus $G$ and $G^{\star}$ are similar, and therefore have the same spectrum.  
In particular, for any $\left\{\rho_k\right\}_{k=1}^{\tilde d}\in\mathcal{A}$, we have
\begin{equation*} \sigma_{\min}\left(G^{\star}\right)=\sigma_{\min}(G), \end{equation*} 
and hence
\begin{equation*}
  \sigma_{\min}\left(G^{\star}\right)=  \max_{\left\{\rho_\ell\right\}_{\ell=1}^{\tilde d}\in\mathcal{A}}
    \sigma_{\min}(G).
\end{equation*}
Finally, if the vectors $\left\{\B a_k\right\}_{k=1}^{\tilde{d}}$ are pairwise orthogonal, then by~\eqref{cscs}, we obtain
\begin{equation*}
    \left\langle\rho_{\ell}^{\star},\rho_{k}^{\star} \right\rangle=\delta_{\ell k}, \quad \ell,k=1,\dots,\tilde d,
\end{equation*}
hence $G^{\star}=I_{\tilde d}$, which is spectrally optimal by Lemma~\ref{lem:loewner}. \end{proof}

\section{Explicit and practical construction}
\label{sec4}

In this section, we provide a closed-form analytic expression of the interior
quadratic polynomials with respect to Dirichlet-type weights.
This concrete construction illustrates the general framework introduced above and,
in particular, includes the case of constant weight $\Omega = 1$ as a special instance.

\subsection{Dirichlet weight}
The Dirichlet weight is defined by
\begin{equation}
\Omega_{\B{\alpha}}(\B{\sigma})
=
C_{\B{\alpha},d}
\prod_{i=0}^d \lambda_i^{\alpha_i-1}(\B{\sigma}),
\quad 
\alpha_i>0,\quad i=0,\dots,d,\quad \B{\sigma}\in S_d, 
\label{eq:dirichlet-weight}
\end{equation}
where 
\begin{equation*}
    C_{\B \alpha,d}=\frac{\Gamma   \left(d+\sum_{j=0}^d \alpha_j\right)}{\left\lvert  S_d\right\rvert d!\prod_{j=0}^d\Gamma\left(\alpha_j\right) }
\end{equation*}
denotes the normalization constant chosen so that
\begin{equation*}
\int_{S_d}\Omega_{\B{\alpha}}(\B{\sigma})d\B{\sigma}=1.    
\end{equation*}

Now we establish the unisolvence result for the quadratic space under Dirichlet weights on the simplex.

\begin{corollary}\label{corDirichlet}
Let $S_d\subset\mathbb{R}^d$ be a nondegenerate $d$-simplex with barycentric coordinates
$\lambda_0,\dots,\lambda_d$ and Dirichlet weight $\Omega_{\B{\alpha}}(\B{\sigma})$. We denote by $\omega_{j,\B{\alpha}}$ the induced Dirichlet weight on the face $F_j$.
Then the set of degrees of freedom
\begin{equation*}
\Sigma=\left\{\mathcal{I}_j,\mathcal{L}_j,\mathcal{V}_k\, :\, j=0,\dots,d, \, k=1,\dots,\tilde{d}\right\}    
\end{equation*}
is unisolvent on $\mathbb{P}_2\left(S_d\right)$ for any choice of Dirichlet parameters $\alpha_i>0$.
\end{corollary}

\begin{proof}
Let $p\in\mathbb P_2\left(S_d\right)$ satisfy
\begin{equation*}
    \mathcal I_j(p)=\mathcal L_j(p)=\mathcal V_k(p)=0,\quad j=0,\dots,d,\ k=1,\dots,\tilde d.
\end{equation*}
By Theorem~\ref{thm:CNS-nonsym}, we must prove that $A$ and the Schur complement
  \begin{equation*}
        T = M - \tilde{C} G^{-1} C\in\mathbb{R}^{(d+1)\times(d+1)}
    \end{equation*}
are both invertible.

We start by proving that $A$ is invertible. 
Since $\Omega_{\B{\alpha}}$ is a Dirichlet weight, its restriction to $F_j$ is again of
Dirichlet type with parameters $\left\{\alpha_i\right\}_{i\neq j}$, normalized to unit
mass. Therefore, using~\eqref{idbc}, we obtain
\begin{equation*}
[A]_{ji}
=
\mathcal{I}_j\left(\lambda_{i}\right)
=
\begin{cases}
0, & i = j,\\[2pt]
\dfrac{\alpha_i}{S-\alpha_j}, & i \neq j,
\end{cases}
\end{equation*}
where
\begin{equation*}
    S=\sum_{i=0}^d \alpha_i.
\end{equation*}
Hence, we can write
\begin{equation*}
A = D B,
\end{equation*}
where 
\begin{equation*}
D = \operatorname{diag}\left(\frac{1}{S-\alpha_0},\dots,\frac{1}{S-\alpha_d}\right),
\end{equation*}
and $B\in\mathbb{R}^{(d+1)\times(d+1)}$ is given by
\begin{equation*}
[B]_{ji}
=
\begin{cases}
0, & i = j,\\[2pt]
\alpha_i, & i\neq j.
\end{cases}
\end{equation*}
Thus $B$ has the form
 \begin{equation*}
 B= \begin{pmatrix}
0 & \alpha_1 & \cdots &\alpha_d\\
\alpha_0 & 0 & \cdots & \alpha_d \\
\vdots  & \vdots  & \ddots & \vdots  \\
\alpha_{0} &\alpha_{1} & \cdots & 0
\end{pmatrix}.
\end{equation*}
Consequently
 \begin{equation*}
  \det(B)= \det \begin{pmatrix}
0 & \alpha_1 & \cdots &\alpha_d\\
\alpha_0 & 0 & \cdots & \alpha_d \\
\vdots  & \vdots  & \ddots & \vdots  \\
\alpha_{0} &\alpha_{1} & \cdots & 0
\end{pmatrix}=\prod_{i=0}^{d} \alpha_i\det \begin{pmatrix}
0 & 1 & \cdots & 1\\
1 & 0 & \cdots & 1 \\
\vdots  & \vdots  & \ddots & \vdots  \\
1 & 1 & \cdots & 0
\end{pmatrix}.
\end{equation*}
In order to evaluate the last determinant, we replace the first column by the sum of all the columns, and we obtain
\begin{equation*}
   \det(B)=d \prod_{i=0}^{d} \alpha_i \det \begin{pmatrix}
1 & 1 & \cdots & 1\\
1 & 0 & \cdots & 1 \\
\vdots  & \vdots  & \ddots & \vdots  \\
1 & 1 & \cdots & 0
\end{pmatrix}.
\end{equation*}
Next, subtracting the first row from each of the remaining rows, we arrive at an upper triangular matrix
\begin{equation*}
   \det(B)=d \prod_{i=0}^{d} \alpha_i \det \begin{pmatrix}
1 & 1 & \cdots & 1\\
0 & -1 & \cdots & 0 \\
\vdots  & \vdots  & \ddots & \vdots  \\
0 & 0 & \cdots & -1
\end{pmatrix}
= (-1)^{d} d \prod_{i=0}^{d} \alpha_i \neq 0.
\end{equation*}
Therefore,
\begin{equation*}
\det(A)
=
\det(D)\det(B)
=
(-1)^d d
\left(\prod_{i=0}^d \frac{\alpha_i}{S-\alpha_i}\right)
\neq 0,
\end{equation*}
which proves that $A$ is nonsingular.

It remains to show that $T$ is invertible. By Theorem~\ref{newtheam}, we may choose
$\mathbb{V}$, $\mathbb{W}$ and the face test functions $q_j$ such that $C=0$ and
$G=I_{\tilde d}$, so that $T = M$. Moreover, the explicit construction in the proof
of Theorem~\ref{newtheam} ensures that each row of $M$ contains exactly one nonzero
entry in a distinct column, and hence $M$ is invertible. Therefore, $T$ is invertible
and the proof is complete.
\end{proof}

\begin{remark}
    In this setting, the $L^2_{\Omega_{\B{\alpha}}}$--orthogonal projection of any quadratic monomial
\begin{equation*}
\lambda_j\lambda_{\ell},\quad 0\le j<\ell\le d,
\end{equation*}
onto $\mathbb{S}_2\left(S_d\right)=\mathbb{P}_2\left(S_d\right)\ominus\mathbb{P}_1\left(S_d\right)$ admits the closed-form representation
\begin{equation*}
\rho_{j\ell}(\B{\sigma})
=
\lambda_j(\B \sigma)\lambda_{\ell}(\B \sigma)
- k_{\ell}\lambda_j(\B \sigma)
- k_{j}\lambda_\ell(\B \sigma)
+ h_{j\ell},
\end{equation*}
where 
\begin{equation*}
k_{\ell}=\frac{\alpha_\ell}{S+2}, 
\quad
k_{j} =\frac{\alpha_j}{S+2}, 
\quad  
h_{j\ell} = \frac{\alpha_j\alpha_\ell}{(S+1)(S+2)},
\quad
S = \sum_{i=0}^d \alpha_i.
\end{equation*}

We index the set of unordered vertex pairs via a bijection
\begin{equation*}
    \nu: \left\{1,\dots,\frac{d(d+1)}{2}\right\}\mapsto \left\{(j,\ell)\, :\, 0\le j<\ell\le d\right\}
\end{equation*}
and we set $\rho_\iota=\rho_{\nu(\iota)}$, $\iota=1,\dots,\frac{d(d+1)}{2}$.\\

From~\eqref{idbc}, it follows that
\begin{equation*}
    \left\langle \rho_{\iota}, 1\right\rangle_{\Omega}=0, 
    \quad 
    \left\langle \rho_{\iota}, \lambda_j\right\rangle_{\Omega}=0,
    \quad j=0,\dots,d,\quad 
    \iota=1,\dots,\frac{d(d+1)}{2},
\end{equation*}
that is, each $\rho_{\iota}$ is orthogonal to $\mathbb{P}_1\left(S_d\right)$ with respect to $\langle\cdot,\cdot\rangle_{\Omega}$.
Since $\dim\left(\mathbb{S}_2\left(S_d\right)\right)=\frac{d(d+1)}{2}$, the set $\left\{\rho_{\iota}\right\}_{\iota=1}^{\frac{d(d+1)}{2}}$ forms a basis of $\mathbb{S}_2\left(S_d\right)$. 
\end{remark}

\begin{remark}
The constant-weight case $\Omega = 1$ is recovered by choosing 
$\alpha_j = 1$ for every $j=0,\dots,d$. In this case, the coefficients reduce to
\begin{equation*}
    k_{\ell} = k_{j} = \frac{1}{d+3},
    \quad
    h_{j\ell} = \frac{1}{(d+2)(d+3)},
\end{equation*}
and the corresponding polynomials take the explicit form
\begin{equation*}
    \rho_{j\ell}(\B{\sigma})
    =
    \lambda_j(\B{\sigma})\lambda_{\ell}(\B{\sigma})
    - \frac{1}{d+3}\left(\lambda_j(\B{\sigma}) + \lambda_{\ell}(\B{\sigma})\right)
    + \frac{1}{(d+2)(d+3)}.
\end{equation*}
For clarity, we report the explicit expressions in the lowest dimensions.
\begin{itemize}
\item For $d=2$ (triangular simplex), we have $S=d+1=3$, hence
\begin{equation*}
    k_{\ell}=k_j=\frac{1}{5},
    \quad
    h_{j\ell} = \frac{1}{20}.
\end{equation*}
Therefore,
\begin{equation*}
\rho_{j\ell}(\B{\sigma})
=
\lambda_j(\B{\sigma})\lambda_\ell(\B{\sigma})
- \frac{1}{5}\left(\lambda_j(\B{\sigma})+\lambda_\ell(\B{\sigma})\right)
+ \frac{1}{20},
\quad 0\le j<\ell\le 2.
\end{equation*}

\item For $d=3$ (tetrahedral simplex), we have $S=d+1=4$, hence
\begin{equation*}
    k_{\ell}=k_j=\frac{1}{6},
    \quad
    h_{j\ell} = \frac{1}{30}.
\end{equation*}
Therefore,
\begin{equation*}
\rho_{j\ell}(\B{\sigma})
=
\lambda_j(\B{\sigma})\lambda_\ell(\B{\sigma})
- \frac{1}{6}\left(\lambda_j(\B{\sigma})+\lambda_\ell(\B{\sigma})\right)
+ \frac{1}{30},
\quad 0\le j<\ell\le 3.
\end{equation*}
\end{itemize}
\end{remark}

\subsection{Affine weight case}

The affine weight is defined by
\begin{equation}\label{eq:affine_weight}
\Omega(\B \sigma) = \sum_{i=0}^{d} \alpha_i\lambda_i(\B \sigma),
\quad \alpha_i > 0, \quad i=0,\dots,d, \quad  \B{\sigma}\in S_d,
\end{equation}
and is normalized so that
\begin{equation*}
\int_{S_d}\Omega(\B{\sigma})d\B{\sigma}=1.    
\end{equation*}
In analogy with the Dirichlet case, we now establish the unisolvence of the
quadratic degrees of freedom under the affine weight.

\begin{corollary}\label{cor:affine_weight}
Let $S_d\subset\mathbb{R}^d$ be a nondegenerate simplex with barycentric coordinates
$\lambda_0,\dots,\lambda_d$, and let $\Omega$ be the affine weight.
Let $\omega_j$ denote the induced normalized face density on $F_j$.
Then the set of degrees of freedom
\begin{equation*}
\Sigma=\left\{\mathcal{I}_j,\mathcal{L}_j,\mathcal{V}_k\ :\ j=0,\dots,d, \, k=1,\dots,\tilde{d}\right\}    
\end{equation*}
is unisolvent on $\mathbb{P}_2\left(S_d\right)$ for any choice of parameters $\alpha_i>0$.
\end{corollary}
\begin{proof}
Let $p\in\mathbb P_2\left(S_d\right)$ satisfy
\begin{equation*}
    \mathcal I_j(p)=\mathcal L_j(p)=\mathcal V_k(p)=0,\quad j=0,\dots,d,\ k=1,\dots,\tilde d.
\end{equation*}
By Theorem~\ref{thm:CNS-nonsym}, it suffices to show that both $A$ and
\begin{equation*}
T = M - \tilde{C} G^{-1} C \in \mathbb{R}^{(d+1)\times(d+1)}
\end{equation*}
are invertible. We start by proving that $A$ is invertible. Since $\lambda_j=0$ on $F_j$, we have
\begin{equation}\label{eq:A_entries_affine}
[A]_{jj}=0, \quad 
[A]_{ji}
=
\mathcal{I}_j\left(\lambda_{i}\right)
=
\int_{F_j} \mu_{j,i}(\B x) \omega_j(\B x) d\B x
=
\frac{\displaystyle \sum_{\substack{s=0 \\ s\neq j}}^{d} \alpha_s \int_{F_j} \mu_{j,i}(\B x)\mu_{j,s}(\B x) d\B x}
{\displaystyle \sum_{\substack{s=0 \\ s\neq j}}^{d} \alpha_s \int_{F_j} \mu_{j,s}(\B x) d\B x}.
\end{equation}
Using~\eqref{idbc}, we get
\begin{eqnarray*}
    \int_{F_j} \mu_{j,i}(\B x)d\B x &=& \frac{\left|F_j\right|}{d}, \\
\int_{F_j} \mu_{j,i}^2(\B x)d\B x &=& \frac{2\left|F_j\right|}{d(d+1)},\\
\int_{F_j} \mu_{j,i}(\B x)\mu_{j,s}(\B x)d\B x&=&\frac{\left|F_j\right|}{d(d+1)}, \ i\neq s.
\end{eqnarray*}
Substituting these expressions into~\eqref{eq:A_entries_affine} yields
\begin{equation}\label{eq:Arow}
[A]_{jj}=0, \quad [A]_{ji}= \frac{S_j+\alpha_i}{(d+1)S_j},
\ i\neq j,
\end{equation}
where 
\begin{equation}\label{Sj1}
    S_j = \sum_{\substack{i=0 \\ i\neq j}}^{d} \alpha_i. 
\end{equation}
Let 
$\B x=\left(x_0,\dots,x_d\right)^\top\in \mathbb{R}^{d+1}$ satisfy
\begin{equation*}
    A\B x=\B 0.
\end{equation*}
We set
\begin{equation}\label{notations}
    X = \sum_{i=0}^{d} x_i,
\quad
Y = \sum_{i=0}^{d} \alpha_i x_i,
\quad
\alpha_{\mathrm{tot}}= \sum_{i=0}^{d} \alpha_i.
\end{equation}
From~\eqref{eq:Arow}, for each $j=0,\dots,d$, we obtain
\begin{equation*}
    0 = [A\B x]_j = \sum_{\substack{i=0 \\ i\neq j}}^{d} [A]_{ji} x_{i}= \sum_{\substack{i=0 \\ i\neq j}}^{d}\frac{S_j+\alpha_i}{(d+1)S_j}x_i
= \frac{1}{(d+1)S_j}\left(S_j \sum_{\substack{i=0 \\ i\neq j}}^{d} x_i + \sum_{\substack{i=0 \\ i\neq j}}^{d} \alpha_i x_i \right),
\end{equation*}
which implies
\begin{equation*}
    S_j \sum_{\substack{i=0 \\ i\neq j}}^{d} x_i + \sum_{\substack{i=0 \\ i\neq j}}^{d} \alpha_i x_i=0.
\end{equation*}
Therefore, adding the missing $i=j$ terms and using~\eqref{Sj1}, we obtain
\begin{equation*}
    S_j \sum_{i=0}^{d} x_i + \sum_{\substack{i=0}}^{d} \alpha_i x_i=S_j x_j+\alpha_j x_j= \sum_{\substack{i=0 \\ i\neq j}}^{d} \alpha_i x_j +\alpha_j x_j,
\end{equation*}
that is
\begin{equation}\label{eq:key}
S_j X + Y=
\alpha_{\mathrm{tot}}x_j,
\quad j=0,\dots,d.
\end{equation}
Summing~\eqref{eq:key} over $j=0,\dots,d$, yields
\begin{equation}\label{qwqwq}
    \left( \sum_{j=0}^{d} S_j \right) X + (d+1)Y
=  \alpha_{\mathrm{tot}} X.
\end{equation}
Since 
\begin{equation}\label{Sj}
  S_j=\alpha_{\mathrm{tot}}-\alpha_j,
\end{equation}
we have
\begin{equation}\label{combs}
    \sum_{j=0}^{d} S_j=(d+1) \alpha_{\mathrm{tot}}-\alpha_{\mathrm{tot}}=d\alpha_{\mathrm{tot}}.
\end{equation}
Substituting into~\eqref{qwqwq} gives
\begin{equation*}
    d\alpha_{\mathrm{tot}}X + (d+1)Y=\alpha_{\mathrm{tot}} X,
\end{equation*}
and hence
\begin{equation}\label{eq:yX}
Y=-\frac{d-1}{d+1}\alpha_{\mathrm{tot}}X.
\end{equation}
On the other hand, using~\eqref{notations} and \eqref{eq:key}, we have
\begin{equation*}
Y = \sum_{i=0}^{d} \alpha_i x_i
= \sum_{i=0}^{d} \alpha_i\left( \frac{S_i X + Y}{\alpha_{\mathrm{tot}}} \right)
\end{equation*}
and by~\eqref{Sj}, we get
\begin{equation*}
   Y=\left( \frac{\displaystyle \alpha_{\mathrm{tot}}^2 - \sum_{i=0}^{d}\alpha_i^2}{\alpha_{\mathrm{tot}}} \right) X +Y.
\end{equation*}
Hence
\begin{equation*}
    \left( \frac{\displaystyle \alpha_{\mathrm{tot}}^2 - \sum_{i=0}^{d}\alpha_i^2}{\alpha_{\mathrm{tot}}} \right) X=0.
\end{equation*}
Since $\alpha_i>0$ for all $i$, we have
\begin{equation*}
\alpha_{\mathrm{tot}}^2 - \sum_{i=0}^{d}\alpha_i^2=2\sum_{0\le i<j\le d}\alpha_i\alpha_j > 0,
\end{equation*}
which implies $X=0$. By~\eqref{eq:yX}, it follows that $Y=0$ as well.
Finally,~\eqref{eq:key} gives $x_j=0$ for $j=0,\dots,d$, and therefore
$\B{x}=\B{0}$, i.e., $\ker(A)=\{0\}$.
Hence $A$ is invertible.

It remains to show that $T$ is invertible. By Theorem~\ref{newtheam}, we may choose
$\mathbb{V}$, $\mathbb{W}$ and the face test functions $q_j$ such that $C=0$ and
$G=I_{\tilde d}$, so that $T=M$. Moreover, the same construction ensures that each
row of $M$ contains exactly one nonzero entry in a distinct column, hence $M$ is
invertible. Therefore, $T$ is invertible and the proof is complete.
\end{proof}

\section{Parametric Scaling of Moment Functionals and Optimization Strategy}
\label{sec5}
In this section, we introduce a parametric diagonal scaling of the moment functionals
\begin{equation*}
    \Sigma=\left\{ \mathcal{I}_j, \mathcal{L}_j, \mathcal{V}_k \, :\, j=0,\dots,d, \ k=1,\dots,\tilde{d} \right\},
\end{equation*}
which preserves the analytical framework developed above while allowing for the optimization
of stability and conditioning properties of the resulting discrete systems. We introduce strictly positive scaling parameters
\begin{equation*}
   \theta_j>0,\quad \upsilon_k>0, \quad j=0,\dots,d, \quad k=1,\dots,\tilde{d},
\end{equation*}
and define the parameter vectors
\begin{equation*}
   \B \theta=\left(\theta_0,\dots,\theta_d\right)^{\top}, \quad \B \upsilon=\left(\upsilon_1,\dots,\upsilon_{\tilde{d}}\right)^{\top}.
\end{equation*}
The scaled moment functionals are defined as
\begin{equation}\label{eq:scaled_functionals}
\tilde{\mathcal{L}}_j=\theta_j \mathcal{L}_j,\quad
\tilde{\mathcal{V}}_k=\upsilon_k \mathcal{V}_k, \quad j=0,\dots,d,\quad k=1,\dots,\tilde{d}.
\end{equation}
The parameters
$\B{\theta}$ and $\B{\upsilon}$ act only on the quadratic moment subsystem
associated with $\mathbb{S}_2\left(S_d\right)$.
Accordingly, the matrix $\tilde H$ associated with the scaled quadratic moment functionals
$\tilde{\mathcal{L}}_j,\tilde{\mathcal{V}}_k$, $j=0,\dots,d,$ $k=1,\dots,\tilde{d},$ can be expressed in terms of the standard matrix
$H$ defined in~\eqref{matH} as
\begin{equation*}
\tilde H=
\begin{pmatrix}
D_{\mathcal{V}} & 0\\[2pt]
0 & D_{\mathcal{L}}
\end{pmatrix}
H
=
\begin{pmatrix}
D_{\mathcal{V}} & 0\\[2pt]
0 & D_{\mathcal{L}}
\end{pmatrix}
\begin{pmatrix}
G & C\\[2pt]
\tilde C & M
\end{pmatrix},
\end{equation*}
where
\begin{equation*}
D_{\mathcal{V}}=\mathrm{diag}\left(\upsilon_1,\dots,\upsilon_{\tilde d}\right),\quad
D_{\mathcal{L}}=\mathrm{diag}\left(\theta_0,\dots,\theta_d\right).
\end{equation*}
Since $D_{\mathcal{V}}$ and $D_{\mathcal{L}}$ are invertible diagonal matrices, this scaling
preserves the rank of the quadratic moment matrix and therefore does not affect the
unisolvence properties established in Theorem~\ref{thm:CNS-nonsym}.

\subsection{Parametric Dirichlet densities}
We set
\begin{equation*}
\B{\alpha}=\left(\alpha_0,\dots,\alpha_d\right)^{\top},\quad \alpha_i>0, \quad i=0,\dots,d,
\end{equation*}
and define the parametric Dirichlet density
\begin{equation*}
\Omega_{\B \alpha}(\B \sigma)
=C_{\B \alpha,d}
 \prod_{j=0}^d \lambda_j^{\alpha_j-1}(\B \sigma),
\end{equation*}
where $C_{\B{\alpha},d}$ denotes the normalization constant.
Collecting density and scaling parameters, we define the full parameter vector
\begin{equation*}
\B p=\left(\B \alpha,\B \theta,\B \upsilon\right)^{\top}\in \mathbb{R}_+^{2(d+1)+\tilde d}.    
\end{equation*}

\subsection{Optimization objectives}

The aim of the optimization is to enhance the stability of the scheme by tuning
the density parameters $\B{\alpha}$ and the scaling factors
$\B{\theta}$ and $\B{\upsilon}$.
These parameters directly influence the conditioning of the global system and,
consequently, its numerical performance.
To this end, we introduce the parametric vector $\B p=\left(\B \alpha,\B \theta,\B \upsilon\right)^{\top}\in \mathbb{R}_+^{2(d+1)+\tilde d}$, which allows us
to formulate optimization problems focused on the spectral properties of the
histopolation system.

We recall the identity
\begin{equation*}
\beta(\B p)=\sqrt{\sigma_{\min}\left(\hat S(\B p)\right)},
\end{equation*}
where $\hat S(\B p)$ denotes the normalized Schur operator defined in~\eqref{matThat}.
We therefore consider the maximization problem
\begin{equation}\label{eq:max_infsup}
\max_{\B p\in\mathbb{R}_{+}^{2(d+1)+\tilde{d}}}
\beta(\B p)
=
\max_{\B p\in\mathbb{R}_{+}^{2(d+1)+\tilde{d}}}
\sqrt{\sigma_{\min}\left(\hat S(\B p)\right)}.
\end{equation}
Alternatively, we can minimize the conditioning of the parametrized quadratic moment matrix, that is
\begin{equation}\label{eq:min_conditioning}
\min_{\B p\in\mathbb{R}_{+}^{2(d+1)+\tilde{d}}}
\kappa_2\left(\tilde H(\B p)\right).
\end{equation}
Since the scalings are diagonal, these problems admit a natural interpretation as
\emph{automatic spectral preconditioning} of the quadratic moment subsystem.

\subsection{Optimization procedure}
The parameter vector $\B p$ has moderate dimension and the objective functions
$\beta(\B p)$ and $\kappa_2\left(\tilde H(\B p)\right)$ are continuous but generally nonconvex.  
We therefore rely on the following optimization methods:
\begin{itemize}
    \item the Nelder--Mead simplex algorithm,
    \item bounded coordinate search,
    \item or a projected quasi--Newton method on the positive orthant.
\end{itemize}
In our implementation, the update step is carried out using a projected quasi–Newton method, ensuring positivity of all parameters.
At each iteration, the following steps are carried out:

\begin{enumerate}
    \item Assemble the Dirichlet density $\Omega_{\B \alpha}$ and the induced
          face densities $\omega_{j,\B \alpha}$.
    \item Compute the moment matrices
          \begin{equation*}
              G=G(\B p),\quad 
            C=C(\B p),\quad \tilde C=\tilde C(\B p), \quad M=M(\B p).
          \end{equation*}
    \item Form the scaled quadratic moment matrix
          \begin{equation*}
              \tilde H(\B p)
              =
              \begin{pmatrix}
              D_{\mathcal{V}} & 0\\[2pt]
              0 & D_{\mathcal{L}}
              \end{pmatrix}
              \begin{pmatrix}
              G & C \\
              \tilde C & M
              \end{pmatrix}.
  \end{equation*}
    \item Construct the normalized Schur operator $\hat S(\B p)$.
    \item Evaluate the chosen objective: either $\beta(\B p)$ or 
          $\kappa_2(\tilde H(\B p))$.
    \item Update the parameter vector $\B p$ according to the selected
          optimization scheme.
\end{enumerate}

The overall computational cost is modest, as all operations occur on a
single simplex and involve matrices of fixed size.

\begin{algorithm}[ht]
\caption{Parametric optimization of the moment system}
\begin{algorithmic}[1]
\State Choose an initial guess 
       $\B p^{(0)} = (\B \alpha^{(0)},\B \theta^{(0)},\B \upsilon^{(0)}) > 0$
\For{$n=0,1,2,\dots$ until convergence}
    \State Build $\Omega_{\B \alpha^{(n)}}$ and the face densities $\omega_{j,\B \alpha^{(n)}}$
    \State Assemble $G^{(n)},\, C^{(n)},\, \tilde C^{(n)}, \, M^{(n)}$
    \State Form $H^{(n)}$ using $G^{(n)},\, C^{(n)},\, \tilde C^{(n)}, \, M^{(n)}$
    \State Form $\tilde H^{(n)} = 
           \begin{pmatrix}
                D_{\mathcal{V}}^{(n)} & 0 \\
                0 & D_{\mathcal{L}}^{(n)}
           \end{pmatrix}
           H^{(n)}$
    \State Compute the Schur operator $\hat S^{(n)}$
    \State Evaluate the objective $\Phi^{(n)}$, with
           $\Phi^{(n)}=\beta(\B p^{(n)})$ or $\Phi^{(n)}=\kappa_2(\tilde H^{(n)})$
    \State Update $\B p^{(n+1)}=\texttt{OptimizeStep}(\B p^{(n)},\Phi^{(n)})$
\EndFor
\State \Return optimized parameters $\B p^\star$
\end{algorithmic}
\end{algorithm}


\subsection{Computational cost}
We finally note that the optimized choice of moment weights does not increase the
computational cost of the method. The assembly of the face and interior blocks
only requires the evaluation of local quadratic moments and the solution of a
$(d+1)\times(d+1)$ linear system per element. Therefore, the overall complexity
remains identical to that of the unoptimized scheme.
\begin{remark}
The diagonal scaling parameters $(\B\theta,\B\upsilon)$ act through invertible left scalings
of the quadratic moment functionals and therefore preserve the rank of the associated
quadratic moment matrix. Moreover, by the unisolvence results established in Corollary~\ref{corDirichlet}, the
full moment system remains unisolvent for all admissible Dirichlet parameters
$\B\alpha\in\mathbb{R}_+^{d+1}$.
\end{remark}

\section{Numerical experiments}\label{sec6}
In this section, we present a series of numerical experiments designed to validate the
theoretical framework developed in the previous sections.
The goals are threefold:
\begin{itemize}
    \item[$(i)$] to assess the approximation properties of the enriched quadratic histopolation operator;  
\item[$(ii)$] to examine its robustness with respect to mesh quality;
\item[$(iii)$] to quantify the influence of the inf-sup stability constant 
   $ \beta = \sqrt{\sigma_{\min}\left(\hat S\right)}$.
\end{itemize} 
All computations were carried out in \textsc{Matlab}, and the linear functionals defining 
the reconstruction operators were evaluated using Gaussian quadrature rules of sufficiently
high order, so that the integration error is negligible with respect to the reconstruction
error.

We consider the following set of test functions
\begin{equation*}
    f_1(x,y,z)=\sin(2\pi x)\sin(2\pi y)\sin(2\pi z), \quad f_2(x,y,z)= \sin(2 \pi xyz),
\end{equation*}
\begin{equation*}
    f_3(x,y,z)=\frac{1}{x^2+y^2+z^2+25}, \quad f_4(x,y,z)=e^{x^2+y^2+z^2},  
\end{equation*}
\begin{equation*}
    f_5(x,y,z)=\sin(x)\cos(y)e^{-z^2}, \quad f_6(x,y,z)= \log\left(x^3y^3z^3+1/4\right),
\end{equation*}
\begin{equation*}
    f_7(x,y,z)=\sqrt{(x-0.5)^2 + (y-0.5)^2 + (z-0.5)^2},
\end{equation*}
\begin{eqnarray*}
f_8(x,y,z) &=&\sin\left(10\sqrt{(x-0.5)^2 + (y-0.5)^2 + (z-0.5)^2} \right) \\ &\times& e^{-\sqrt{(x-0.5)^2 + (y-0.5)^2 + (z-0.5)^2}},
\end{eqnarray*}
\begin{equation*}
    f_9(x,y,z)= \sin(2 \pi xyz)e^{x^2+y^2+z^2},
\end{equation*}
\noindent
all defined on the cubic domain $\mathcal{Q}=[0,1]^3$.
This set includes smooth and non-smooth functions, functions with radial symmetry,
and examples exhibiting significant oscillations; taken together, they provide a
reasonably demanding benchmark for assessing histopolation-based approximation.

We compare the accuracy of the classical linear histopolation method with that of the
enriched quadratic scheme developed in this paper.
The weight is chosen to be a Dirichlet density, which is affine-invariant and admits
closed-form expressions for all face and volume moments.
The comparison is performed on both uniformly refined meshes and quasi-uniform
tetrahedral grids, thereby assessing the behaviour of the method under mesh refinement
as well as its sensitivity to variations in mesh quality.

\subsection{Uniform triangulation}
We first examine a family of uniform tetrahedral meshes obtained by subdividing the
cube into $(n-1)^3$ Cartesian cells, each of which is split into six congruent tetrahedra.
For a given resolution parameter $n\in\mathbb{N}$, we define
\begin{equation*}
\mathcal{T}_{n}=\left\{K_i\,:\, i=1,\dots,6(n-1)^3\right\},     
\end{equation*}
see Fig.~\ref{f1agwn}.  This classical construction yields shape-regular meshes with controlled element quality,
making it a natural baseline for the present comparison.
\begin{figure}
    \centering
\includegraphics[width=0.49\linewidth]{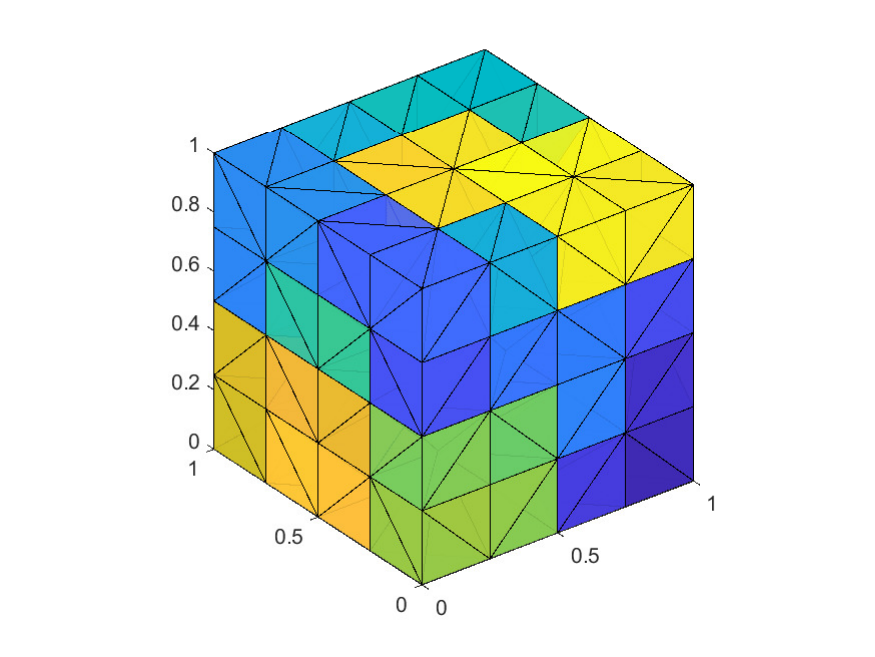}
\includegraphics[width=0.49\linewidth]{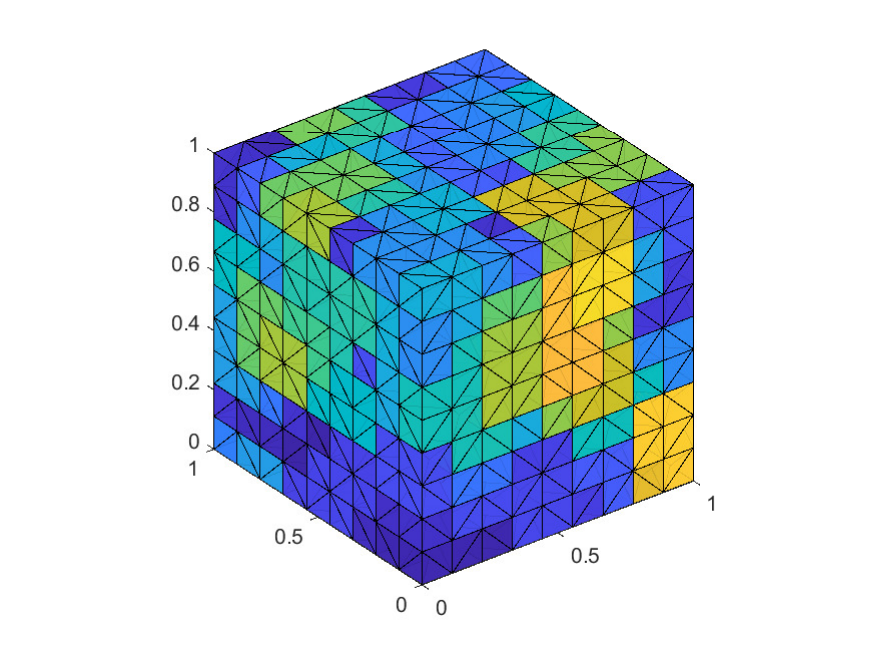}
\includegraphics[width=0.49\linewidth]{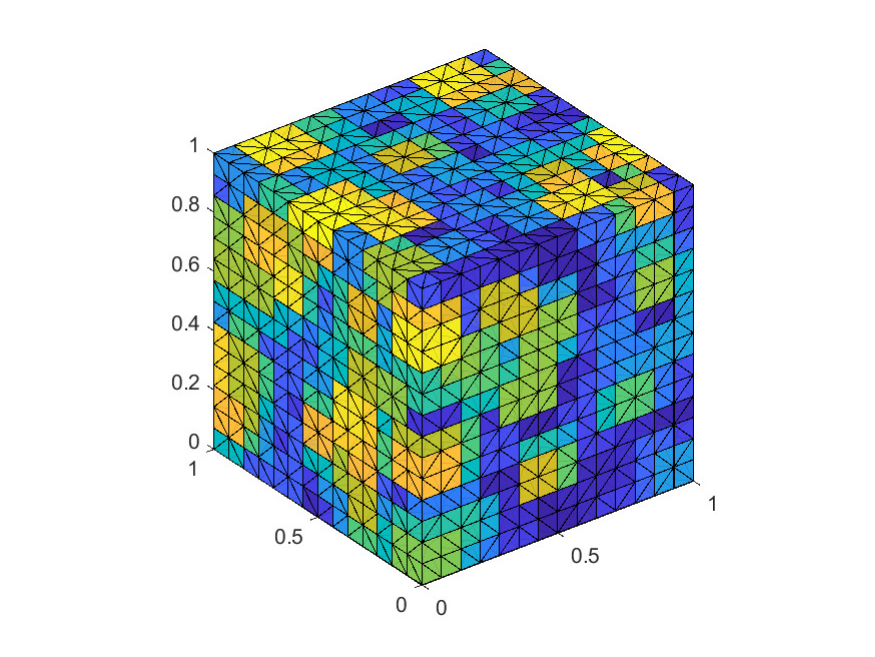}
\includegraphics[width=0.49\linewidth]{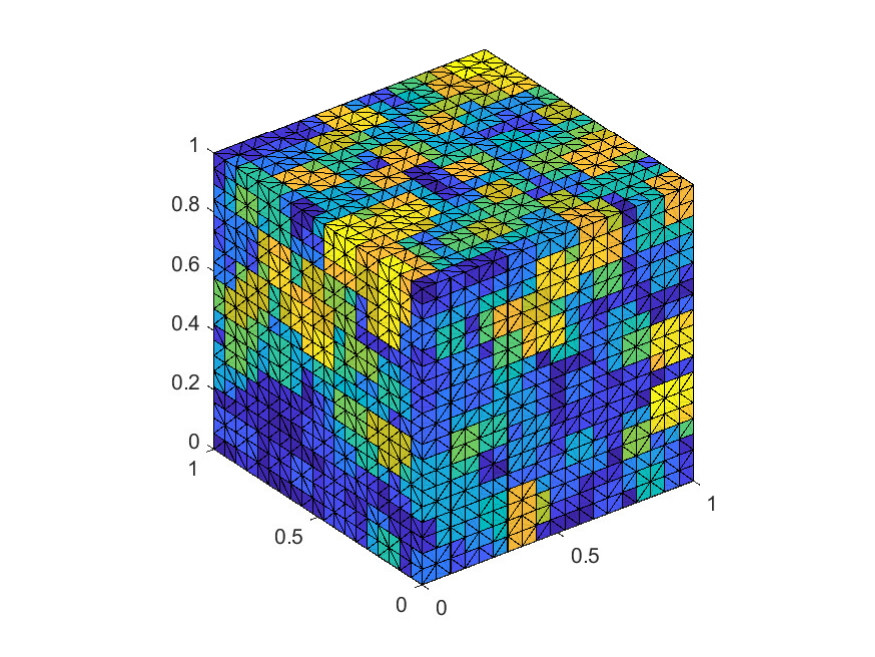}
       \caption{Uniform Delaunay triangulations $\mathcal{T}_n$ for $n=5$, $10$, $15$, and $20$.}
    \label{f1agwn}
\end{figure}
For each test function $f_i$, we compute the $L^2$ error associated with the classical local 
linear histopolation method ($\varepsilon_1^{\mathrm{CH}}$) and that obtained using the
enriched quadratic scheme with Dirichlet densities and optimized parameters
$\B p^{\star}$ ($\varepsilon_{2,\B p^{\star}}^{\mathrm{enr}}$).
The results, shown in Fig.~\ref{plots}, report a clear and consistent trend: the quadratic
scheme produces smaller errors over the entire range of mesh sizes, and its decay rate is
consistently steeper. 

In particular, combining the polynomial reproduction property of the enriched quadratic
histopolation operator with the uniform inf--sup stability of the Dirichlet moment systems
established above, we infer that the resulting elementwise reconstruction defines a family of linear operators
\begin{equation*}
    \tilde{\Pi}_{2,h}: L^2(\mathcal Q)\to \mathbb P_2(\mathcal T_h)\subset L^2(\mathcal Q), \quad h>0.
\end{equation*}
Since $\tilde{\Pi}_{2,h}$ reproduces quadratic polynomials, standard approximation results for
quadratic finite elements (see, e.g.,~\cite[Ch.~3]{Ciarlet2002TFE}) ensure an $L^2$
convergence rate of order $O(h^{3})$ on shape-regular meshes for sufficiently smooth data.
For test functions with limited regularity, a reduced asymptotic rate is expected.
Here the required uniform inf--sup bound follows from Theorem~\ref{thm:affine-robust}, since in the Dirichlet setting the local stability constant $\beta_K$ is
uniformly bounded under affine transport of the weights and test functions.
This behaviour is precisely reflected in the numerical results reported in
Fig.~\ref{plots}.

\begin{figure}
    \centering
\includegraphics[width=0.32\linewidth]{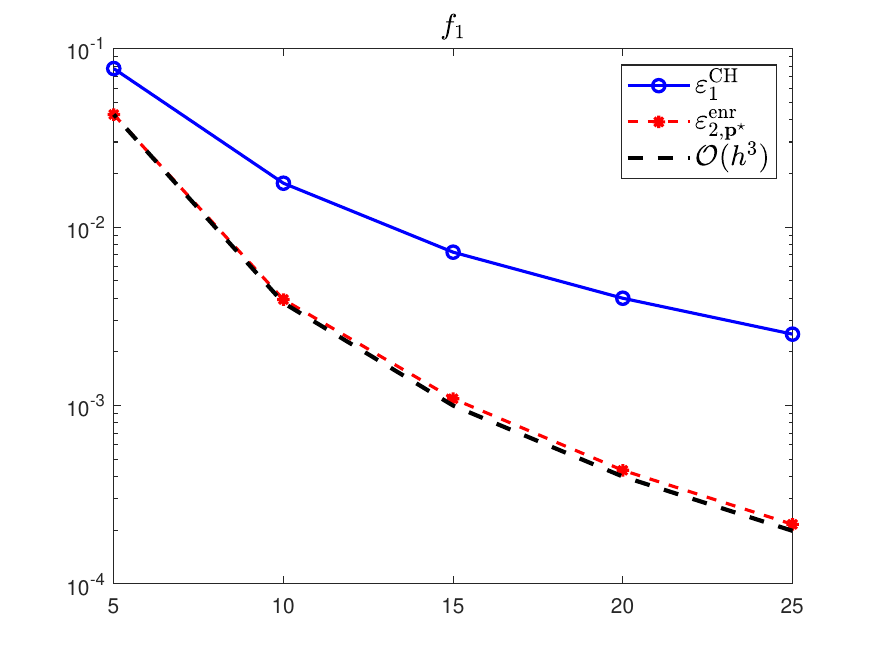}
\includegraphics[width=0.32\linewidth]{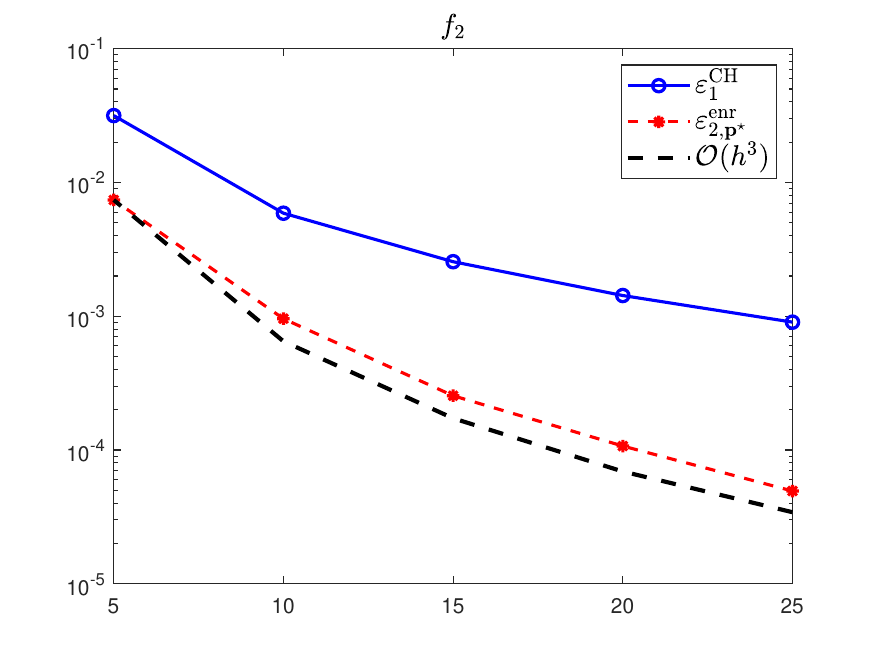}
\includegraphics[width=0.32\linewidth]{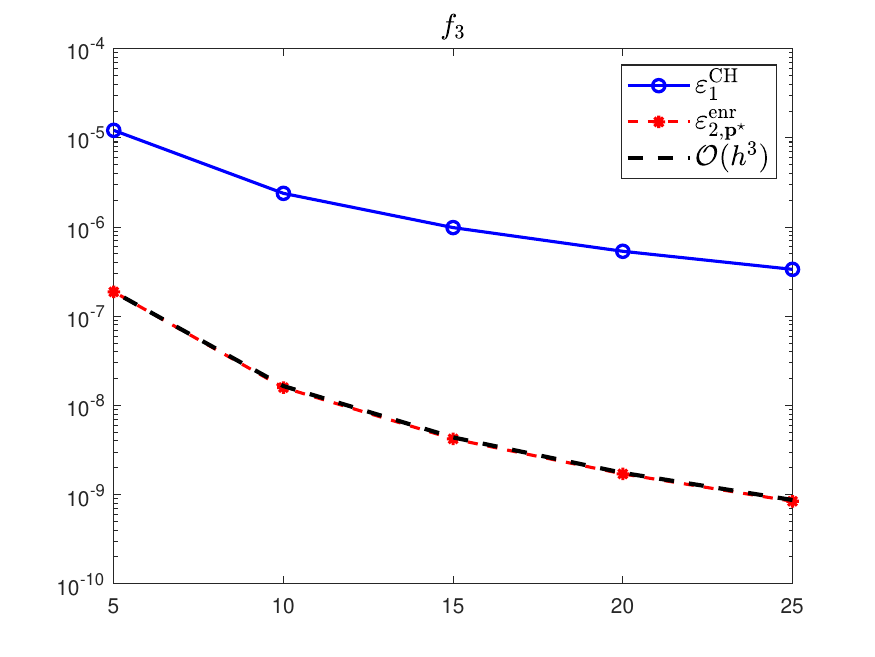}
\includegraphics[width=0.32\linewidth]{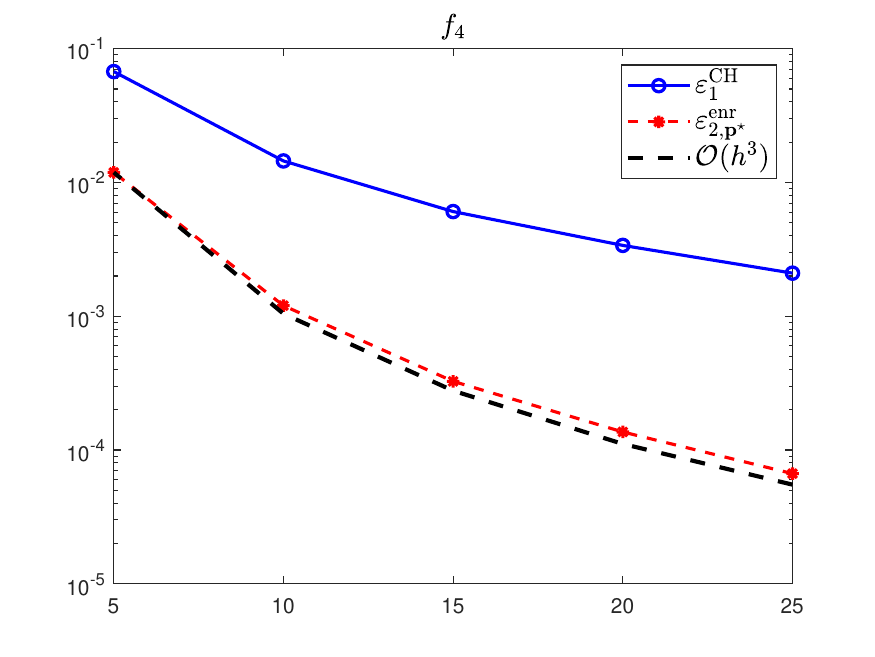}
\includegraphics[width=0.32\linewidth]{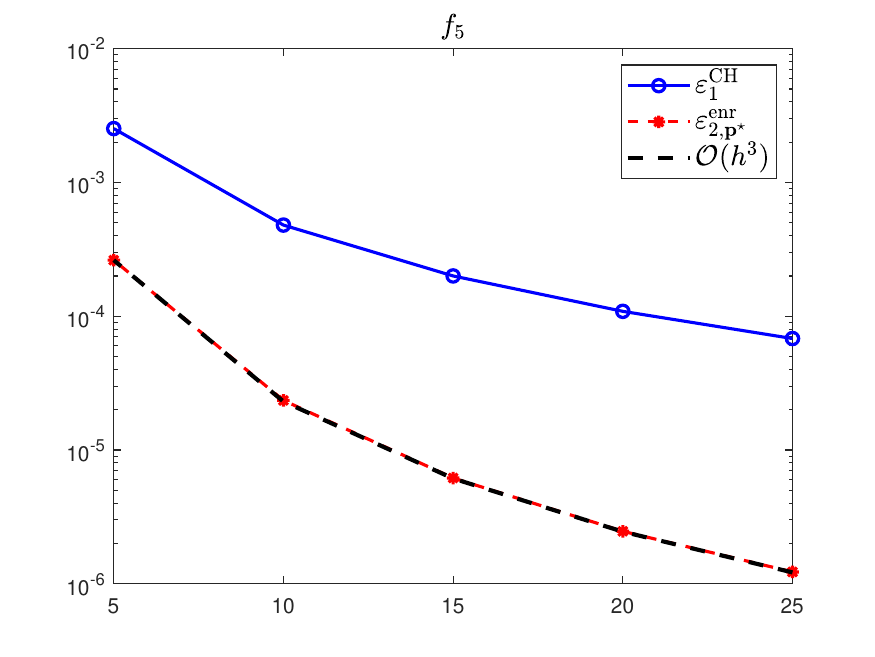}
\includegraphics[width=0.32\linewidth]{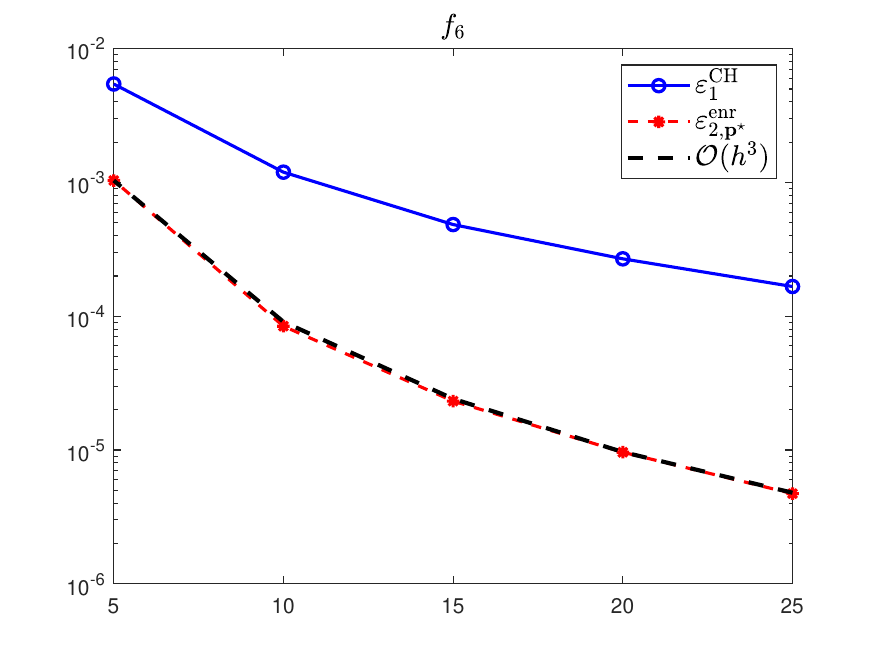}
\includegraphics[width=0.32\linewidth]{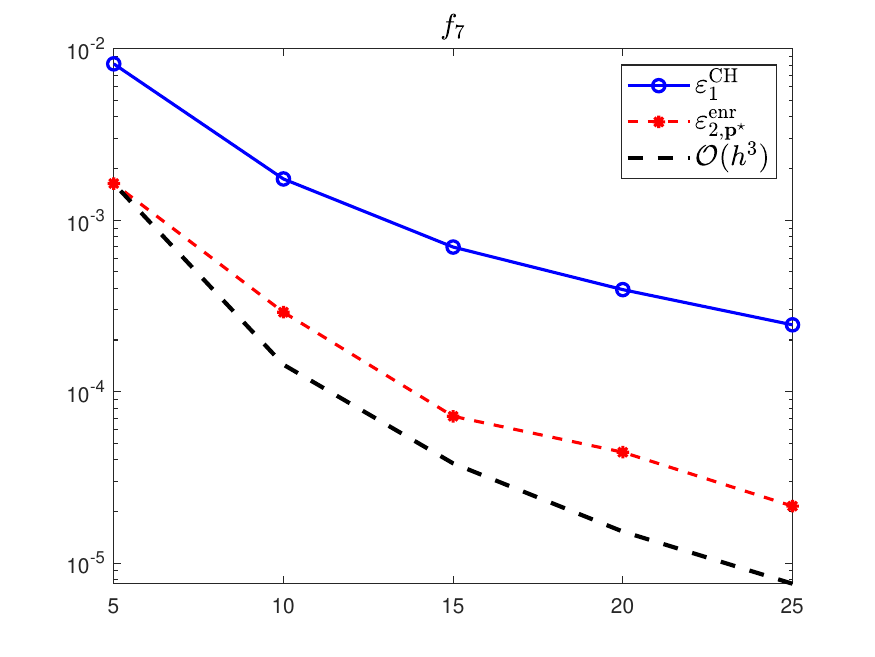}
\includegraphics[width=0.32\linewidth]{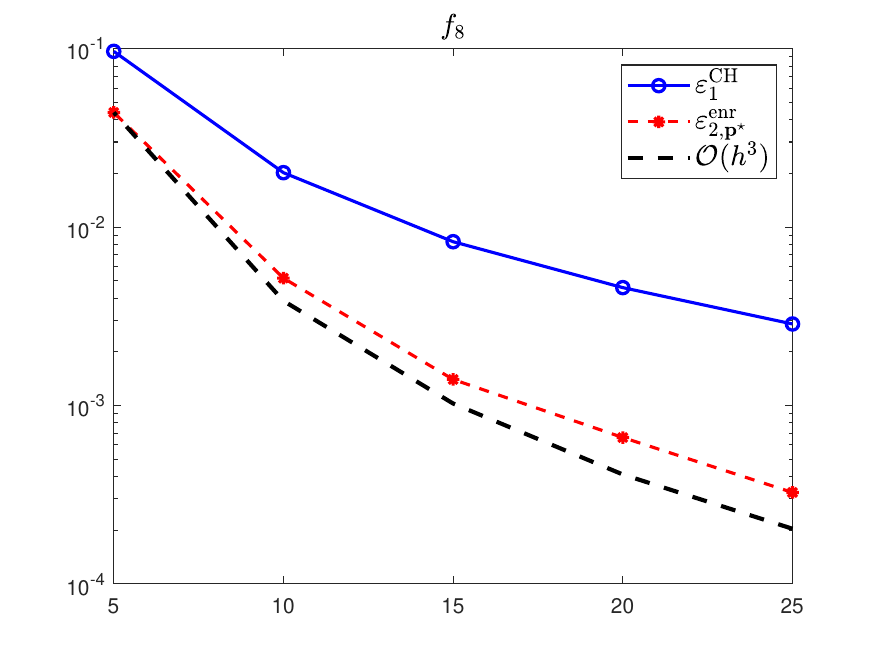}
\includegraphics[width=0.32\linewidth]{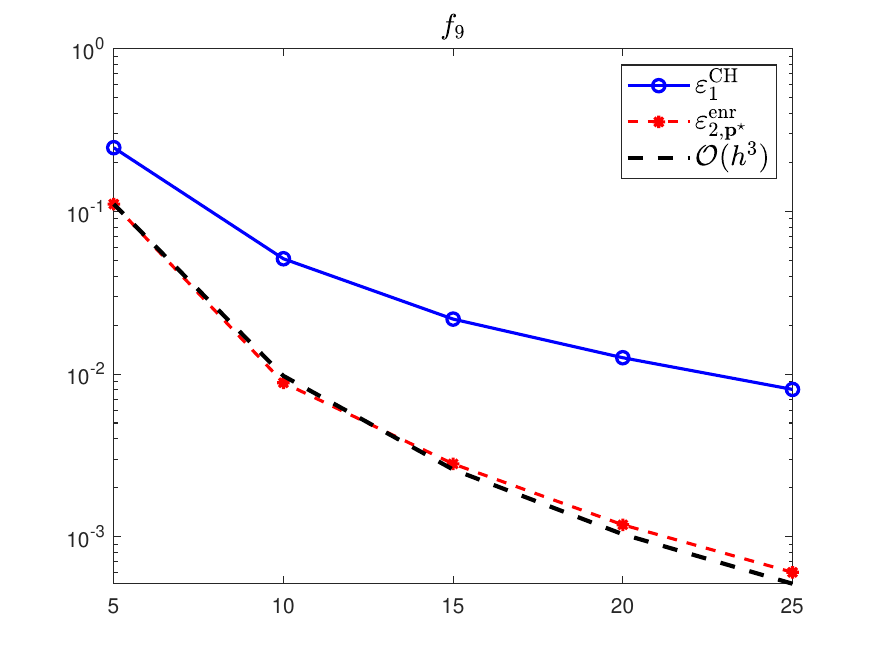}
\caption{Comparison of the $L^2$ errors $\varepsilon_1^{\mathrm{CH}}$ (linear histopolation) and $\varepsilon_{2,\B{p}^{\star}}^{\mathrm{enr}}$ (proposed quadratic scheme) over uniform meshes.}
    \label{plots}
\end{figure}

\subsection{Quasi-uniform triangulation}
We next investigate the behaviour of the method on quasi-uniform meshes.
In this setting, each triangulation $\tilde{\mathcal T}_n$ is obtained by starting from
a regular Cartesian grid on $[0,1]^3$ and then applying a small random perturbation to
the interior grid points, while keeping all boundary vertices fixed.
The perturbed point cloud is subsequently tessellated using the
\texttt{delaunayn} routine in \textsc{Matlab}, producing triangulations that are
unstructured yet remain shape-regular, with mesh size comparable to that of the
underlying Cartesian grid.

Although the grids are unstructured, Theorem~\ref{thm:affine-robust} still applies
elementwise: each tetrahedron is an affine image of the reference simplex, and,
for Dirichlet weights transported by barycentric pullback, the associated local
moment systems satisfy uniform inf--sup bounds.
As a consequence, the local stability constant $\beta_K$ remains uniformly bounded
across elements, which is consistent with the behaviour observed on both uniform
and quasi-uniform meshes.

For each test function $f_i$, $i=1,\dots,9$, we compute the same $L^2$ errors as in
the uniform case, namely the error $\varepsilon_1^{\mathrm{CH}}$ associated with the
classical local linear histopolation operator and the error
$\varepsilon_{2,\B p^{\star}}^{\mathrm{enr}}$ associated with the proposed enriched
quadratic scheme.
The resulting error curves as functions of the effective mesh size $h$ are reported
in Fig.~\ref{plots1}.
The qualitative behaviour is essentially the same as that observed on the structured
meshes: for all test functions, the quadratic reconstruction is uniformly more accurate
than the linear method, and the difference between the slopes of the two error curves
is clearly visible.
As in the uniform case, the error decay of $\varepsilon_{2,\B p^{\star}}^{\mathrm{enr}}$
with mesh refinement is consistently steeper than that of
$\varepsilon_1^{\mathrm{CH}}$, indicating a higher effective convergence rate.
The fact that this behaviour is preserved on quasi-uniform, mildly unstructured meshes
confirms that the method is robust with respect to moderate mesh distortion and loss
of alignment.

\begin{figure}
    \centering
\includegraphics[width=0.32\linewidth]{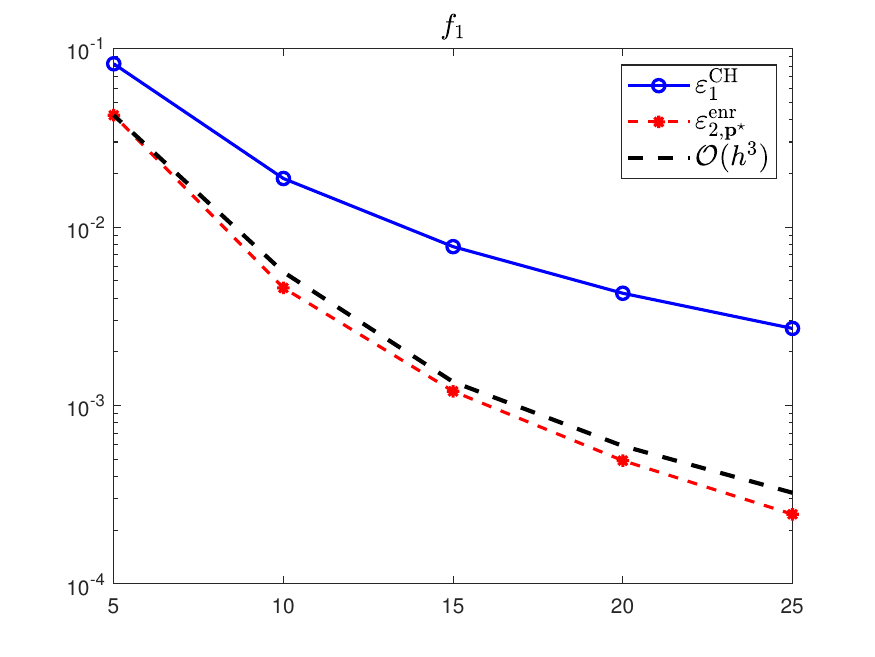}
\includegraphics[width=0.32\linewidth]{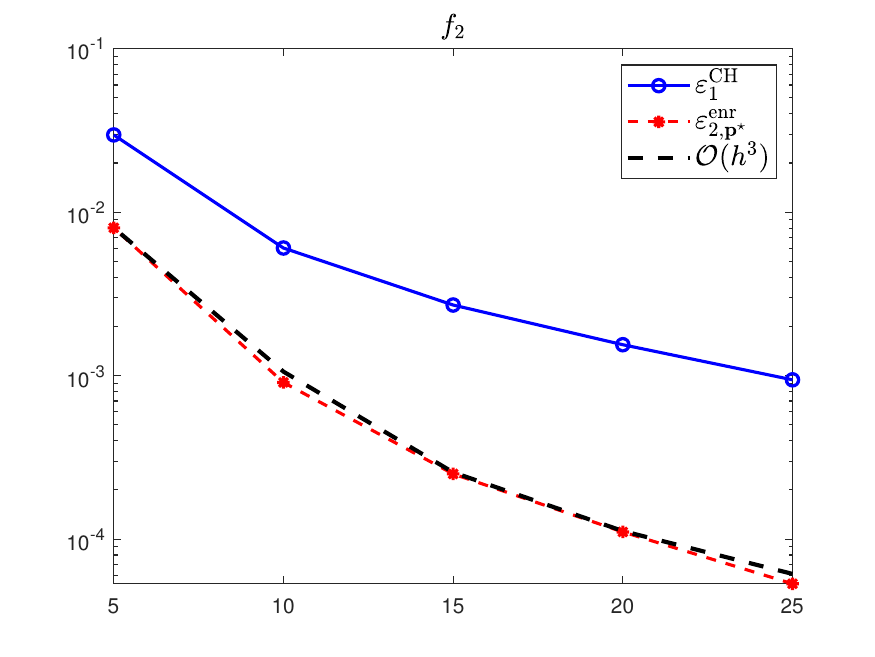}
\includegraphics[width=0.32\linewidth]{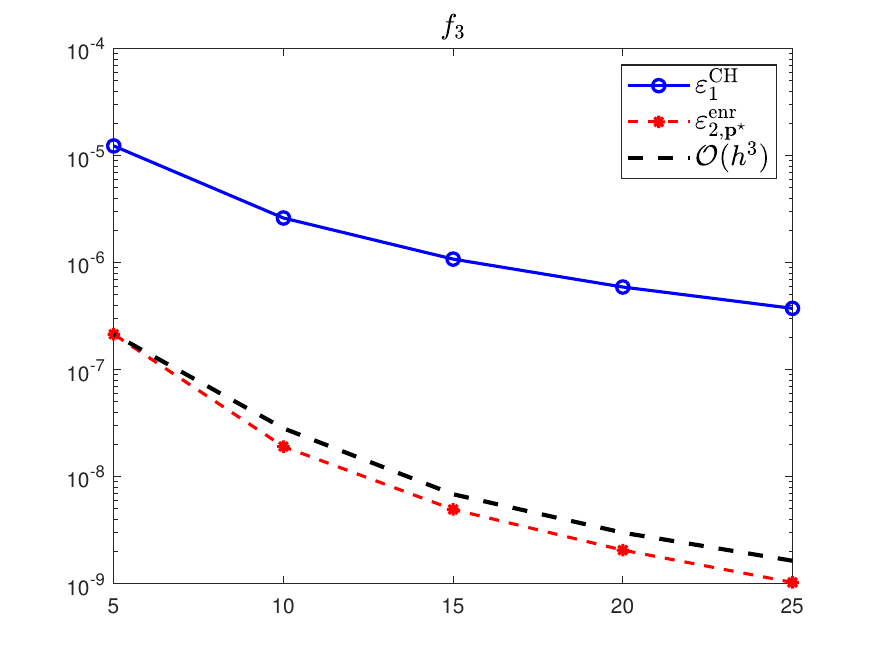}
\includegraphics[width=0.32\linewidth]{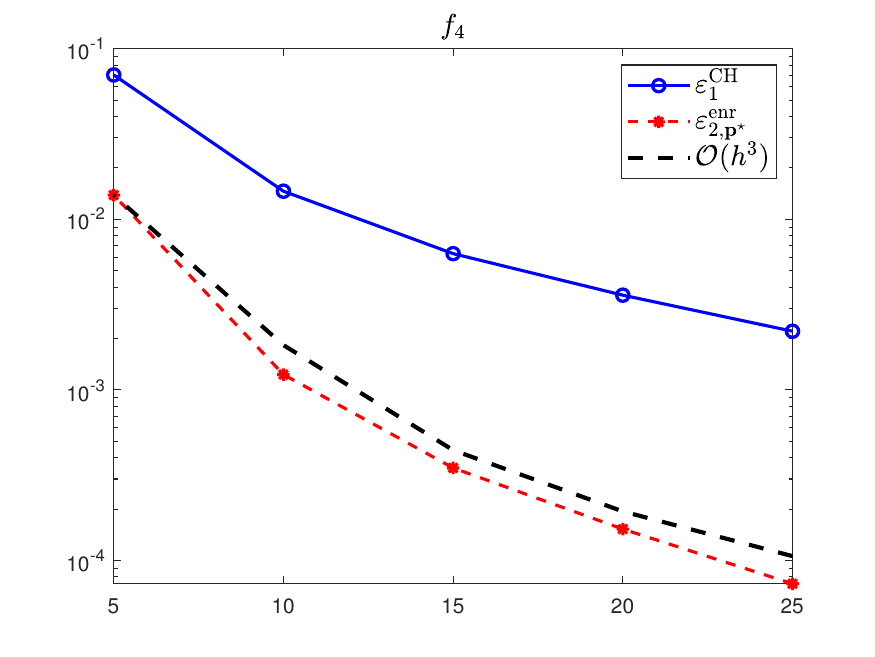}
\includegraphics[width=0.32\linewidth]{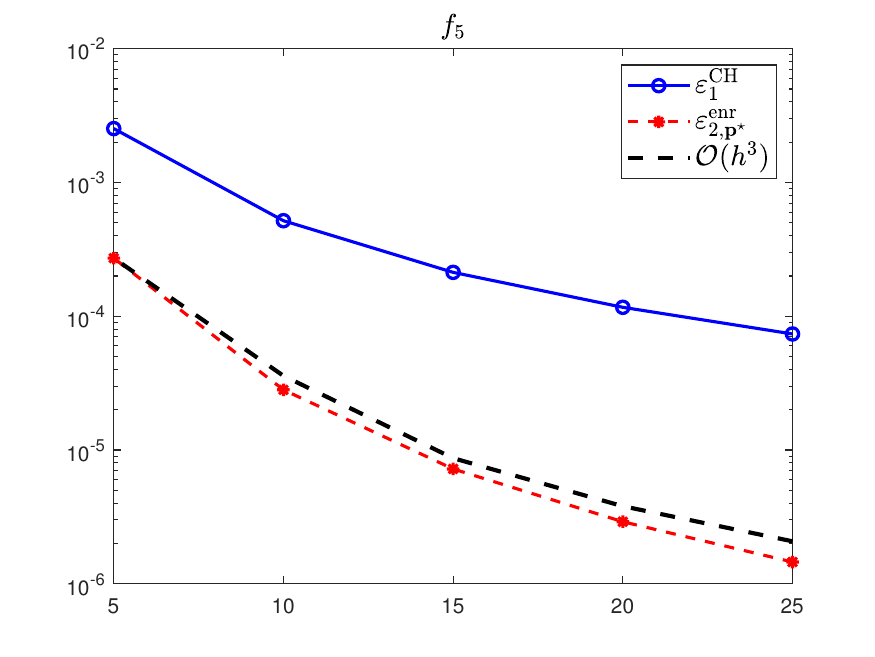}
\includegraphics[width=0.32\linewidth]{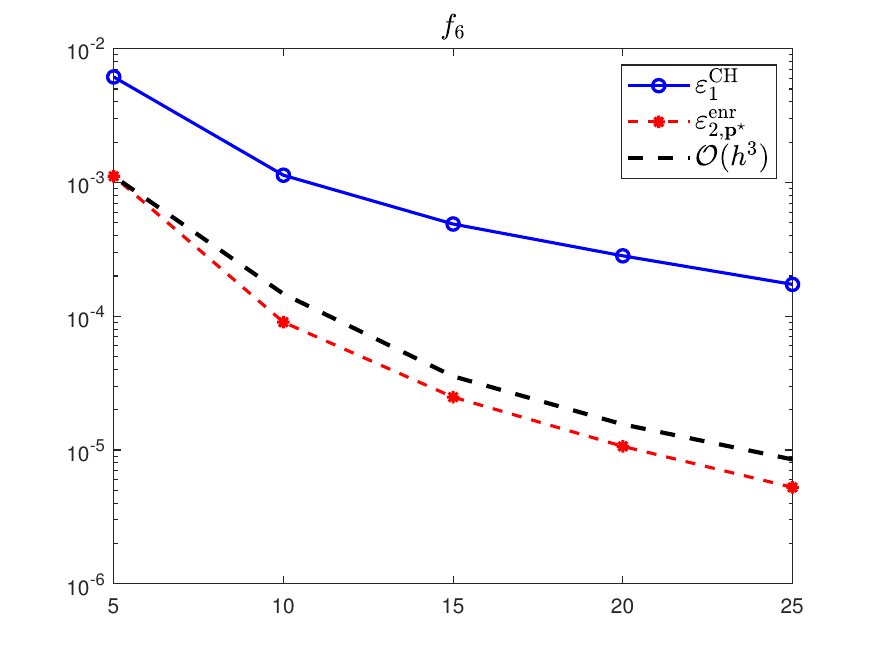}
\includegraphics[width=0.32\linewidth]{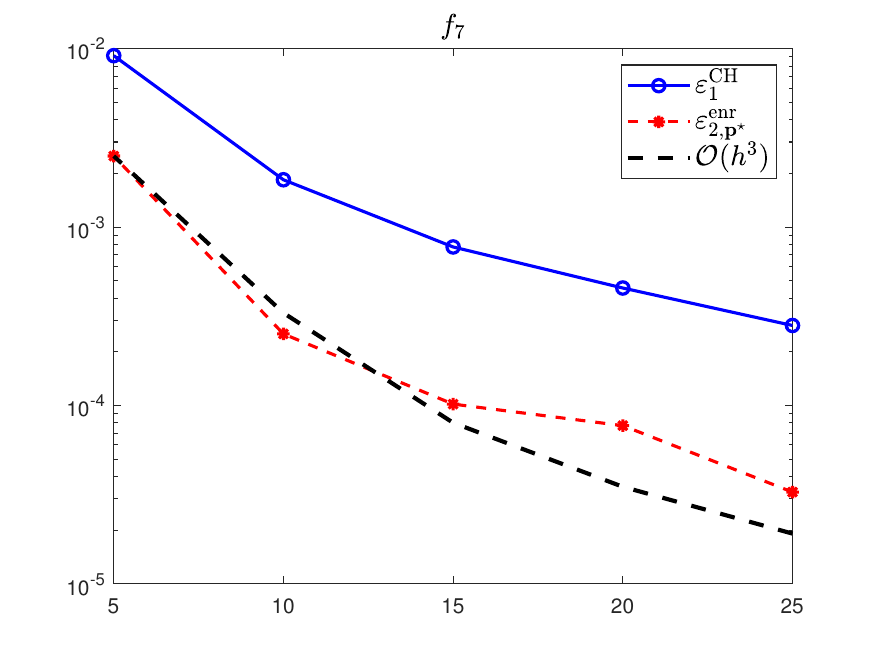}
\includegraphics[width=0.32\linewidth]{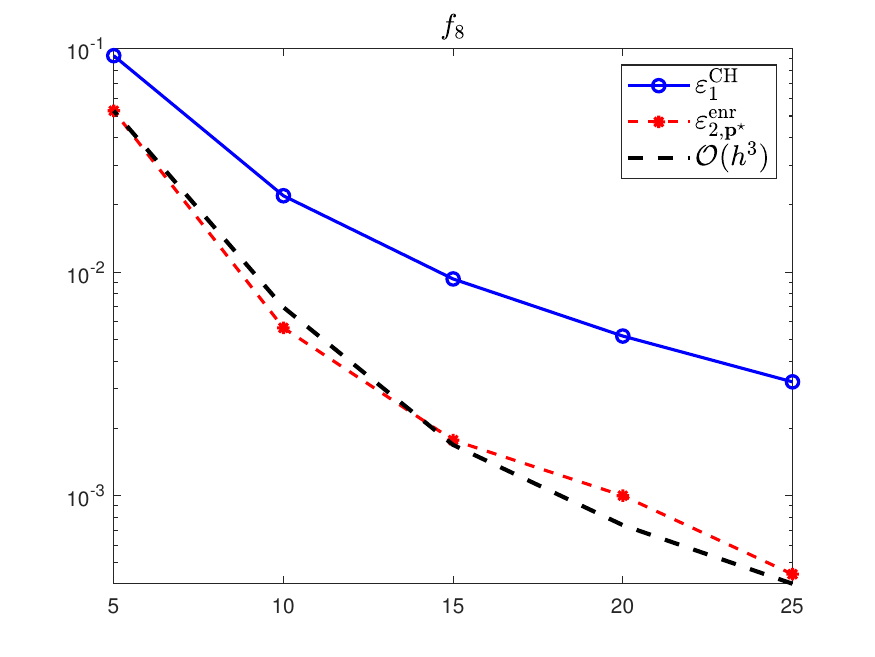}
\includegraphics[width=0.32\linewidth]{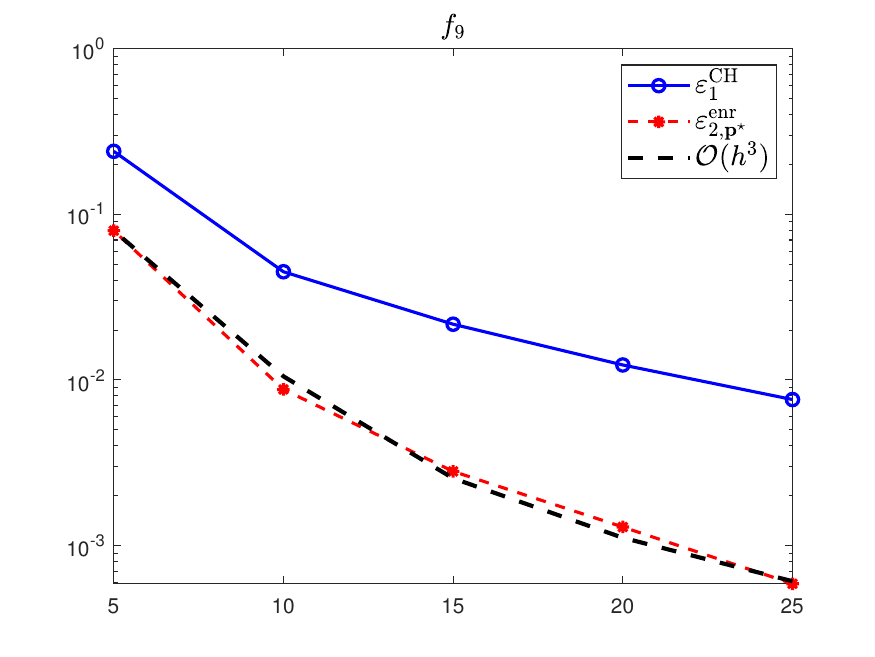}
    \caption{Comparison of the $L^2$ errors $\varepsilon_1^{\mathrm{CH}}$ (linear histopolation) and $\varepsilon_{2,\B{p}^{\star}}^{\mathrm{enr}}$ (proposed quadratic scheme) over quasi-uniform Delaunay meshes generated
    from perturbed Cartesian grids.}
    \label{plots1}
\end{figure}

\subsection{Computation of the inf--sup constant}\label{comp}
We conclude this study by analyzing the behaviour of the local inf--sup stability
parameter
\begin{equation*}
    \beta(\alpha)
    =
    \sqrt{
        \sigma_{\min}\left(
            \hat S(\alpha)
        \right)
    },
\end{equation*}
where, for a given Dirichlet parameter $\alpha$,
\begin{equation*}
    \hat S(\alpha)
    =
    G(\alpha)^{-1/2}
    S(\alpha)
    G(\alpha)^{-1/2}.
\end{equation*}
Throughout this experiment, we consider the symmetric choice
\begin{equation*}
\alpha_1=\alpha_2=\alpha_3=\alpha_4=\alpha.    
\end{equation*}
The stability constant is computed as the square root of the smallest eigenvalue
of $\hat S=\hat S(\alpha)$.
The results of this experiment are displayed in Fig.~\ref{beta}, where the values
of $\beta=\beta(\alpha)$ for $\alpha\in[2,5]$ are reported.

The numerical trend is clear: $\beta(\alpha)$ decreases monotonically as $\alpha$
increases.
This behaviour can be attributed to the progressive concentration of the
Dirichlet densities near the vertices of each face, which leads to smaller inf--sup constants.
At the same time, this dependence on $\alpha$ is fully consistent with the
affine-invariant stability established in Theorem~\ref{thm:affine-robust}, which
ensures that the local inf--sup constant remains independent of the element
geometry.

\begin{figure}[htbp]
    \centering
    \includegraphics[width=0.50\linewidth]{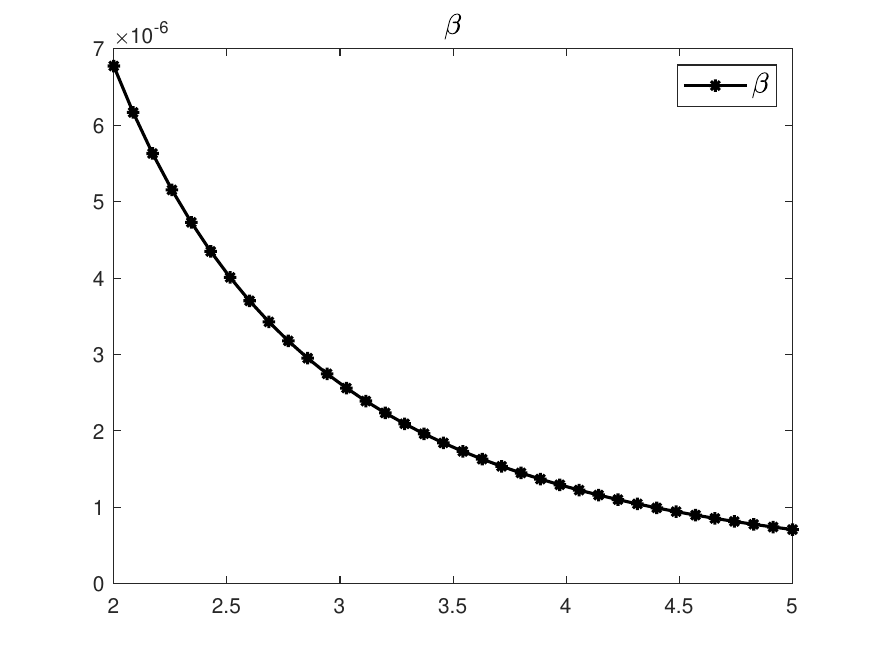}
    \caption{
        Behaviour of the inf--sup stability parameter
        $\beta(\alpha)$
        for symmetric Dirichlet densities on a reference tetrahedron.
        The constant decreases monotonically with~$\alpha$.
    }
    \label{beta}
\end{figure}

The monotonic decay of $\beta(\alpha)$ highlights the sensitivity of the Schur
complement to the choice of the density parameters.
In particular, large values of $\alpha$ reduce the stability of the local
quadratic reconstruction, whereas moderate choices of $\alpha$ yield
significantly stronger inf--sup bounds.
This experiment therefore provides a direct numerical validation of the
theoretical characterization of the inf--sup constant established in
Section~\ref{sec2}.

\begin{remark}
In strongly anisotropic regimes, the operator $S(\alpha)$ may exhibit a markedly
flattened spectrum, which in turn deteriorates the discrete inf--sup constant.
A simple spectral shift,
\[
S_{\mathrm{reg}}(\alpha) = S(\alpha) + \alpha_{\mathrm{reg}} I, \quad \alpha_{\mathrm{reg}} > 0,
\]
provides a mild form of stabilization: it preserves the structure of the moment
operator while preventing excessive compression of the smallest eigenvalues.
For any Dirichlet parameter $\alpha$, the associated regularized stability
constant is defined as
\[
\beta_{\mathrm{reg}}(\alpha)
= \sqrt{\sigma_{\min}\left(G^{-1/2} S_{\mathrm{reg}}(\alpha) G^{-1/2}\right)} .
\]
As illustrated in Fig.~\ref{beta}, even a very small shift $\alpha_{\mathrm{reg}}$
can produce a clearly visible improvement.
This suggests that the modification $S(\alpha)\mapsto S_{\mathrm{reg}}(\alpha)$
offers a practical safeguard in highly anisotropic configurations, without affecting the main structural properties of the moment formulation.
\end{remark}

\section*{Conclusions and Future Work}

In this work, we have developed a unified and fully constructive framework for
weighted quadratic histopolation on simplicial meshes.
The approach relies on a canonical face/interior decomposition of
$\mathbb{P}_2\left(S_d\right)$ combined with a structured Schur-type analysis of
the associated moment matrices.
This reduction yields a dimension-independent characterization of unisolvence and
leads to a sharp inf--sup condition.
A central outcome of the analysis is that the stability of enriched quadratic
histopolation is determined by a low-dimensional reduced operator
$\hat S(\Omega)$, whose smallest eigenvalue provides the inf--sup constant.
This perspective offers a transparent spectral interpretation of enrichment,
conditioning, and robustness, and naturally motivates systematic strategies for
spectral optimization.
Moreover, we show that this reduced stability mechanism is robust under affine
transformations of the elements, ensuring that the inf--sup constant remains
uniformly bounded on shape-regular meshes, independently of local geometric
distortions.

Building on these theoretical results, we have constructed explicit families of
face and interior quadratic polynomials that yield well-conditioned moment
matrices.
We have also introduced a parametric scaling of densities and moment functionals
that preserves the analytical structure of the formulation while enabling
automatic spectral tuning.
Three-dimensional numerical experiments on unstructured meshes confirm the
expected $O\left(h^{3}\right)$ convergence rate and are fully consistent with the theoretical
inf--sup estimates.
They further demonstrate the practical effectiveness of the proposed spectral
optimization strategy in reducing the condition number of the resulting discrete
systems.

Overall, the framework developed here provides a mathematically transparent and computationally effective foundation for higher-order histopolation on general simplicial meshes; the present contribution addresses enriched quadratic constructions, while the extension to higher degrees is deferred. Looking ahead, we plan to: 
\begin{itemize}
    \item[$(i)$] identify optimal or minimax orthogonal decompositions;
    \item[$(ii)$] generalize the spectral optimization strategy to higher-degree polynomial spaces;
    \item[$(iii)$] investigate applications to inverse problems and reconstruction tasks with broader classes of weighted moment data.
\end{itemize}

\section*{Acknowledgments}
This research has been achieved as part of RITA ``Research
 ITalian network on Approximation'' and as part of the UMI group ``Teoria dell'Approssimazione
 e Applicazioni''. The research was supported by GNCS-INdAM 2025 project ``Polinomi, Splines e Funzioni Kernel: dall'Approssimazione Numerica al Software Open-Source''. The work of F. Nudo has been funded by the European Union NextGenerationEU under the Italian National Recovery and Resilience Plan (PNRR), Mission 4, Component 2, Investment 1.2 ``Finanziamento di progetti presentati da giovani ricercatori'',\ pursuant to MUR Decree No. 47/2025. The research was supported by the grant  Bando Professori visitatori 2022 which has allowed the visit of Prof. Allal Guessab to the Department of Mathematics and Computer Science of the University of Calabria in the spring 2022.

\bibliographystyle{spmpsci}
\bibliography{bibliography}

\end{document}